%% file: acyclicplanes-final.tex
\documentclass[11pt,oneside,english]{amsart}
\usepackage[T1]{fontenc}
\usepackage[utf8]{inputenc}
\usepackage{geometry}
\usepackage{amsthm}
\usepackage{amstext}
\usepackage{amssymb}
\usepackage{xypic}
\makeatletter
\numberwithin{equation}{section}
\numberwithin{figure}{section}
\theoremstyle{plain}
\newtheorem{thm}{\protect\theoremname}[section]
  \theoremstyle{definition}
  \newtheorem{defn}[thm]{\protect\definitionname}
  \theoremstyle{definition}
  \newtheorem{example}[thm]{\protect\examplename}
  \newtheorem{notation}[thm]{\protect\notationname}
  \newtheorem{convention}[thm]{\protect\conventionname}
  \theoremstyle{plain}
  \newtheorem{prop}[thm]{\protect\propositionname}
  \theoremstyle{plain}
  \newtheorem{lem}[thm]{\protect\lemmaname}
  \theoremstyle{remark}
  \newtheorem{rem}[thm]{\protect\remarkname}
  \theoremstyle{plain}
  \newtheorem{cor}[thm]{\protect\corollaryname}
  \theoremstyle{plain}
  \newtheorem{question}[thm]{\protect\questionname}

\usepackage{amsfonts}
\usepackage{graphicx}
\usepackage{color}
\theoremstyle{remark}
\newtheorem{parn}{}[subsubsection]

\makeatother

\usepackage{babel}
  \providecommand{\corollaryname}{Corollary}
  \providecommand{\definitionname}{Definition}
  \providecommand{\notationname}{Notation} 
  \providecommand{\conventionname}{Convention} 
  \providecommand{\examplename}{Example}
  \providecommand{\lemmaname}{Lemma}
  \providecommand{\propositionname}{Proposition}
  \providecommand{\questionname}{Question}
  \providecommand{\remarkname}{Remark}
\providecommand{\theoremname}{Theorem}

\begin{document}

\author{Adrien Dubouloz}

\address{IMB UMR5584, CNRS, Univ. Bourgogne Franche-Comté, F-21000 Dijon,
France.}

\email{adrien.dubouloz@u-bourgogne.fr}

\author{Frédéric Mangolte}

\address{LAREMA UMR6093\\
CNRS, Univ. Angers, Univ. Bretagne-Loire\\  
F-49000 Angers, France.}

\email{frederic.mangolte@univ-angers.fr}

\thanks{This project was partialy funded by ANR Grant \textquotedbl{}BirPol\textquotedbl{}
ANR-11-JS01-004-01. }

\title[Fake Real Planes]{Fake Real Planes: exotic affine algebraic models of $\mathbb{R}^{2}$}
\begin{abstract}
We study real rational models of the euclidean plane $\mathbb{R}^{2}$ up to isomorphisms and up to birational diffeomorphisms. The analogous study in the compact case, that is the classification of real rational models of the real projective plane $\mathbb{R}\mathbb{P}^{2}$ is well known: up to birational diffeomorphisms, there is only one model. A fake real plane is a nonsingular affine surface 
defined over the reals with homologically trivial complex locus and real locus diffeomorphic to $\mathbb{R}^2$ but which is not isomorphic to the real affine plane. We prove that fake planes exist by giving many examples and we tackle the question: does there exist fake planes 
whose real locus is not birationally diffeomorphic to the real affine plane? 
\end{abstract}
\maketitle

\section*{Introduction}

An algebraic complexification of a real smooth $\mathcal{C}^{\infty}$-manifold
$M$ is a smooth complex quasi-projective algebraic variety $V$ endowed
with an anti-regular involution $\sigma$ such that $M$ is diffeomorphic
to the \emph{real locus} $V^{\sigma}$ of $V$. Equivalently $V$ is the complex
model $X\times_{\mathrm{Spec}(\mathbb{R})}\mathrm{Spec}(\mathbb{C})$
of a smooth quasi-projective algebraic variety $X$ defined over $\mathbb{R}$,
whose set of $\mathbb{R}$-rational points is diffeomorphic to $M$
when equipped with the euclidean topology. Some manifolds such as
real projective spaces $\mathbb{RP}^{n}$ and real euclidean affine
spaces $\mathbb{R}^{n}$ have natural algebraic complexifications,
given by the complex projective and affine spaces $\mathbb{P}_{\mathbb{C}}^{n}$
and $\mathbb{A}_{\mathbb{C}}^{n}$ respectively. But these also admit
infinitely many other complexifications, and it is a natural problem
to try to classify them up to appropriate notions of equivalence. 

For example, the real projective plane $M=\mathbb{R}\mathbb{P}^{2}$
is also the real locus of the quotient $V_{1}$ of the smooth quartic hypersurface $x^{4}+y^{4}+z^{4}-w^{4}=0$
in $\mathbb{P}_{\mathbb{C}}^{3}$ by the fixed point-free involution $[x:y:z:w]\mapsto [-y:x:-w:z]$.
The complex surface $V_{1}$ endowed with the induced real structure is not isomorphic to $V_{0}=\mathbb{P}_{\mathbb{C}}^{2}$
endowed with the usual complex conjugation, and actually, since $V_{1}$ is an Enriques surface (see e.g. \cite[Chapter VIII]{BPV}), $V_{1}$
and $V_{0}$ are not even birational to each other. On the other hand,
we obtain infinite families of pairwise non-isomorphic rational projective
complexifications $V$ of $\mathbb{R}\mathbb{P}^{2}$ by blowing-up
sequences of pairs of non-real complex conjugate points of $\mathbb{P}_{\mathbb{C}}^{2}$.

Thus whatever is the equivalence relation: biregular, birational or deformation, there are an infinite number of possibilities. It is natural to pay attention to complexifications with ``minimal topology'', say in the sense of Betti numbers.
Recall that a \emph{fake projective plane}, as defined by Mumford \cite{Mu79}, is a nonsingular complex projective surface $S$, whose Betti numbers are those of $\mathbb{CP}^{2}$, and which is not biregularly isomorphic to $\mathbb{P}_{\mathbb{C}}^{2}$.
One could define a \emph{fake real projective plane} to be a complex fake projective plane with anti-regular involution, whose real locus is diffeomorphic to $\mathbb{RP}^{2}$ and which is not biregularly isomorphic to $\mathbb{P}_{\mathbb{R}}^{2}$, but despite of the existence of $100$ fake projective planes up to biregular isomorphism \cite{PY07,PY10,CS10}, none of them admits a real structure as proved by Kulikov and Kharlamov \cite[Thm. 5.1]{KK02}. Thus there is no fake real projective plane at all.

In addition to $\mathbb{R}\mathbb{P}^{2}$, the connected compact
surfaces $M$ diffeomorphic to the real locus of smooth rational projective
surface minimal over $\mathbb{R}$ are the sphere $S^{2}$, the torus
$S^{1}\times S^{1}$ and the Klein bottle $K=\mathbb{R}\mathbb{P}^{2}\#\mathbb{R}\mathbb{P}^{2}$.
Their respective minimal complexifications are $\mathbb{P}_{\mathbb{C}}^{2}$,
the quadric hypersurface $x^{2}+y^{2}+z^{2}-w^{2}=0$ in $\mathbb{P}_{\mathbb{C}}^{3}$,
the Hirzebruch surfaces of even index $\mathbb{F}_{2k}$, $k\geq0$,
and the Hirzebruch surfaces of odd index $\mathbb{F}_{2k+1}$, $k\geq1$. The minimality of the
complexification $V$ endowed with the real structure is equivalent
to the minimality of its topology as a compact complex manifold among
all complexifications of $M$. The above description shows that, for each of the surfaces $\mathbb{R}\mathbb{P}^{2}$, $S^{2}$, $S^{1}\times S^{1}$ and $K$, complexifications 
with minimal topology are either unique, or all diffeomorphic to each
others, belonging to a unique equivalence class of deformation. 

Coming back to surfaces obtained by blowing-up
sequences of pairs of non-real conjugated points of $\mathbb{P}_{\mathbb{C}}^{2}$, 
it follows from their construction that they are all $\mathbb{R}$-biregularly
birationally equivalent: this means that for every pair $V$ and $V'$
of such complexifications, there exists birational maps $\varphi:V\dashrightarrow V'$
and $\varphi'\colon V'\dashrightarrow V$ inverse to each other and whose
restrictions to the real loci of $V$ and $V'$ are diffeomorphisms
inverse to each other. A striking result of Biswas and Huisman \cite{BiH07} (see
also \cite{HM09}) asserts that rational algebraic complexifications of
a given smooth connected compact surface $M$ are all $\mathbb{R}$-biregularly
birationally equivalent. 
This classification result gave rise to many
further discoveries, see the survey article \cite{Ma15} and the bibliography
given there.

In this paper, we lay the importance on affine complexifications and
we discover that, contrary to the projective case, the easiest example
$M=\mathbb{R}^{2}$ possesses a lot of affine complexifications $S$
with ``minimal topology'' which are not biregularly isomorphic to
$\mathbb{A}^{2}$. We call them \emph{fake real planes}. In contrast
with the projective case, where the notion of rational complexification
with minimal topology is unambiguous due to the fact that the topology
of a smooth rational complexification $V$ of a given compact surface
$M$ is fully determined by its Picard rank $\rho(V)=\mathrm{rk}(N_{1}(V))$,
there are many possibilities to define a notion of minimality
of the topology of an affine complexification, even in the rational
case. For instance, there exists many rational affine complexifications
of $\mathbb{R}^{2}$ with vanishing second homology group $H_{2}(S;\mathbb{Z})$
but nontrivial fundamental group. Inspired by the work of Totaro \cite{To03},
who defined a \emph{good affine complexification} of $M$ to be a
complexification $S$ for which the inclusion of $M\hookrightarrow S$
of $M$ as the real locus of $S$ is a homotopy equivalence, a natural
notion of minimality of affine complexifications $S$ of $M=\mathbb{R}^{2}$
is to require that $S$ is contractible as a smooth complex manifold.
Here we mainly consider weaker variants in which we require only that
the inclusion $M\hookrightarrow S$ induces an isomorphism between
the respective homology groups of $M$ and $S$, taken with integral
or rational coefficients. So for $M=\mathbb{R}^{2}$, the corresponding
smooth affine complexifications $S$ are respectively $\mathbb{Z}$-acyclic
and $\mathbb{Q}$-acyclic complex surfaces.

Our first goal is to show that fake real planes do exist by
exhibiting many examples and to give elements of classification of
these objects up to biregular isomorphism depending on natural algebro-geometric
invariants such as their logarithmic Kodaira dimension. Any contractible affine complexification $S$ of $\mathbb{R}^{2}$ of non positive logarithmic Kodaira dimension is isomorphic to $\mathbb{A}^{2}$ (\cite[Theorem 4.7.1 (1), p. 244]{MiyBook} and \cite{MS80}), in particular, there is no good affine complexification of $\mathbb{R}^{2}$
of logarithmic Kodaira dimension $0$. 
In contrast, we give several families of rational good affine complexifications of $\mathbb{R}^{2}$
of logarithmic Kodaira dimension $1$ (see Examples \ref{Ex:homotopy-Kod1}) and $2$ respectively, which
are therefore not biregularly isomorphic to $\mathbb{A}^{2}$. 
Another striking family of examples is provided by some Ramanujam surfaces \cite{Ram71}, a famous class of smooth complex contractible affine surfaces of logarithmic Kodaira dimension $2$, which admit real structures with real locus diffeomorphic to $\mathbb{R}^2$ (see \ref{Ex:Ramanujam-Surf}).  
 
As a step towards a classification of fake planes, we establish a real counter-parts of a series of results due to thom Dieck and Petrie \cite{tDPe93} describing the structure of $\mathbb{Z}$-acyclic and
$\mathbb{Q}$-acyclic smooth complex affine surfaces in terms of blow-ups
of arrangements of rational curves in $\mathbb{P}_{\mathbb{C}}^{2}$.
As an application, we obtain a complete classification of $\mathbb{Z}$-acyclic
smooth affine complexifications of $\mathbb{R}^{2}$ of logarithmic
Kodaira dimension $1$ (Theorem \ref{thm:Fake-Kodaira-1}) and a precise
description of the structure of $\mathbb{Q}$-acyclic smooth affine
complexifications of $\mathbb{R}^{2}$ of logarithmic Kodaira dimension
$2$ (Theorem \ref{thm:Kod2-arrangement}) formulated in terms of
arrangements of $\mathbb{R}$-rational curves in $\mathbb{P}_{\mathbb{R}}^{2}$. 

In a second step, we tackle the classification of fake real planes
with $\mathbb{Q}$-acyclic smooth affine complexifications up to $\mathbb{R}$-biregular
birational equivalence: we prove that a large class of such surfaces
are $\mathbb{R}$-biregularly birationally equivalent to $\mathbb{A}^{2}$.
More precisely, we establish that every smooth $\mathbb{Q}$-acyclic
complexification $S$ of $\mathbb{R}^{2}$ with negative logarithmic
Kodaira dimension admits a surjective morphism $\pi\colon S\rightarrow\mathbb{A}_{\mathbb{R}}^{1}$
defined over $\mathbb{R}$, with general fiber isomorphic to the affine
line (Theorem \ref{thm:Q-acyclic-Neg-Desc}), and we show that every
such fibered surface $\pi:S\rightarrow\mathbb{A}^{1}$ with at most
one singular fiber is $\mathbb{R}$-biregularly birationally equivalent
to $\mathbb{A}^{2}$ (Theorem \ref{thm:Bir-rectif}). 

We saw above that for $\mathbb{R}\mathbb{P}^{2}$, $S^{2}$, $S^{1}\times S^{1}$ and $K=\mathbb{R}\mathbb{P}^{2}\#\mathbb{R}\mathbb{P}^{2}$, there exists either a unique minimal rational complexification or at most one family of pairwise non isomorphic but $\mathbb{R}$-biregularly birationally equivalent minimal rational complexifications. We conclude the paper with the construction of an infinite number of fake planes with moduli of arbitrary positive dimension of pairwise non isomorphic, deformation equivalent, $\mathbb{Q}$-acyclic euclidean planes all $\mathbb{R}$-biregularly birationally equivalent to $\mathbb{A}_{\mathbb{R}}^{2}$ (see $\S$ \ref{sub:Moduli-of--biregularly}). 

In contrast with the projective case, the main difficulty to understand
the notion of $\mathbb{R}$-biregular birational equivalence comes
from the lack of natural numerical invariants to distinguish classes.
In particular, neither the topology of the complexification of a given fake real plane $S$ nor
its logarithmic Kodaira dimension are invariants of its $\mathbb{R}$-biregular
birational class. And even though Theorem \ref{thm:Bir-rectif} is
a significant step towards a complete classification of fake real
planes up to $\mathbb{R}$-biregular birational equivalence, the fact
that its proof depends on the construction of explicit elementary
$\mathbb{R}$-biregular birational links between appropriate projective
models of the affine complexification of $S$ does not give any clear
insight on possible numerical invariants of $\mathbb{R}$-biregular
birational equivalence classes. As a consequence, the question of
existence of fake real planes not $\mathbb{R}$-biregularly birationally equivalent to $\mathbb{A}^{2}$ remains open, a good candidate being the real Ramanujam surfaces mentioned above (see also $\S$ \ref{sub:Rectif-kod0}
for another candidate of logarithmic Kodaira dimension $0$). 

\tableofcontents{}

\section{Preliminaries}

In this article, the term $\mathbf{k}$-variety will always refer
to a geometrically integral scheme $X$ of finite type over a base
field $\mathbf{k}$ of characteristic zero. A morphism of $\mathbf{k}$-varieties
is a morphism of $\mathbf{k}$-schemes. In the sequel, $\mathbf{k}$
will be most of the time equal to either $\mathbb{R}$ or $\mathbb{C}$,
and we will say that $X$ is a real, respectively complex, algebraic
variety. 

A complex algebraic variety $X$ will be said to be defined over $\mathbb{R}$
if there exists a real algebraic variety $X_{0}$ and an isomorphism
of complex algebraic varieties between $X$ and the \emph{complexification}
$X{}_{0,\mathbb{C}}=X_{0}\times_{\mathrm{Spec}(\mathbb{R})}\mathrm{\mathrm{Spec}(\mathbb{C})}$
of $X_{0}$, where $\mathrm{Spec}(\mathbb{C})\rightarrow\mathrm{Spec}(\mathbb{R})$
is the morphism induced by the usual inclusion $\mathbb{R}\hookrightarrow\mathbb{C}=\mathbb{R}[x]/(x^{2}+1)$.
We will often use the shorter notation $X{}_{0,\mathbb{C}}=X_{0}\otimes_{\mathbb{R}}\mathbb{C}$.
A complex variety of the form $X_{0,\mathbb{C}}$ is naturally endowed
with an additional action of the Galois group $\mathrm{Gal}(\mathbb{C}/\mathbb{R})=\mathbb{Z}_{2}$
given by the anti-regular involution $\mathrm{\sigma}=\mathrm{id}_{X_{0}}\times(x\mapsto-x)$,
which we call the \emph{real structure }on\emph{ $X_{0,\mathbb{C}}$. }

\subsection{Points and curves on surfaces }
\begin{notation}
\label{def:Rational_Points} Given a real algebraic surface $S$,
we denote by $S(\mathbb{R})$ and $S(\mathbb{C})$ the sets of $\mathbb{R}$-rational
and $\mathbb{C}$-rational points respectively. We always consider
$S(\mathbb{R})$ as a subset of $S(\mathbb{C})$ via the map induced
by the inclusion $\mathbb{R}\hookrightarrow\mathbb{C}$. In what follows,
the elements of $S(\mathbb{R})$, $S(\mathbb{C})$ and $S(\mathbb{C})\setminus S(\mathbb{R})$
will be frequently referred to as \emph{real points}, \emph{complex
points} and \emph{non-real points} of $S$ respectively.

The subset $(S_{\mathbb{C}}(\mathbb{C}))^{\sigma}$ of $S_{\mathbb{C}}(\mathbb{C})$
consisting of $\mathbb{C}$-rational points of $S_{\mathbb{C}}$ that
are fixed by the real structure $\sigma$ is called the \emph{real
locus} of $S_{\mathbb{C}}$, and we identify it with $S(\mathbb{R})$
in the natural way. 
\end{notation}
\noindent

By a \emph{curve} on a surface $S$ defined over $\mathbf{k}$, we mean a geometrically reduced closed sub-scheme $C\subset S$
of pure codimension $1$ defined over $\mathbf{k}$. We denote by $\mathbb{Z}\langle C\rangle$ the free abelian group generated by
the irreducible components of $C$.

\begin{defn}
\label{def:curve_SNC_div}   

1) A \emph{Smooth Normal Crossing $($SNC$)$ divisor} $B$ on a smooth surface $S$ defined over
$\mathbf{k}$ is a curve $B$ on $S$ whose base extension $B_{\overline{\mathbf{k}}}$ to the algebraic closure $\overline{\mathbf{k}}$ of $\mathbf{k}$ has smooth irreducible components
and ordinary double points only as singularities. Equivalently, for
every closed point $p\in B_{\overline{\mathbf{k}}}\subset S_{\overline{\mathbf{k}}}$, the local equations of the irreducible components of $B_{\overline{\mathbf{k}}}$ passing through $p$ form a part of a regular sequence
in the maximal ideal $\mathfrak{m}_{S_{\overline{\mathbf{k}}},p}$ of the local ring $\mathcal{O}_{S_{\overline{\mathbf{k}}},p}$
of $S_{\overline{\mathbf{k}}}$ at $p$. 

We say that $B$ is a \emph{strictly SNC divisor} if every two of the irreducible components of $B_{\overline{\mathbf{k}}}$ intersect in at most one point.
  
2) The \emph{dual graph} $\Gamma B=(\Gamma_{v}B,\Gamma_{e}B)$
of an SNC divisor $B$ on a smooth surface $S$ defined over an algebraically closed field is the graph with vertex set $\Gamma_{v}B$
the set of irreducible components of $B$ and with edges set $\Gamma_{e}B$ the set of double points of $B$. An edge in $\Gamma B$ connects
the two vertices which intersect in it. Note that $B$ is a strictly SNC divisor if and only if any two vertices of its dual graph are connected by at most one edge. 
\end{defn}

\begin{defn}
\label{def:rational_tree} Let $V$ be a smooth surface defined over
$\mathbf{k}$, and let $\overline{\mathbf{k}}$ be an algebraic closure
of $\mathbf{k}$.

1) A \emph{geometrically rational tree} on $V$ is an SNC divisor
$B$ defined over $\mathbf{k}$ such that every irreducible component
of $B_{\overline{\mathbf{k}}}\subset S_{\overline{\mathbf{k}}}$ is
a $\overline{\mathbf{k}}$-rational complete curve and the dual graph
$\Gamma(B_{\overline{\mathbf{k}}})$ is a tree. 

2) A \emph{geometrically rational chain} on $V$ is a geometrically
rational tree $B$ defined over $\mathbf{k}$ such that $\Gamma(B_{\overline{\mathbf{k}}})$ is a chain. 
The irreducible components $B_{0},\ldots,B_{r}$ of $B_{\overline{\mathbf{k}}}$ can be ordered
in such a way that $B_{i}\cdot B_{j}=1$ if $|i-j|=1$ and $0$ otherwise.
A geometrically rational chain $B$ with such an ordering on the set
of irreducible components of $B_{\overline{\mathbf{k}}}$ is said
to be \emph{oriented}. The components $B_{0}$ and $B_{r}$ are called
respectively the left and right boundaries of $B$, and we say by
extension that an irreducible component $B_{i}$ of $B_{\overline{\mathbf{k}}}$
is on the left of another one $B_{j}$ when $i<j$. The sequence of
self-intersections $[B_{0}^{2},\ldots,B_{r}^{2}]$ is called the \emph{type}
of the oriented geometrically rational chain $B$. 

An \emph{oriented} $\mathbf{k}$-\emph{subchain} of $B$ is a geometrically
rational chain $Z$ whose support is contained in that of $B$. We say
that an oriented geometrically rational chain $B$ is composed of
$\mathbf{k}$-subchains $Z_{1},\ldots,Z_{s}$ and we write $B=Z_{1}\vartriangleright\cdots\vartriangleright Z_{s}$
if the $Z_{i}$ are oriented $\mathbf{k}$-subchains of $B$ whose
union is $B$ and the irreducible components of $Z_{i,\overline{\mathbf{k}}}$
precede those of $Z_{j,\overline{\mathbf{k}}}$ for $i<j$. 
\end{defn}

\subsection{Birational maps and log-resolutions}

\indent\newline\indent Recall that the domain of definition of a rational
map $\varphi:X\dashrightarrow Y$ between two $\mathbf{k}$-schemes $X$ and $Y$ is the largest open subset $\mathrm{dom}_{\varphi}$ on which $\varphi$ is represented by a morphism. We say that $\varphi$ is
regular at a closed point $x$ if $x\in\mathrm{dom}_{\varphi}$. A
rational map $\varphi:X\dashrightarrow Y$ is called \emph{birational}
if it admits a rational inverse $\psi:Y\dashrightarrow X$.

\noindent In the sequel, we will frequently make use of the following
type of birational morphisms:
\begin{example}
\label{ex:subdiv_exp} (Subdivisional expansion of a surface at a
point \cite[$\S$ 2.4]{tDPe93}). Let $S$ be a smooth surface defined over $\mathbf{k}$ and
let $p\in S$ be a closed point with residue field $\kappa(p)$. A
\emph{subdivisionial expansion with center at} $p$ is a birational
morphism $\tau:S'\rightarrow S$ restricting to an isomorphism over
$S\setminus\{p\}$ and such that $\tau^{-1}(p)$ is a chain of smooth
$\kappa(p)$-rational curves, containing a unique irreducible component
$A_{0}(p)\simeq\mathbb{P}_{\kappa(p)}^{1}$ with normal bundle of
degree $-\deg\mathbb{\kappa}(p)/\mathbf{k}$. Given an ordered sequence
of regular parameters $(x_{-},x_{+})\in\mathfrak{m}_{S,p}$ in the
local ring $\mathcal{O}_{S,p}$ of $S$ at $p$, there exists a unique
pair of coprime integers $1\leq\mu_{-}\leq\mu_{+}$ such that $\tau:S'\rightarrow S$
coincides with the minimal resolution of the indeterminacies at $p$
of the rational map $x_{+}^{\mu_{+}}/x_{-}^{\mu_{-}}:S\dashrightarrow\mathbb{P}^{1}$ (see \cite[Theorem 2.6 (d)]{Za99}). For instance, the particular case $\mu_{\pm}=1$ is nothing but the
blow-up $\tau:S'=\mathrm{Proj}_{S}(\bigoplus_{n\geq0}\mathcal{I}_{p}^{n})\rightarrow S$
of $S$ at $p$, where $\mathcal{I}_{p}\subset\mathcal{O}_{S}$ denotes
the ideal sheaf of $p$. 

\begin{figure}[ht]
\centering 
\input{subdiv-f.tex} 
\caption{A subdivisional expansion at a point $p$. The non horizontal components represent the exceptional divisor of $\tau$.}  
\label{fig:subdiv}
\end{figure} 

In the sequel, we will mostly use such birational morphisms in the
particular case where $x_{-}$ and $x_{+}$ are the respective local
equations of integral curves $C_{-}$ and $C_{+}$ on $S$ intersecting
transversally at $p$. The integer $\mu_{\pm}$ is then equal to the
coefficient of $A_{0}(p)$ in the total transform of $C_{\pm}$ and
will say that $\tau:S'\rightarrow S$ is the subdivisional expansion
of $S$ at the $\kappa(p)$-rational point $(C_{-}\cap C_{+})_{p}$,
with multiplicities $(\mu_{-},\mu_{+})$.\end{example}

\begin{convention} In the rest of the article, we will often use the same notation for a divisor on a surface and its proper transform on another surface by a birational map. 
\end{convention}

\begin{defn}
A $\mathbf{k}$-variety $X$ is called $\mathbf{k}$\emph{-rational}
if there exist a birational map $\varphi:\mathbb{P}_{\mathbf{k}}^{\dim X}\dashrightarrow X$.
A geometrically reduced $\mathbf{k}$-scheme of finite type $X$ is
called \emph{geometrically rational} if every irreducible component
of $X_{\overline{\mathbf{k}}}$ is $\overline{\mathbf{k}}$-rational,
where $\overline{\mathbf{k}}$ denotes an algebraic closure of $\mathbf{k}$. 
\end{defn}
\noindent

\begin{defn}
An SNC (resp. strictly SNC) divisor $B$ on a smooth complete surface $V$ defined over
$\mathbf{k}$ is said to be \emph{SNC-minimal} (resp. \emph{strictly SNC-minimal}) over $\mathbf{k}$
if there does not exist any projective strictly birational morphism
$\tau:V\rightarrow V'$ onto a smooth surface defined over $\mathbf{k}$
with exceptional locus contained in $B$ such that $\tau_{*}(B)$
is SNC (resp. stricly SNC). 
\end{defn}

\begin{example}
Let $V$ be a smooth projective surface defined over $\mathbb{R}$
and let $B_{0}$ and $\overline{B}_{0}$ be a pair of smooth $\mathbb{C}$-rational
curves in $V_{\mathbb{C}}$ exchanged by the real structure $\sigma$
and intersecting transversally in a single point. Then $B_{0}\cup\overline{B}_{0}$
is the complexification of a geometrical rational chain $B$ on $V$
which is SNC minimal over $\mathbb{R}$ even if $B_{0}^{2}=\overline{B}_{0}^{2}=-1$. 
\end{example}

\begin{defn} \label{defn:log-res}
Let $(S,C)$ be a pair consisting of a smooth surface $S$ and a curve $C\subset S$ defined over $\mathbf{k}$. A \emph{log-resolution} of $(S,C)$ is a projective birational morphism $\tau:S'\rightarrow S$ defined over $\mathbf{k}$ such that $S'$ is smooth and the union $C'$ of the reduced total transform $\tau^{-1}(C)$ of $C$ with the exceptional locus $\mathrm{Ex}(\tau)$ of $\tau$ is an SNC divisor on $S'$. We say that $\tau:(S',C')\rightarrow (S,C)$ is a \emph{strict} log-resolution if $C'$ is strictly SNC.   
\end{defn}

\subsection{Smooth projective completions and logarithmic Kodaira dimension}

\indent\newline\noindent By virtue of Nagata compactification Theorem \cite{Na62}
and classical desingularization theorems (see e.g. \cite{Wa35}), every smooth
surface $S$ defined over $\mathbf{k}$ admits an open immersion $S\hookrightarrow V$
into a smooth projective surface with SNC boundary divisor $B=V\setminus S$,
both defined over $\mathbf{k}$. In what follows the term \emph{smooth projective completion} of a surface $S$ will be used to refer to any pair $(V,B)$ consisting of a smooth projective surface $V$ and a reduced SNC divisor $B$ on it such that $V\setminus B$ is isomorphic to $S$. A smooth projective completion $(V,B)$ of $S$ will be called SNC-minimal if $B$ is an SNC-minimal divisor on $V$.

The \emph{(logarithmic) Kodaira dimension}
$\kappa(S)$ of $S$ is then defined as the Iitaka dimension \cite{Ii70}
of the pair $(V;\omega_{V}(\log B))$ where $\omega_{V}(\log B)=(\det\Omega_{V/\mathbf{k}}^{1})\otimes\mathcal{O}_{V}(B)$.
The so-defined element $\kappa(S)\in\{-\infty,0,1,2\}$ is independent
of the choice of a smooth complete model $(V,B)$ \cite{Ii77}, and
it coincides with the usual notion of Kodaira dimension in the case
where $S$ is already complete.
Furthermore, it is invariant under arbitrary extensions of the base
field $\mathbf{k}$, as a consequence of the flat base change theorem
\cite[Proposition III.9.3]{Ha77}. 
A surface of Kodaira dimension $2$ is usually said to be of \emph{general
type}.

\subsection{Euclidean topology}

\indent\newline\noindent Recall that when $\mathbf{k}=\mathbb{R}$
or $\mathbb{C}$, the set of $\mathbf{k}$-rational point of a $\mathbf{k}$-variety
$X$ can be endowed with the euclidean topology. Namely, every $\mathbf{k}$-rational
point $p$ admits an affine open neighborhood $U_{p}$ and the choice
of a closed immersion $j:U_{p}\hookrightarrow\mathbb{A}_{\mathbf{k}}^{N}$
enables to equip $j(U_{p}(\mathbf{k}))$ with the subspace topology
induced by the usual euclidean topology on $\mathbb{A}_{\mathbf{k}}^{N}(\mathbf{k})\simeq\mathbf{k}^{N}$.
The so-constructed topology on $X(\mathbf{k})$ is well-defined and
independent of the choices made \cite[Lemme 1 and Proposition 2]{Se55}.
When $X$ is smooth, $X(\mathbf{k})$ is a  $\mathcal{C}^{\infty}$-manifold when equipped with the structure locally inherited by the standard $\mathcal{C}^{\infty}$-structure on $\mathbf{k}^{N}$.
\begin{convention}
Given a real algebraic variety $X$, we always consider the sets $X(\mathbb{R})$
and $X_{\mathbb{C}}(\mathbb{C})$ as equipped with their respective
euclidean topologies. The real structure $\sigma$ on $X_{\mathbb{C}}$
is in particular a continuous involution of $X_{\mathbb{C}}(\mathbb{C})$,
and we consider $X(\mathbb{R})$ as a subspace of $X_{\mathbb{C}}(\mathbb{C})$
via its identification with the set $X_{\mathbb{C}}^{\sigma}(\mathbb{C})$
of fixed points of $\sigma$. 
\end{convention}
Recall that given a coefficient ring $A$, a topological manifold
$M$ is called $A$-\emph{acyclic} if all its homology groups $H_{i}(M;A)$,
$i\geq1$, are trivial. Recall the following classical topological characterization
of $\mathbb{R}^{2}$ as a smooth manifold:
\begin{prop}
\label{prop:TopCharac_Euc_Plane} A smooth $2$-dimensional real manifold $M$  is diffeomorphic to $\mathbb{R}^{2}$ if
and only if it is connected and $\mathbb{Z}_{2}$-acylic. 
\end{prop}


\section{Algebro-topological characterizations of $\mathbb{Q}$-homology euclidean
planes } 

\begin{defn}
\label{def:hom_euc_plane} A \emph{homology} (resp. $\mathbb{Q}$-homology)
\emph{euclidean plane} is a smooth real algebraic surface
$S$ such that $S(\mathbb{R})$ is diffeomorphic to $\mathbb{R}^{2}$
and whose complexification $S_{\mathbb{C}}(\mathbb{C})$ is $\mathbb{Z}$-acyclic (resp. $\mathbb{Q}$-acyclic).
\end{defn}

Recall that by virtue of respective results of Fujita \cite{Fu82} and Gurjar-Pradeep-Shastri
\cite{GuP97,GuPS97}, a $\mathbb{Q}$-acyclic complex surface is affine and rational.
\begin{prop} \label{prop:hom_euc_plane_aff_rat} A $\mathbb{Q}$-homology euclidean plane is affine and $\mathbb{R}$-rational. 
\end{prop}
\begin{proof}
Let $S$ be $\mathbb{Q}$-homology euclidean plane and let $(V,B)$ be a smooth projective completion of $S$ defined over $\mathbb{R}$. Then $V$ is geometrically rational, with non empty connected real locus $V(\mathbb{R})$ as $S(\mathbb{R})\approx \mathbb{R}^{2}$, hence is $\mathbb{R}$-rational by virtue of \cite{Co12} (see also \cite[Corollary~VI.6.5]{Sil89}). The $\mathbb{R}$-rationality of $S$ follows. 
\end{proof}

\subsection{Criteria for $\mathbb{Q}$-acyclicity and structure of the real locus }

Every smooth affine surface $S$ admits a smooth projective completion $(V,B)$ with geometrically connected
SNC boundary divisor $B$. The following well-known lemma (see e. g. \cite[Lemma 4.2.1]{MiyBook})
provides a characterization of the $\mathbb{Q}$-acyclicity of $S_{\mathbb{C}}(\mathbb{C})$
in terms of the geometry of $B$ and the properties of the natural
map $j_{\mathbb{C}}:\mathbb{Z}\langle B_{\mathbb{C}}\rangle\rightarrow\mathrm{Cl}(V_{\mathbb{C}})$
associating to an irreducible component of $B_{\mathbb{C}}$ its class
in the divisor class group $\mathrm{Cl}(V_{\mathbb{C}})$ of $V_{\mathbb{C}}$. 
\begin{lem}
\label{lem:Q_hom_alg_top_charac} Let $(V,B)$ be a smooth projective completion of a smooth complex surface $S$. Then $S(\mathbb{C})$ is $\mathbb{Q}$-acyclic if and only if $B$ is a rational
tree and the map $j\otimes\mathrm{id}:\mathbb{Z}\langle B\rangle\otimes_{\mathbb{Z}}\mathbb{Q}\rightarrow\mathrm{Cl}(V)\otimes_{\mathbb{Z}}\mathbb{Q}$ is an isomorphism. Furthermore, $S(\mathbb{C})$ is $\mathbb{Z}$-acyclic if and only if $j:\mathbb{Z}\langle B\rangle\rightarrow\mathrm{Cl}(V)$ is an isomorphism. 
\end{lem}

Not all smooth real algebraic surfaces $S$ with $\mathbb{Q}$-acyclic
complexifications $S_{\mathbb{C}}$ have their real locus diffeomorphic
to $\mathbb{R}^{2}$. For instance, the real locus of the complement
$S$ of a smooth conic $C$ in $\mathbb{P}_{\mathbb{R}}^{2}$ is either
diffeomorphic to $\mathbb{R}\mathbb{P}^{2}$ if $C(\mathbb{R})=\emptyset$
or to the disjoint union of $\mathbb{R}^{2}$ with a Möebius band otherwise.
In this context, the algebro-topological criterion of Lemma \ref{lem:Q_hom_alg_top_charac}
can be refined as follows: since the pair $(V_{\mathbb{C}},B_{\mathbb{C}})$
is defined over $\mathbb{R}$, the free abelian groups $\mathbb{Z}\langle B_{\mathbb{C}}\rangle$
and $\mathrm{Cl}(V_{\mathbb{C}})$ both inherits a structure of $G$-module
for the group $G=\{1,\sigma\}\simeq\mathbb{Z}_{2}$ generated by the
real structure $\sigma$ on $V_{\mathbb{C}}$\footnote{On $\mathrm{Cl}(V_{\mathbb{C}})$, we consider the natural action induced by $\sigma$: if $d$ is the class in $\mathrm{Cl}(V_{\mathbb{C}})$ of a real divisor $D$, then $\sigma(d)=d$. Recall that if $[D]$ is the fundamental class of $D$ in $H_{2}(V_{\mathbb{C}}(\mathbb{C});\mathbb{Z})$, then $\sigma_*([D])=-[D]$, in fact the cycle map $\mathrm{cl}:\mathrm{Cl}(V_{\mathbb{C}})\rightarrow H_{2}(V_{\mathbb{C}}(\mathbb{C});\mathbb{Z})$ is anti-equivariant.}. 
Furthermore, the complexification
of divisors gives rise to a homomorphism $\mathrm{Cl}(V)\rightarrow\mathrm{Cl}(V_{\mathbb{C}})$
whose image is contained in the subgroup $\mathrm{Cl}(V_{\mathbb{C}})^{\sigma}$ of $\sigma$-invariant classes. 
Recall that for every $G$-module $M$, the Galois cohomology groups
$H^{1}(G,M)=\mathrm{Ker}(\mathrm{id}_{M}+\sigma)/\mathrm{Im}(\mathrm{id}_{M}-\sigma)$
and $H^{2}(G,M)=\mathrm{Ker}(\mathrm{id}_{M}-\sigma)/\mathrm{Im}(\mathrm{id}_{M}+\sigma)$
are both $\mathbb{Z}_{2}$-vector spaces.
We have the following criterion: 
\begin{prop}
\label{prop:Galois_real_locus_charac} Let $(V,B)$ be a smooth projective completion defined over $\mathbb{R}$ of an $\mathbb{R}$-rational real algebraic surface $S$. Suppose that $B$ is a geometrically rational tree and let $j_{\mathbb{C}}:\mathbb{Z}\langle B_{\mathbb{C}}\rangle\rightarrow\mathrm{Cl}(V_{\mathbb{C}})$
be the natural homomorphism. Then the following hold:

1) $S(\mathbb{R})$ is diffeomorphic to $\mathbb{R}^{2}$ if and only
if $B(\mathbb{R})$ is non empty and the homomorphism $H^{2}(j_{\mathbb{C}}):H^{2}(G,\mathbb{Z}\langle B_{\mathbb{C}}\rangle)\rightarrow H^{2}(G,\mathrm{Cl}(V_{\mathbb{C}}))$
induced by $j_{\mathbb{C}}$ is an isomorphism. 

2) If in addition $\mathrm{Cl}(V)\rightarrow\mathrm{Cl}(V_{\mathbb{C}})$
is an isomorphism, then the second condition can be replaced by the
requirement that $j_{\mathbb{C}}\otimes\mathrm{id}:\mathbb{Z}\langle B_{\mathbb{C}}\rangle\otimes_{\mathbb{Z}}\mathbb{Z}_{2}\rightarrow\mathrm{Cl}(V_{\mathbb{C}})\otimes_{\mathbb{Z}}\mathbb{Z}_{2}$
is an isomorphism.
\end{prop}

\begin{proof}
Because $V$ is $\mathbb{R}$-rational, the cycle map $\mathrm{cl}:\mathrm{Cl}(V_{\mathbb{C}})\rightarrow H_{2}(V_{\mathbb{C}}(\mathbb{C});\mathbb{Z})$ is an isomorphism by the Lefschetz Theorem on $(1,1)$-cycles and we get an isomorphism between $\mathrm{Cl}(V)$ and the group $H_{2}(V_{\mathbb{C}}(\mathbb{C});\mathbb{Z})^{-\sigma}$ of anti-invariant classes.
The Borel-Haefliger homomorphism, \cite[$\S$5]{BH61}, induces an isomorphism between $H^{2}(G,\mathbb{Z}\langle C_{\mathbb{C}}\rangle)$ and $H_{1}(C(\mathbb{R});\mathbb{Z}_{2})$ for any geometrically rational curve $C$ with non empty real locus, hence 
for any geometrically rational tree with non empty real locus. As a consequence, passing to a minimal model, we get an isomorphism between $H^{2}(G,\mathrm{Cl}(V_{\mathbb{C}}))$ and $H_{1}(V(\mathbb{R});\mathbb{Z}_{2})$ if $V$ is $\mathbb{R}$-rational (see e. g. \cite[Proposition~3.2 and Theorem~3.4]{Sil89} for further details).

From Proposition \ref{prop:TopCharac_Euc_Plane}, we have
that $S(\mathbb{R})\approx \mathbb{R}^{2}$
if and only if $H^{*}(S(\mathbb{R});\mathbb{Z}_{2})=\mathbb{Z}_{2}$.
The long exact sequence of homology for the pair $(V(\mathbb{R}),B(\mathbb{R}))$
together with Poincar\'e duality 
$$
H_{i}(V(\mathbb{R}),B(\mathbb{R});\mathbb{Z}_{2})\simeq H^{2-i}(S(\mathbb{R});\mathbb{Z}_{2})
$$
yields the exact sequence 
\begin{eqnarray*}
0 & \rightarrow & H_{2}(V(\mathbb{R});\mathbb{Z}_{2})\rightarrow H^{0}(S(\mathbb{R});\mathbb{Z}_{2})\rightarrow H_{1}(B(\mathbb{R});\mathbb{Z}_{2})\stackrel{i_{\mathbb{R}}}{\rightarrow}H_{1}(V(\mathbb{R});\mathbb{Z}_{2})\rightarrow H^{1}(S(\mathbb{R});\mathbb{Z}_{2})\rightarrow\\
 & \rightarrow & H_{0}(B(\mathbb{R});\mathbb{Z}_{2})\rightarrow H_{0}(V(\mathbb{R});\mathbb{Z}_{2})\rightarrow H^{2}(S(\mathbb{R});\mathbb{Z}_{2})\rightarrow0,
\end{eqnarray*}
Again, the $\mathbb{R}$-rationality of $V$ implies that $V(\mathbb{R})$
is non empty and connected, so that $H_{0}(V(\mathbb{R});\mathbb{Z}_{2})\simeq H_{2}(V(\mathbb{R});\mathbb{Z}_{2})=\mathbb{Z}_{2}$. 
Furthermore, $B$ is a geometrically rational
tree in a smooth real algebraic surface, hence it follows from classification of involutions on a tree and classification of real structures on $\mathbb{P}^{1}_{\mathbb{C}}$ that $B(\mathbb{R})$ 
is either empty, or a point or a connected union of curves homeomorphic
to $S^1$. So either $H_{0}(B(\mathbb{R});\mathbb{Z}_{2})=0$
if $B(\mathbb{R})$ is empty, and then $H^{2}(S(\mathbb{R});\mathbb{Z}_{2})=\mathbb{Z}_{2}$,
or the map $H_{0}(B(\mathbb{R});\mathbb{Z}_{2})\rightarrow H_{0}(V(\mathbb{R});\mathbb{Z}_{2})$
is an isomorphism. 
We conclude
that $S(\mathbb{R})\approx \mathbb{R}^{2}$
if and only if $B(\mathbb{R})$ is not empty and $i_{\mathbb{R}}$
is an isomorphism. The first assertion is then a consequence of 
 the following commutative diagram (vertical isomorphism on the left is the Borel-Haefliger homomorphism)
\begin{eqnarray*}
H_{1}(B(\mathbb{R});\mathbb{Z}_{2}) & \stackrel{i_{\mathbb{R}}}{\longrightarrow} & H_{1}(V(\mathbb{R});\mathbb{Z}_{2})\\
\uparrow\wr &  & \uparrow\wr\\
H^{2}(G,\mathbb{Z}\langle B_{\mathbb{C}}\rangle) & \stackrel{H^{2}(j_{\mathbb{C}})}{\longrightarrow} & \mathrm{H^{2}(G,\mathrm{Cl}(V_{\mathbb{C}}))}.
\end{eqnarray*}
For the second assertion, it is enough to observe that if $\mathrm{Cl}(V)\rightarrow\mathrm{Cl}(V_{\mathbb{C}})$
is an isomorphism then $\mathrm{Cl}(V_{\mathbb{C}})=\mathrm{Cl}(V_{\mathbb{C}})^{\sigma}$,
$(\mathrm{id}_{\mathrm{Cl}(V_{\mathbb{C}})}+\sigma)\mathrm{Cl}(V_{\mathbb{C}})=2\mathrm{Cl}(V_{\mathbb{C}})$
and so $H^{2}(G,\mathrm{Cl}(V_{\mathbb{C}}))=\mathrm{Cl}(V_{\mathbb{C}})\otimes_{\mathbb{Z}}\mathbb{Z}_{2}$.
For the same reason, $H^{2}(G,\mathbb{Z}\langle B_{\mathbb{C}}\rangle)=\mathbb{Z}\langle B_{\mathbb{C}}\rangle\otimes_{\mathbb{Z}}\mathbb{Z}_{2}$
and the conclusion follows. 
\end{proof}

\subsection{\label{sub:Curves-arrangments} $\mathbb{Q}$-homology euclidean
planes obtained from rational plane curves arrangements }

Here we setup a real counterpart of a general blow-up construction
already used by tom Dieck and Petrie \cite{tDPe93} in the complex
case which leads to a rough description of $\mathbb{Q}$-homology
euclidean planes in terms of a datum consisting of a suitable arrangement
$D$ of  rational curves in $\mathbb{P}_{\mathbb{R}}^{2}$
and a subtree $B$ of the total transform of $D$ in a log-resolution $\tau:V\rightarrow\mathbb{P}_{\mathbb{R}}^{2}$
of the pair $(\mathbb{P}_{\mathbb{R}}^{2},D)$. This construction
will be refined later on in subsection \ref{sub:Cycle-cutting-General-Type}
to describe in a more precise fashion the structure of homology euclidean
planes of general type. \\

\subsubsection{\label{par:curve_arrangement_setup}\noindent }

Let $\mathbf{k}=\mathbb{R}$ or $\mathbb{C}$ and let $(V_{0},D_0)$ be a pair consisting of a smooth $\mathbf{k}$-rational
projective surface $V_{0}$ and a reduced curve $D_0$ defined over $\mathbf{k}$, with geometrically
rational irreducible components. Let $\tau:(V,D)\rightarrow (V_0,D_0)$, where $D=\tau^{-1}(D_0)$ 
be a strict log-resolution of $(V_{0},D_0)$ such that the image of the exceptional locus 
of $\tau$ is contained in $D_0$. 

Now suppose that there exists a geometrically rational subtree $B\subset D$ defined over $\mathbf{k}$
satisfying the following properties:

a) The support of $B$ contains the proper transform $\tau_{*}^{-1}D_0$
of $D_0$,

b) $\mathrm{rk}(\mathbb{Z}\langle D_{\mathbb{C}}\rangle)-\mathrm{rk}(\mathbb{Z}\langle B_{\mathbb{C}}\rangle)=\mathrm{rk}(\mathbb{Z}\langle D_{0,\mathbb{C}}\rangle)-\mathrm{rk}(\mathrm{Cl}(V_{0,\mathbb{C}}))$.

By assumption, the set $\mathcal{E}_{0}$ of irreducible components
of $D_{\mathbb{C}}$ not contained in the support of $B_{\mathbb{C}}$
is a subset of the set $\mathcal{E}$ of exceptional divisors of $\tau_{\mathbb{C}}:V_{\mathbb{C}}\rightarrow V_{0,\mathbb{C}}$. 
Letting $\mathcal{E}_{1}=\mathcal{E}\setminus\mathcal{E}_{0}$, we have natural identifications $ \mathbb{Z}\langle D_{\mathbb{C}}\rangle=\mathbb{Z}\langle B_{\mathbb{C}}\rangle \oplus \mathbb{Z}\langle\mathcal{E}_{0}\rangle$, \[\mathbb{Z}\langle B_{\mathbb{C}}\rangle=\mathbb{Z}\langle\tau_{*}^{-1}D_{0,\mathbb{C}}\rangle\oplus\mathbb{Z}\langle\mathcal{E}_{1}\rangle= \mathbb{Z}\langle D_{0,\mathbb{C}}\rangle\oplus\mathbb{Z}\langle\mathcal{E}_{1}\rangle\]  and  $\mathrm{Cl}(V_{\mathbb{C}})=\mathrm{Cl}(V_{0,\mathbb{C}})\oplus\mathbb{Z}\langle\mathcal{E}_{1}\rangle\oplus\mathbb{Z}\langle\mathcal{E}_{0}\rangle$. We let $R$ be the kernel of \[\pi:=\mathrm{pr}_{\mathrm{Cl}(V_{0,\mathbb{C}})\oplus\mathbb{Z}\langle\mathcal{E}_{1}\rangle}\circ j_{\mathbb{C}}=d_{\mathbb{C}}\oplus\mathrm{id}_{\mathbb{Z}\langle\mathcal{E}_{1}\rangle}:\mathbb{Z}\langle B_{\mathbb{C}}\rangle=\mathbb{Z}\langle D_{0,\mathbb{C}}\rangle\oplus\mathbb{Z}\langle\mathcal{E}_{1}\rangle \rightarrow \mathrm{Cl}(V_{0,\mathbb{C}})\oplus\mathbb{Z}\langle\mathcal{E}_{1}\rangle,\] and we let $\varphi=\mathrm{pr}_{\mathbb{Z}\langle\mathcal{E}_{0}\rangle} \circ j_{\mathbb{C}}|_R:R \rightarrow \mathbb{Z}\langle\mathcal{E}_{0}\rangle $. 

When $\mathbf{k}=\mathbb{R}$, the fact that $D_0$ and $B$ are defined
over $\mathbb{R}$ guarantees that $R$, $\mathbb{Z}\langle\mathcal{E}_{0}\rangle$ and $\mathbb{Z}\langle\mathcal{E}_{1}\rangle$
have the additional structures of $G$-modules for the Galois group
$G=\{1,\sigma\}\simeq\mathbb{Z}_{2}$ generated by the real structure
$\sigma$ on $V_{\mathbb{C}}$ and that $\pi$ and $\varphi$ are homomorphisms
of $G$-module. 

\begin{lem}
\label{lem:plane_arrangement_top_charac}\label{Rk:Arrangement-top-carac-gen} $($See also \cite[3.6-3.9]{tDPe93}$)$
Let $(V_{0},D_0)$ be a pair consisting of a smooth $\mathbf{k}$-rational
projective surface $V_{0}$ and a reduced curve $D_0$ defined over $\mathbf{k}$, with geometrically
rational irreducible components. Let $\tau:(V,D)\rightarrow (V_0,D_0)$
be a strict log-resolution of $(V_{0},D_0)$ such that the image of the exceptional locus 
of $\tau$ is contained in $D_0$, and let  $B\subset D$ be a geometrically rational subtree defined over $\mathbf{k}$ satisfying conditions a) and b) above. Then the following hold for the smooth quasi-projective surface $S=V\setminus B$:

a) $S_{\mathbb{C}}(\mathbb{C})$ is $\mathbb{Q}$-acyclic if and only if $d_{\mathbb{C}}:\mathbb{Z}\langle D_{0,\mathbb{C}}\rangle\otimes_{\mathbb{Z}}\mathbb{Q}\rightarrow\mathrm{Cl}(V_{0,\mathbb{C}})\otimes_{\mathbb{Z}}\mathbb{Q}$ is surjective and
$\varphi\otimes\mathrm{id}:R\otimes_{\mathbb{Z}}\mathbb{Q}\rightarrow\mathbb{Z}\langle\mathcal{E}_{0}\rangle\otimes_{\mathbb{Z}}\mathbb{Q}$
is an isomorphism.

b) $S_{\mathbb{C}}(\mathbb{C})$ is $\mathbb{Z}$-acylic if and only
if $d_{\mathbb{C}}:\mathbb{Z}\langle D_{0,\mathbb{C}}\rangle\rightarrow\mathrm{Cl}(V_{0,\mathbb{C}})$
is surjective and $\varphi:R\rightarrow\mathbb{Z}\langle\mathcal{E}_{0}\rangle$
is an isomorphism. 

c) If $\mathbf{k}=\mathbb{R}$ then $S(\mathbb{R})$ is diffeomorphic
to $\mathbb{R}^{2}$ if and only if $H^{2}(d_{\mathbb{C}}):H^{2}(G,\mathbb{Z}\langle D_{0,\mathbb{C}}\rangle)\rightarrow H^{2}(G,\mathrm{Cl}(V_{0,\mathbb{C}}))$ is surjective and $H^{2}(\varphi):H^{2}(G,R)\rightarrow H^{2}(G,\mathbb{Z}\langle\mathcal{E}_{0}\rangle)$
is an isomorphism. Furthermore, when $d:\mathbb{Z}\langle D_0\rangle\rightarrow\mathrm{Cl}(V_0)$ is surjective and $\mathrm{Cl}(V)\rightarrow\mathrm{Cl}(V_{\mathbb{C}})$ is an isomorphism, the second condition is satisfied if and only
if $\varphi\otimes\mathrm{id}:R\otimes_{\mathbb{Z}}\mathbb{Z}_{2}\rightarrow\mathbb{Z}\langle\mathcal{E}_{0}\rangle\otimes_{\mathbb{Z}}\mathbb{Z}_{2}$ is an isomorphism.
\end{lem}
\begin{proof}
Since a homomorphism of modules $f:M\rightarrow M'\oplus M''$ is an isomorphism if and only if the projection $\mathrm{pr}_{M'}\circ f$ is surjective and the kernel of this projection is mapped isomorphically onto $M''$ by the restriction of $\mathrm{pr}_{M''}\circ f$, the assertions are straightforward consequences of Lemma \ref{lem:Q_hom_alg_top_charac} and Proposition \ref{prop:Galois_real_locus_charac}, applied to $j_{\mathbb{C}}\otimes \mathrm{id}$,  $j_{\mathbb{C}}$ and $H^2(j_{\mathbb{C}})$.  
\end{proof}

\begin{prop}
\label{thm:plan_arrangement_structure} $($See also \cite[Proposition 2.2]{tDPe93}$)$. Let $\mathbf{k}=\mathbb{R}$
or $\mathbb{C}$ and let $S$ be a smooth $\mathbf{k}$-rational surface
such that $S_{\mathbb{C}}(\mathbb{C})$ is $\mathbb{Q}$-acyclic.
In the case where $\mathbf{k}=\mathbb{R}$, assume further that $S(\mathbb{R})$
is non compact. Then there exists an arrangement $D_0$ of reduced geometrically
rational curves in $\mathbb{P}_{\mathbf{k}}^{2}$ and a rational subtree
$B$ of the reduced total transform of $D_0$ in a strict log-resolution $\tau:(V,D) \rightarrow (\mathbb{P}_{\mathbf{k}}^{2},D_0)$
of the pair $\left(\mathbb{P}_{\mathbf{k}}^{2},D_0\right)$ satisfying properties a) and b) in
$\S$ \ref{par:curve_arrangement_setup} above such that $S\simeq V\setminus B$. 
\end{prop}
\begin{proof}
Let $(V',B')$ be a smooth projective completion of
$S$ defined over $\mathbf{k}$, with boundary geometrically rational tree
$B'$. Since $V'$ is $\mathbf{k}$-rational the output $W$
of a MMP process $\alpha:V'\rightarrow W$ over $\mathbf{k}$ ran from
$V'$ is isomorphic over $\mathbf{k}$ either to $\mathbb{P}_{\mathbf{k}}^{2}$,
or to a Hirzebruch surface $\pi_{n}:\mathbb{F}_{n}=\mathbb{P}(\mathcal{O}_{\mathbb{P}_{\mathbf{k}}^{1}}\oplus\mathcal{O}_{\mathbb{P}_{\mathbf{k}}^{1}}(-n))\rightarrow\mathbb{P}_{\mathbb{\mathbf{k}}}^{1}$,
$n\in\mathbb{Z}_{\geq0}\setminus\{1\}$, or to the smooth quadric
$Q=\{x^{2}+y^{2}+z^{2}-t^{2}=0\}\subset\mathbb{P}_{\mathbf{k}}^{3}$,
the latter being isomorphic to $\mathbb{F}_{0}$ when $\mathbf{k}=\mathbb{C}$, see \cite{Co12} or \cite[Theorem~33, p. 206]{Ko01real}.

Let us assume for the moment that $W\simeq\mathbb{P}_{\mathbf{k}}^{2}$.
Since $S$ is affine, the geometrically rational tree $B'$ is the
support of an ample divisor and hence it cannot be fully contained
in the exceptional locus of $\alpha$. Its image $D_0=\alpha_{*}B'$
is thus a reduced divisor defined over $\mathbf{k}$, with geometrically
rational irreducible components, containing the image of the exceptional
locus of $\alpha$. Furthermore, since $S_{\mathbb{C}}(\mathbb{C})$
is $\mathbb{Q}$-acyclic, the map $j_{\mathbb{C}}\otimes\mathrm{id}:\mathbb{Z}\langle D_{0,\mathbb{C}}\rangle\otimes_{\mathbb{Z}}\mathbb{Q}\rightarrow\mathrm{Cl}(\mathbb{P}_{\mathbb{C}}^{2})\otimes_{\mathbb{Z}}\mathbb{Q}$
is surjective, because $j_{\mathbb{C}}\otimes \mathrm{id}:\mathbb{Z}\langle B'_{\mathbb{C}}\rangle\otimes_{\mathbb{Z}}\mathbb{Q}\rightarrow\mathrm{Cl}(V_{\mathbb{C}}')\otimes_{\mathbb{Z}}\mathbb{Q}$
is an isomorphism. Let $\beta:V\rightarrow V'$ be a strict log-resolution
of the pair $(V,\alpha^{-1}(D_0))$ defined over $\mathbf{k}$, with SNC-minimal exceptional locus, and restricting to an isomorphism over $V'\setminus \alpha^{-1}(D_0)$. By construction, $\tau=\alpha\circ\beta:V\rightarrow\mathbb{P}_{\mathbf{k}}^{2}$
is a strict log-resolution of $(\mathbb{P}_{\mathbf{k}}^{2},D_0)$ and $S$
is isomorphic to the complement in $V$ of the geometrically rational
subtree tree $B=\beta^{-1}(B')$ of $\tilde{B}$. Furthermore, since
the exceptional locus of $\alpha$ is a disjoint union of geometrically
rational trees, the image in $V'$ of the exceptional locus of $\beta$
is contained in the support of $B$. The minimality of $\beta$
then implies that $\mathrm{rk}(\mathbb{Z}\langle\tilde{B}_{\mathbb{C}}))$
and $\mathrm{rk}(\mathbb{Z}\langle B_{\mathbb{C}}))$ differ precisely
by the number of irreducible components of the exceptional locus $\mathrm{Exc}(\alpha_{\mathbb{C}})$
of $\alpha_{\mathbb{C}}$ which are not contained in the support of
$B_{\mathbb{C}}'$. 
Since $S_{\mathbb{C}}(\mathbb{C})$ is $\mathbb{Q}$-acyclic, we get  
$\mathrm{rk}(\mathbb{Z}\langle B_{\mathbb{C}}'))=\mathrm{rk(\mathrm{Cl}(V_{\mathbb{C}}'))=1+}\mathrm{rk}(\mathbb{Z}\langle\mathrm{Exc}(\alpha_{\mathbb{C}})\rangle)$. Combined with $\mathrm{rk}(\mathbb{Z}\langle B_{\mathbb{C}}'\rangle)=\mathrm{rk}(\mathbb{Z}\langle D_{\mathbb{C}}\rangle)+\mathrm{rk}\mathbb{Z}\langle\mathrm{Exc}(\alpha_{\mathbb{C}})\cap B_{\mathbb{C}}'\rangle$, 
we conclude that $\mathrm{rk}(\mathbb{Z}\langle\tilde{B}_{\mathbb{C}}))-\mathrm{rk}(\mathbb{Z}\langle B_{\mathbb{C}}))=\mathrm{rk}(\mathbb{Z}\langle D_{\mathbb{C}}\rangle)-1$.
So $S$ is isomorphic to a surface of the form $V\setminus B$ constructed as in $\S$ \ref{par:curve_arrangement_setup} above. 

Now it remains to show that the initial smooth projective completion $(V',B')$ of $S$
and the MMP process $\alpha:V'\rightarrow W$ can be chosen so that
$W\simeq\mathbb{P}_{\mathbf{k}}^{2}$. Starting with an arbitrary
smooth projective completion $(V_{0},B_{0})$ of $S$ with
boundary geometrically rational tree $B_{0}$ and an arbitrary MMP
process $\alpha_{0}:V_{0}\rightarrow W_{0}$, we proceed as follows.

Suppose first that $W_{0}$ is a Hirzebruch surface $\pi_{n}:\mathbb{F}_{n}\rightarrow\mathbb{P}_{\mathbf{k}}^{1}$,
$n\geq2$, with exceptional section $C\simeq\mathbb{P}_{\mathbb{\mathbf{k}}}^{1}$
of self-intersection $-n$. First note that $D_{0}=(\alpha_{0})_{*}B_{0}$
cannot be equal to $C$ only, for otherwise $V_0\setminus\alpha_{0}^{-1}(D_{0})\subset S=V_0\setminus B_{0}$
would contain a complete curve, for instance the inverse image of
a section of $\pi_{n}$ disjoint from $C$, in contradiction with
the affineness of $S$. Furthermore every irreducible component of
$D_{0}$ distinct from $C$ intersects $C$ in a finite number of
closed points. So if $\mathbf{k}=\mathbb{R}$ (resp. $\mathbf{k}=\mathbb{C})$,
it follows that there exists a non real $\mathbb{C}$-rational point
$p$ of $\mathbb{P}_{\mathbb{R}}^{1}$ (resp. a closed points $p\in\mathbb{P}_{\mathbb{C}}^{1}$)
such that the intersection of $\pi_{n}^{-1}(p)$ with $D_{0}$ is
nonempty and not fully contained in $D_{0}\cap C$. Let $\varphi:\mathbb{F}_{n}\dashrightarrow\mathbb{F}_{n-2}$
(resp. $\varphi:\mathbb{F}_{n}\dashrightarrow\mathbb{F}_{n-1}$) be
the elementary birational map consisting of blowing-up a point $q\in\pi_{n}^{-1}(p)\cap (D_{0}\setminus C)$
and contracting the proper transform of $\pi_{n}^{-1}(p)$. We obtain
a commutative diagram 
\begin{eqnarray*}
(V_{0},B_{0}) & \stackrel{f}{\leftarrow} & (V_{1},B_{1})\\
\alpha_{0}\downarrow &  & \downarrow\alpha_{1}\\
W_{0}=\mathbb{F}_{n} & \stackrel{\varphi}{\dashrightarrow} & W_{1}=\begin{cases}
\mathbb{F}_{n-2} & \textrm{if }\mathbf{k}=\mathbb{R}\\
\mathbb{F}_{n-1} & \textrm{if }\mathbf{k}=\mathbb{C}
\end{cases}
\end{eqnarray*}
where $f$ and $B_{1}$ are defined as follows: if $q$ belongs to
the image of the exceptional locus of $\alpha_{0}$ then $f$ is an
isomorphism of pairs, otherwise $f$ is the blow-up of the point $q\in B_{0}(\mathbb{C})$
and $B_{1}=f^{-1}(B_{0})$. In each case, $B_{1}$ is a geometrically
rational tree defined over $\mathbf{k}$, $V_{1}\setminus B_{1}\simeq S$
and the induced birational map $\alpha_{1}:V_{1}\rightarrow W_{1}$
is a process of MMP over $\mathbf{k}$. Arguing by induction, we reach
a smooth projective completion $(V_{\ell},B_{\ell})$ of $S$
defined over $\mathbf{k}$ with geometrically rational tree boundary
$B_{\ell}$ and a process of MMP over $\mathbf{k}$ $\alpha_{\ell}:V_{\ell}\rightarrow\mathbb{F}_{\varepsilon}$,
where $\varepsilon=0,1$. In the case where $\varepsilon=1$, we eventually
obtain the desired birational morphism $\tau:V_{\ell}\rightarrow\mathbb{P}_{\mathbf{k}}^{2}$
defined over $\mathbf{k}$ by contracting the negative section of
$\pi_{1}$. 

So it remains to treat the case where $W_{0}$ is isomorphic either
to $\mathbb{F}_{0}=\mathbb{P}_{\mathbf{k}}^{1}\times\mathbb{P}_{\mathbf{k}}^{1}$
or to the smooth quadric $Q=\{x^{2}+y^{2}+z^{2}-t^{2}=0\}\subset\mathbb{P}_{\mathbf{k}}^{3}$.
The hypothesis implies that $D_{0}=(\alpha_{0})_{*}B_{0}$ has a $\mathbf{k}$-rational
point $p$. Indeed, this is clear if $\mathbf{k}=\mathbb{C}$ and,
in the case where $\mathbf{k}=\mathbb{R}$, the emptiness of $D_{0}(\mathbb{R})$
would imply that of $\alpha_{0}^{-1}(D_{0})$, and we would have 
\[
S(\mathbb{R})=V_{0}(\mathbb{R})\setminus B_{0}(\mathbb{R})\supset V_{0}(\mathbb{R})\setminus\alpha_{0}^{-1}(D_{0})(\mathbb{R})=W_{0}(\mathbb{R})\approx \begin{cases}
\mathbb{T}^{2} & \textrm{if }W_{0}=\mathbb{P}_{\mathbb{R}}^{1}\times\mathbb{P}_{\mathbb{R}}^{1}\\
\mathbb{S}^{2} & \textrm{if }W_{0}=Q,
\end{cases}
\]
in contradiction with the non compactness of $S(\mathbb{R})$. In
the case where $W_{0}=\mathbb{P}_{\mathbf{k}}^{1}\times\mathbb{P}_{\mathbf{k}}^{1}$,
we let $\varphi:\mathbb{P}_{\mathbb{\mathbf{k}}}^{1}\times\mathbb{P}_{\mathbf{k}}^{1}\dashrightarrow\mathbb{F}_{1}$
be the blow-up of $p$ followed by the contraction of the fiber of
$\mathrm{pr}_{1}$ containing $p$. Similarly as in the previous case,
we obtain a commutative diagram 
\begin{eqnarray*}
(V_{0},B_{0}) & \stackrel{f}{\leftarrow} & (V_{1},B_{1})\\
\alpha_{0}\downarrow &  & \downarrow\alpha_{1}\\
W_{0}=\mathbb{P}_{\mathbb{R}}^{1}\times\mathbb{P}_{\mathbb{R}}^{1} & \stackrel{\varphi}{\dashrightarrow} & \mathbb{F}_{1}=W_{1},
\end{eqnarray*}
where $f$ is either an isomorphism of pairs if $p$ belongs to the
image of the exceptional locus of $\alpha_{0}$, or the blow-up of
the point $p\in B_{0}(\mathbf{k})$ in which case $B_{1}=f^{-1}(B_{0})$.
By construction, $(V_{1},B_{1})$ is a smooth projective completion of $S$ with
geometrically rational tree boundary $B_{1}$ and $\alpha_{1}:V_{1}\rightarrow\mathbb{F}_{1}$
is a process of MMP over $\mathbf{k}$. The composition of $\alpha_{1}$
with the contraction of the exceptional section of $\pi_{1}$ is the
desired morphism $\tau:V_{1}\rightarrow\mathbb{P}_{\mathbf{k}}^{2}$.

Finally, in the remaining case where $\mathbf{k}=\mathbb{R}$ and
$W_{0}=Q$, we let $\varphi:Q\dashrightarrow\mathbb{P}_{\mathbb{R}}^{2}$
be the blow-up of $p$ followed by the contraction of the unique curve
$\Delta\simeq\mathbb{P}_{\mathbb{C}}^{1}$ passing through $p$ and
whose complexification $\Delta_{\mathbb{C}}\simeq\mathbb{P}_{\mathbb{C}}^{1}\cup \mathbb{P}_{\mathbb{C}}^{1}$ is of type $(1,1)$
in $\mathrm{Cl}(Q_{\mathbb{C}})$. Again, we obtain a commutative
diagram 
\begin{eqnarray*}
(V_{0},B_{0}) & \stackrel{f}{\leftarrow} & (V_{1},B_{1})\\
\alpha_{0}\downarrow &  & \downarrow\alpha_{1}\\
W_{0}=Q & \stackrel{\varphi}{\dashrightarrow} & W_{1}=\mathbb{P}_{\mathbb{R}}^{2}
\end{eqnarray*}
where $f$ is the blow-up of $p\in B_{0}(\mathbb{R})$ and $B_{1}=f^{-1}(B_{1})$
if $p$ does not belong to the image of the exceptional locus of $\alpha_{0}$
and an isomorphism of pairs otherwise. By construction, $(V_{1},B_{1})$
is a smooth projective completion of $S$ with geometrically rational
tree boundary $B_{1}$ and $\tau=\alpha_{1}:V_{1}\rightarrow\mathbb{P}_{\mathbb{R}}^{2}$
is the desired morphism. 
\end{proof}

\section{Elements of classification of homology euclidean planes}

\label{thm:hom_kod_neg_zero} In this section, we consider homology
euclidean planes $S$ up to biregular isomorphisms of schemes over $\mathbb{R}$
according to their Kodaira dimension. The cases where
$S$ have Kodaira dimension $0$ or $-\infty$ are easily dispensed
by the following observations: first there is no smooth complex $\mathbb{Z}$-acyclic
surface of Kodaira dimension $0$ at all (see e.g. \cite[Theorem 4.7.1 (1), p. 244]{MiyBook})
and second, a smooth complex $\mathbb{Z}$-acyclic surface of negative Kodaira dimension
is isomorphic to $\mathbb{A}_{\mathbb{C}}^{2}$ by virtue of \cite{MS80}.
Combined with the fact that there is no nontrivial form of the affine
$2$-space over a field of characteristic zero \cite{Kam75}, this
implies that $\mathbb{A}_{\mathbb{R}}^{2}$ is the only smooth $\mathbb{Z}$-acyclic surface
of non positive Kodaira dimension, up to isomorphisms of schemes
over $\mathbb{R}$. 

So we are left with the problem of classifying homology euclidean
planes of Kodaira dimension $1$ and $2$. A complete classification
in the first case is given in the next subsection, in the form of
a real counterpart of the existing description for complex surfaces.
In contrast, the classification of $\mathbb{Q}$-acyclic surfaces of general type
remains much more elusive, already in the complex case. Therefore,
we only establish a real counterpart of the ``cutting-cycle construction''
due to tom Dieck and Petrie \cite{tDPe93} in the complex case, from
which we derive real models of existing families of examples in the
complex case.

\subsection{\label{sub:Kod-1}Homology euclidean planes of Kodaira dimension $1$. }

Here we establish the real counterpart of the classification of smooth
complex $\mathbb{Z}$-acyclic surfaces of Kodaira dimension $1$ following Gurjar
and Miyanishi \cite{GuMi87} (see also \cite{tDPe90} for a complementary alternative construction of contractible
such surfaces in the complex case).

Let us first briefly review the scheme of the classification in the complex case, following the presentation of Chapter 3, Section 4 in \cite{MiyBook}. By \cite[Theorem 4.7.1]{MiyBook} a smooth $\mathbb{Z}$-acyclic complex surface of Kodaira dimension $1$ has an untwisted $\mathbb{A}^1_*$-fibration $q :S \rightarrow  \mathbb{P}^1_{\mathbb{C}}$, that is, a surjective morphism whose generic fiber is isomorphic to punctured affine line $\mathbb{A}^1_*$ over the field of rational functions on $\mathbb{P}^1_{\mathbb{C}}$. More precisely, it follows from \cite[Chapter 2, Theorem 6.1.5]{MiyBook} that given an arbitrary smooth projective completion $(V,B)$ of $S$, $q$ is the restriction to $S$ of the rational map $V \dashrightarrow \mathbb{P}^1_{\mathbb{C}}$ defined by the complete linear system $|m(K_V+B)|$ for sufficiently big $m\geq 1$. In what follows we refer $q:S \rightarrow  \mathbb{P}^1_{\mathbb{C}}$ to as the \emph{log-canonical} $\mathbb{A}^1_*$-fibration of $S$. 
By \cite[Lemma 4.5.1]{MiyBook}, every scheme theoretic fiber of $q$ is isomorphic to $\mathbb{A}^1_{*,\mathbb{C}}$ when equipped with its reduced structure, except for one fiber which is isomorphic to $\mathbb{A}^1_{\mathbb{C}}$. The $\mathbb{Z}$-acyclicity of $S$ then implies (see \cite[$\S$ 4.5.2 and Theorem 4.6.1]{MiyBook}) the existence of a smooth projective completion $(V,B)$ into a projective surface equipped with a $\mathbb{P}^1$-fibration $\pi:V\rightarrow \mathbb{P}^1_{\mathbb{C}}$, that is a surjective morphism with generic fiber isomorphic
to the projective line over the function field of $\mathbb{P}^1_{\mathbb{C}}$. The completion $(V,B)$ is obtained by a very specific sequence of blow-ups $\tau:V\rightarrow \mathbb{P}^2_{\mathbb{C}}$ (see $\S$ \ref{par:Kod1-Construction} below), for which  $q$ coincides with the restriction of $\pi$. The classification in the complex case can then be summarized as follows (see also \cite[Theorem 2.6]{Za99}):

\begin{thm}\label{thm:Kod1-Cplx} Let $S$ be a smooth complex $\mathbb{Z}$-acyclic surface of Kodaira dimension $1$ and let $q:S \rightarrow \mathbb{P}^1_{\mathbb{C}}$ be its log-canonical untwisted $\mathbb{A}^1_*$-fibration. Then there exists a pair $(V,B)$, consisting of a smooth projective $\mathbb{P}^1$-fibered surface $\pi:V\rightarrow \mathbb{P}^1_{\mathbb{C}}$ and a rational tree $B\subset V$ constructed by the procedure described in $\S$ \ref{par:Kod1-Construction}, and an isomorphism $S\simeq V\setminus{B}$ making the following diagram commutative  \[\xymatrix{ S \ar[r]^-{\simeq} \ar[d]_{q} & V\setminus B \ar[d]^{\pi|_{V\setminus D}} \\ \mathbb{P}^1_{\mathbb{C}} \ar@{=}[r] & \mathbb{P}^1_{\mathbb{C}}.}\]
\end{thm}

\noindent Our counterpart for surfaces defined over $\mathbb{R}$ reads as follows:  

\begin{thm} \label{thm:Fake-Kodaira-1}  Let $S$ be an integral homology euclidean plane $S$ of Kodaira dimension $1$. Then $S$ admits an untwisted $\mathbb{A}^1_*$-fibration $\rho:S\rightarrow \mathbb{P}^1_{\mathbb{R}}$ defined over $\mathbb{R}$, and there exists a pair $(V,B)$,  consisting of a smooth projective surface $\mathbb{P}^1$-fibered surface $\pi:V\rightarrow \mathbb{P}^1_{\mathbb{R}}$ defined over $\mathbb{R}$ and a tree of $\mathbb{R}$-rational curves $B\subset V$ constructed by the procedure described in $\S$ \ref{par:Kod1-Construction}, and an isomorphism $S\simeq V\setminus B$ defined over $\mathbb{R}$ making the following diagram commutative  \[\xymatrix{ S \ar[r]^-{\simeq} \ar[d]_{\rho} & V\setminus B \ar[d]^{\pi|_{V\setminus B}} \\ \mathbb{P}^1_{\mathbb{R}} \ar@{=}[r] & \mathbb{P}^1_{\mathbb{R}}.}\]
\end{thm}

The rest of this subsection is devoted to the proof of Theorem \ref{thm:Fake-Kodaira-1}. We first introduce the appropriate blow-up construction in $\S$ \ref{par:Kod1-Construction} and then proceed to the proof itself in $\S$ \ref{proof:Kod1}.   

\subsubsection{\label{par:Kod1-Construction} A blow-up construction}

Let $\mathbf{k}=\mathbb{R}$ or $\mathbb{C}$. We let $D_0\subset\mathbb{P}_{\mathbf{k}}^{2}$
be the union of a collection $E_{0,0},\ldots,E_{n,0}\simeq\mathbb{P}_{\mathbf{k}}^{1}$
of $n+1\geq3$ lines in $\mathbb{P}_{\mathbf{k}}^{2}$
intersecting in a same $\mathbf{k}$-rational point $x$ and of a
general line $C_{1}\simeq\mathbb{P}_{\mathbf{k}}^{1}$. For every
$i=1,\ldots,n$, we choose a pair of coprime integers $1\leq\mu_{i,-}<\mu_{i,+}$
in such a way that for $v_{-}=^{t}(\mu_{1,-},\ldots,\mu_{n-})\in\mathcal{M}_{n,1}(\mathbb{Z})$
and $\Delta_{+}=\mathrm{diag}(\mu_{1,+},\ldots,\mu_{n,+})\in\mathcal{M}_{n,n}(\mathbb{Z})$,
the following two conditions are satisfied\footnote{These conditions guarantee respectively that the open surface $S$ resulting from the construction has Kodaira dimension 1 and $\mathbb{Z}$-acyclic complexification, see the proof of Proposition $\S$ \ref{prop:Kod1-Construction-Properties}.}:
\begin{equation} \label{Kod1-Conditions}
\textrm{a) }\eta=n-1-{\displaystyle \sum_{i=1}^{n}\frac{1}{\mu_{i,+}}}>0  \quad\textrm{and}\quad  b)\textrm{ The matrix }\mathcal{N}=
\begin{pmatrix}
-1 & -1\cdots -1 \\
v_{-} & \Delta_{+}
\end{pmatrix}
\textrm{ belongs to }\mathrm{GL}_{n+1}(\mathbb{Z}). 
\end{equation}
Then we let $\tau:V\rightarrow\mathbb{P}_{\mathbf{k}}^{2}$ be the
smooth projective surface obtained by the following blow-up procedure: 

1) We first blow-up $x$ with exceptional divisor $C_{0}\simeq\mathbb{P}_{\mathbf{k}}^{1}$.
The resulting surface is isomorphic to the Hirzebruch surface $\pi_{1}:\mathbb{F}_{1}=\mathbb{P}(\mathcal{O}_{\mathbb{P}_{\mathbb{\mathbf{k}}}^{1}}\oplus\mathcal{O}_{\mathbb{P}_{\mathbb{\mathbf{k}}}^{1}}(-1))\rightarrow\mathbb{P}_{\mathbb{\mathbb{\mathbf{k}}}}^{1}$
with $C_{0}$ as the negative section of $\pi_{1}$, the proper transforms
of $E_{0,0},\ldots,E_{n,0}$ are $\mathbf{k}$-rational fibers of
$\pi_{1}$ while the strict transform of $C_{1}$ is a section of
$\pi_{1}$ disjoint from $C_{0}$. 

2) Then for every $i=1,\ldots,n$, we perform the subdivisional expansion
at the $\mathbf{k}$-rational point $p_{i}=C_{1}\cap E_{i,0}$
with multiplicity $(\mu_{i,-},\mu_{i,+})$ (see Example \ref{ex:subdiv_exp}) 
and exceptional divisors $E_{i,1},\ldots,E_{i,r_{i}-1},E_{i,r_{i}}=A_{0}(p_{i})$. 

3) Finally, we perform a sequence of blow-ups starting with the blow-up
of a $\mathbf{k}$-rational point $p_{0}\in E_{0,0}\setminus(C_{0}\cup C_{1})$,
with exceptional divisor $E_{0,1}\simeq\mathbb{P}_{\mathbb{\mathbf{k}}}^{1}$
and continuing with a sequence of $r_{0}-1\geq0$ blow-ups of $\mathbb{R}$-rational
points $p_{0,i}\in E_{0,i}\setminus E_{0,i-1}$, $i=1,\ldots,r_{0}-1$,
with successive exceptional divisors $E_{0,i+1}$. We let $A_{0}(p_{0})=E_{0,r_{0}}$. 

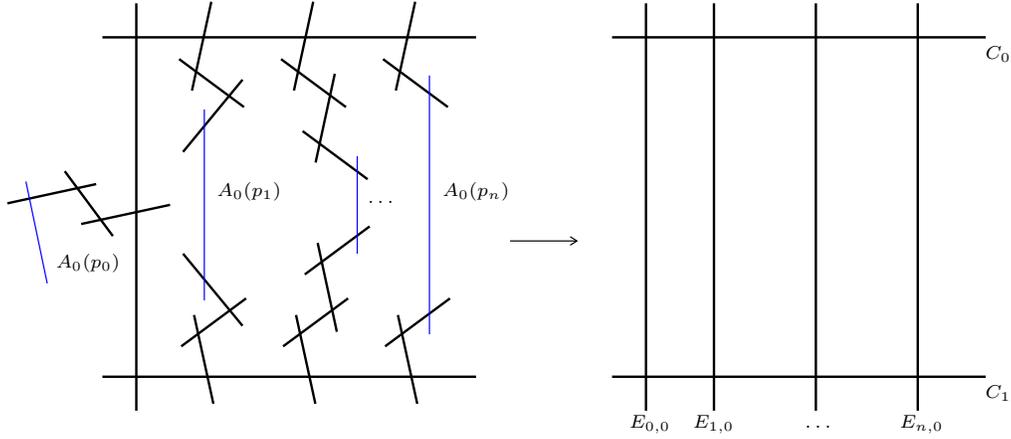
\begin{figure}[ht] 
\centering 
\input{logkod1-f.tex} 
\caption{Construction of homology euclidean plane of logarithmic Kodaira dimension $1$.}        
 \label{fig:logkod1} 
\end{figure} 

The union $B$ of the proper transforms of $C_{0}$, $C_{1}$, and
the divisors $E_{i,j}$, $i=0,\ldots,n$, $j=0,\ldots r_{i}-1$, is
a rational subtree of the total transform $\tilde{B}$ of $D_0$ by
the so-constructed morphism $\tau:V\rightarrow\mathbb{P}_{\mathbf{k}}^{2}$. We let $\tau_1:V\rightarrow  \mathbb{F}_{1}$ be the induced birational morphism, we let $\pi=\pi_{1}\circ\tau_1:V\rightarrow\mathbb{P}_{\mathbf{k}}^{1}$, and we let $S=V\setminus B$. 
By construction,  $\pi|_S:S\rightarrow\mathbb{P}_{\mathbf{k}}^{1}$ is an untwisted $\mathbb{A}^1_*$-fibration defined over $\mathbf{k}$ whose unique degenerate fibers are $A_{0}(p_0)\cap S\simeq \mathbb{A}^1_{\mathbf{k}}$ occuring with multiplicity one, and the curves $A_{0}(p_i)\cap S\simeq \mathbb{A}^1_{*,\mathbf{k}}$, $i=1,\ldots, n$, with respective multiplicities $\mu_{i,+}$. 
 
\begin{prop}
\label{prop:Kod1-Construction-Properties} With the notation above,
the surface $S$ is smooth, $\mathbf{k}$-rational, geometrically
integral. Its Kodaira dimension is equal to $1$ and $S_{\mathbb{C}}(\mathbb{C})$
is $\mathbb{Z}$-acyclic. If $\mathbf{k}=\mathbb{R}$ then $S(\mathbb{R})$
is diffeomorphic to $\mathbb{R}^{2}$. 
\end{prop}
\begin{proof}
The fact that $S$ is defined over $\mathbf{k}$, smooth, $\mathbb{\mathbf{k}}$-rational
and geometrically integral is clear from the construction. Condition \ref{Kod1-Conditions} a) in the blow-up construction implies by virtue of Theorem 4.6.1 (1) in \cite{MiyBook} that $\kappa(S)=\kappa(S_{\mathbb{C}})=1$ . The morphism $d_{\mathbb{C}}:\mathbb{Z}\langle D_{0,\mathbb{C}}\rangle\rightarrow\mathrm{Cl}(\mathbb{P}_{\mathbb{C}}^{2})$ is clearly surjective and the elements $C_{1}-E_{0,0}$ and $E_{i,0}-E_{0,0}$,
$i=1,\ldots,n$ form a basis of its kernel $R$. With the notation
of Lemma \ref{lem:plane_arrangement_top_charac}, the abelian group $\mathbb{Z}\langle\mathcal{E}_{0}\rangle$
is generated by the curves $A_{0}(p_{i})$, $i=0,\ldots n$, and by
construction, the matrix of the homomorphism $\varphi:R\rightarrow\mathbb{Z}\langle\mathcal{E}_{0}\rangle$
is equal to $\mathcal{N}$. Since $\mathcal{N}$ is invertible by
hypothesis, the $\mathbb{Z}$-acylicity of $S_{\mathbb{C}}(\mathbb{C})$
follows from Lemma \ref{lem:plane_arrangement_top_charac} a). In
the case where $\mathbf{k}=\mathbb{R}$, the fact that $S(\mathbb{R})\approx \mathbb{R}^{2}$
follows from c) in the same lemma and the surjectivity of $j:\mathbb{Z}\langle D\rangle\rightarrow\mathrm{Cl}(\mathbb{P}_{\mathbb{R}}^{2})$. 
\end{proof}

\subsubsection{\label{proof:Kod1} Classification of homology euclidean planes of Kodaira dimension
$1$}

We now proceed to the proof of Theorem \ref{thm:Fake-Kodaira-1}. So let  $S$ be a homology euclidean plane of Kodaira dimension $1$ and let $(V_{0},B_{0})$ be a smooth projective completion of $S$ with SNC boundary defined over $\mathbb{R}$. Since $\kappa(S)=1$, a multiple of the positive part of the Zariski decomposition of $K_{V_{0}}+B_{0}$ induces a rational map $\overline{\rho}_{0}:V_{0}\dashrightarrow Z$
over a smooth projective curve $Z$ defined over $\mathbb{R}$ whose
restriction to $S$ is a morphism $\rho:S\rightarrow U$ over a Zariski
open subset $U$ of $Z$ defined over $\mathbb{R}$. Furthermore, the generic
fiber of $\rho$ is a form of the punctured affine line over the function field
of $Z$ \cite[Theorem 6.1.5]{MiyBook}. Since $S$, whence $V_{0}$
is $\mathbb{R}$-rational and contains an $\mathbb{R}$-rational point,
it follows that $Z\simeq\mathbb{P}_{\mathbb{R}}^{1}$. Letting $\beta:V_1\rightarrow V_{0}$
be a minimal resolution of the indeterminacies of $\overline{\rho}_{0}$
and $B_1=\beta^{-1}(B_{0})$, the composition $\overline{\rho}_1=\overline{\rho}_{0}\circ\beta:V_1\rightarrow Z=\mathbb{P}_{\mathbb{R}}^{1}$ is a surjective fibration with generic fiber isomorphic to a form
of the projective line over the function field $\mathbb{R}(t)$ of
$Z$, which restricts to $\rho$ on the open subset $V_1\setminus B_1\simeq S$. 
This implies that $B_1$ contains a $2$-section $C'$ of $\overline{\rho}_1$.
After the base change to $\mathrm{Spec}(\mathbb{C})$, we obtain a
$\mathbb{P}^{1}$-fibration $\overline{\rho}_{1,\mathbb{C}}:V_{1,\mathbb{C}}\rightarrow Z_{\mathbb{C}}\simeq\mathbb{P}_{\mathbb{C}}^{1}$ whose restriction $\rho_{1,\mathbb{C}}:S_{\mathbb{C}}\simeq V_{1,\mathbb{C}}\setminus B_{1,\mathbb{C}}\rightarrow U_{\mathbb{C}}$ coincides with the log-canonical $\mathbb{A}^1_*$-fibration of $S_{\mathbb{C}}$. Since by \cite[Theorem 4.7.1]{MiyBook}, $\rho_{1,\mathbb{C}}$ is an untwisted $\mathbb{A}^1_*$-fibration, it follows that $C'_{\mathbb{C}}$ consists of a pair of distinct irreducible sections of $\overline{\rho}_{1,\mathbb{C}}$. By replacing $(V_1,B_1)$ by the surface obtained by blowing-up $\mathbb{R}$-rational points on $C'$, including infinitely near ones, we may assume that the irreducible components of $C'_{\mathbb{C}}$ are disjoint. Then we let $(V',B')$ be the  smooth projective completion  of $S$ obtained by contracting if necessary all possible exceptional curves of the first kind supported simultaneously in $B_1$ and in the fibers of $\overline{\rho}_1$ while keeping the following properties: a) $C'_{\mathbb{C}}$ consists of two disjoint irreducible components $(C_{0}')_{\mathbb{C}}$ and $(C_{1}')_{\mathbb{C}}$ and b) the successive proper transforms of $B_1$ are simple normal crossing divisors.

After the base change to $\mathrm{Spec}(\mathbb{C})$, we obtain a
$\mathbb{P}^{1}$-fibration $\overline{\rho}_{\mathbb{C}}':V_{\mathbb{C}}'\rightarrow Z_{\mathbb{C}}\simeq\mathbb{P}_{\mathbb{C}}^{1}$
whose restriction $\rho_{\mathbb{C}}':S_{\mathbb{C}}\simeq V'_{\mathbb{C}}\setminus B'_{\mathbb{C}}\rightarrow U_{\mathbb{C}}$ coincides with the log-canonical $\mathbb{A}^1_*$-fibration of $S_{\mathbb{C}}$. By virtue of Theorem \ref{thm:Kod1-Cplx}, there exists a pair $(V,B)$ obtained by a sequence of blow-ups $\tau:V\rightarrow \mathbb{P}^2_{\mathbb{C}}$ as in $\S$ \ref{par:Kod1-Construction} and a commutative diagram 
\[\xymatrix{S_{\mathbb{C}} \ar[d]_{\rho_{\mathbb{C}}'} \ar[r]^-{\simeq} & V\setminus B \ar[d]^{\pi|_{V\setminus B}} \\ \mathbb{P}^1_{\mathbb{C}} \ar@{=}[r] & \mathbb{P}^1_{\mathbb{C}}.}\]

\begin{lem} The birational map $\psi: V \dashrightarrow V'_{\mathbb{C}} $ induced by the isomorphism $  V \setminus B \simeq S_{\mathbb{C}} \simeq V'_{\mathbb{C}}\setminus B'_{\mathbb{C}}$ is an isomorphism of $\mathbb{P}^{1}$-fibered surfaces. 
\end{lem}

\begin{proof} Indeed, let $V\stackrel{\alpha}{\leftarrow}Y\stackrel{\alpha'}{\rightarrow}V'_{\mathbb{C}}$
be a minimal resolution of $\psi$, where $\alpha$ consists
of a possibly empty sequence of blow-ups of points supported on the
successive total transforms of $B$, and let $B_{Y}=\alpha^{-1}(B)$.
Then we have a commutative diagram 
\[
\begin{array}{ccccc}
 &  & Y\\
 & \alpha\swarrow &  & \searrow\alpha'\\
V &  & \stackrel{\psi}{\dashrightarrow} &  & V'_{\mathbb{C}}\\
\pi \downarrow &  &  &  & \downarrow\overline{\rho}_{\mathbb{C}}'\\
\mathbb{P}_{\mathbb{C}}^{1} &  & \stackrel{\sim}{\rightarrow} &  & \mathbb{P}_{\mathbb{C}}^{1}.
\end{array}
\]
Note that the proper transforms of $C_{0}$ and $C_{1}$ in $Y$ are
cross-sections of the $\mathbb{P}^{1}$-fibrations $\pi\circ\alpha$
and $\overline{\rho}_{\mathbb{C}}'\circ\alpha'$. So $\psi$ must
restrict to an isomorphism in a Zariski neighborhood of them and since
$(C_{0}')_{\mathbb{C}}$ and $(C_{1}')_{\mathbb{C}}$ are the only
cross-sections of $\overline{\rho}_{\mathbb{C}}'\circ\alpha'$ supported
on $B'_{\mathbb{C}}$, we have $\psi_{*}(C_{0}+C_{1})=(C_{0}')_{\mathbb{C}}+(C_{1}')_{\mathbb{C}}$.
If $\alpha'$ is not an isomorphism, then it factors through the contraction
of a $(-1)$-curve supported on the proper transform of $B$ in $Y$.
By construction, the only possible such curves are the proper transforms
of the curves $C_{1}$, $C_{0}$ and $E_{0,0}$ of $\mathbb{F}_{1}$. Since $\psi$ does not contract any of the first two, the only possibility would be that the proper transform of $E_{0,0}$ in $Y$ is a $(-1)$-curve,
and this can occur if and only if no proper base point of $\alpha^{-1}$
is supported on the proper transform of $E_{0,0}$ in $V$. This implies
in turn that the proper transform of $E_{0,0}$ in $Y$ still intersects
those of $C_{0}$ and $C_{1}$. But then, after the contraction of
$E_{0,0}$, the images of $C_{0}$ and $C_{1}$ would intersect each
other. This would imply in turn that $\psi_{*}(C_{0})$ and $\psi_{*}(C_{1})$
intersect each other, in contradiction with the construction of $(V',B')$. 
So $\alpha'$ is an isomorphism and we may assume from now on that
$(Y,B_{Y})=(V'_{\mathbb{C}},B'_{\mathbb{C}})$. Now suppose that $\alpha$
is not an isomorphism. Then at least one of its exceptional divisor
is a $(-1)$-curve $E$ supported on the boundary $B'_{\mathbb{C}}$ and
since $B$ is SNC, $E$ intersects at most two other irreducible components
of $B'_{\mathbb{C}}$. 

By construction of $B$, this implies that either $E$ intersects $(C_{0}')_{\mathbb{C}}$ and $(C_{1}')_{\mathbb{C}}$
simultaneously or that $E$ is an irreducible component of a fiber of $\overline{\rho}_{\mathbb{C}}'$ whose image $\overline{E}$ by the real structure on $V_{\mathbb{C}}'$ intersects $E$ and is contained in $B'_{\mathbb{C}}$. The first case is impossible as  $(C_{0}')_{\mathbb{C}}$ and $(C_{1}')_{\mathbb{C}}$ coincide with the proper transforms
of $C_{0}$ and $C_{1}$. In the second case, the image of $\overline{E}$ by the contraction of $E$ would be a curve with self-intersection $0$ contained simultaneously in $B$ and in a fiber of the $\mathbb{P}^1$-fibration $\pi$, hence would be equal to this fiber. But this would contradict the surjectivity of the restriction $q:S\rightarrow \mathbb{P}^1_{\mathbb{C}}$ of $\pi$ to $S$. So $\psi$ is an isomorphism of $\mathbb{P}^{1}$-fibered surfaces.
\end{proof} 

To complete the proof of Theorem \ref{thm:Fake-Kodaira-1}, it remains to observe the following. First since since $A_0(p_0)\cap S_{\mathbb{C}}$ is the unique fiber of $\pi|_{S_{\mathbb{C}}}$ isomorphic to $\mathbb{A}^1_{\mathbb{C}}$, we conclude that $\rho'$ has a unique degenerate fiber isomorphic to $\mathbb{A}^1_{\mathbb{R}}$. Its image by $\rho'$ is thus necessarily an $\mathbb{R}$-rational point of $Z\simeq \mathbb{P}^1_{\mathbb{R}}$ and the structure of $\pi^{-1}(\pi(A_0(p_0)))$ then implies further that the irreducible components of the fiber of $\overline{\rho}'$ over this point are all $\mathbb{R}$-rational. 

The other degenerate fibers $F_{\ell}$ of $\rho'$ are isomorphic to forms of $\mathbb{A}^1_*$ over the corresponding residue fields when equipped with their reduced structure. If $\rho'$ has such a degenerate fiber over a $\mathbb{C}$-rational point of $Z$ then there exists a pair of distinct points $p_i,p_j\in \mathbb{P}^1_{\mathbb{C}}$, $i,j\in \{ 1,\ldots , n \}$  such that the scheme-theoretic fibers of $\pi$ over $p_i$ and $p_j$ are isomorphic. With the notation of $\S$ \ref{par:Kod1-Construction}, this implies in particular that $\mu_{i,\pm}=\mu_{j,\pm}$. But since $\mu_{i,+}\geq 2$, the matrix $\mathcal{N}$ would not belong to $\mathrm{GL}_{n+1}(\mathbb{Z})$, a contradiction. So every other degenerate fiber $F_{\ell}$ is isomorphic to a form of $\mathbb{A}^1_{*,\mathbb{R}}$. In addition, the fiber of $\overline{\rho}_{\mathbb{C}}'$ over $\rho_{\mathbb{C}}'(F_{\ell,\mathbb{C}})$ is a chain of rational curves invariant under the real structure on $V_{\mathbb{C}}'$ and containing the closure of $F_{\ell,\mathbb{C}}$ as a unique invariant $(-1)$-curve. So we are left with two possibilities: either $(C_0)'_{\mathbb{C}}$ and $(C_1)'_{\mathbb{C}}$ are exchanged by the real structure on $V_{\mathbb{C}}'$ and then the fiber of $\overline{\rho}_{\mathbb{C}}'$ over $\rho_{\mathbb{C}}'(F_{\ell,\mathbb{C}})$ consists of a chain of type $[-2,-1,-2]$, or $(C_0)'_{\mathbb{C}}$ and $(C_1)'_{\mathbb{C}}$ are both invariant under the real structure and then the fiber of $\overline{\rho}'$ over $\rho'(F_{\ell})$ is a chain of $\mathbb{R}$-rational curves. In the first case, the fibers of $\pi$ over all the points $p_i$, $i \in \{ 1,\ldots , n \}$, would all be chains of the same type $[-2,-1,-2]$, and since $n\geq 2$ the matrix $\mathcal{N}$ would not belong to $\mathrm{GL}_{n+1}(\mathbb{Z})$, a contradiction.

So all degenerate fibers of $\overline{\rho}'$ are chains of $\mathbb{R}$-rational curves. By contracting all successive possible $(-1)$-curves in these fibers in the reverse order of their creation by $\tau:V\rightarrow \mathbb{P}^2_{\mathbb{C}}$ and then contracting the image of $C_{0}'$, we obtain a morphism $\tau':V'\rightarrow \mathbb{P}^2_{\mathbb{R}}$ defined over $\mathbb{R}$ which presents $V'$ as a surface obtained  by a sequence of blow-ups of $\mathbb{R}$-rational points as described in $\S$ \ref{par:Kod1-Construction}.

\subsubsection{\label{par:Kod1-Examples} Examples}

\begin{example}
\label{Ex:homotopy-Kod1} (Homotopy euclidean planes of Kodaira dimension
$1$). Let $1<a<b$ be a pair of coprime integers and let $\psi:\mathbb{P}_{\mathbb{R}}^{2}=\mathrm{Proj}(\mathbb{R}[x,y,z])\dashrightarrow\mathbb{P}_{\mathbb{R}}^{1}$
be the pencil generated by the curves $\{x^{a}z^{b-a}=0\}$ and $\{y^{b}=0\}$.
So $\psi$ has two proper base points $b_{0}=[0:0:1]$ and $b_{1}=[1:0:0]$,
a general geometrically irreducible member of $\psi$ is an $\mathbb{R}$-rational
cuspidal curve and $\psi$ has precisely two degenerate members: $\psi^{-1}([1:0])$
which is the union of the lines $L_{x}=\{x=0\}$ and $L_{z}=\{z=0\}$
counted with multiplicities $a$ and $b-a$ respectively, and $\psi^{-1}([0:1])$
which is equal to the line $L_{y}=\{y=0\}$ counted with multiplicity
$b$. Up to exchanging the roles of $x$ and $z$, we assume from
now on that $a>b-a$. 

Let $E_{0,0}$ be a general member $q$, for instance $E_{0,0}=q^{-1}([1:1])=\{x^{a}z^{b-a}+y^{b}=0\}$
and let $D=E_{0,0}\cup L_{z}$. Let $p_{0}\in E_{0,0}(\mathbb{R})\setminus\{b_{1}\}$
be a smooth $\mathbb{R}$-rational point, for instance $p_{0}=[1:-1:1]$,
let $\beta_{r_{0}}:X(a,b;r_{0})\rightarrow\mathbb{P}_{\mathbb{R}}^{2}$
be the birational morphism obtained by first blowing-up $p_{0}$ with
exceptional divisor $E_{0,1}$, and then performing a sequence of
$r_{0}-1\geq0$ blow-ups of $\mathbb{R}$-rational points $p_{0,i}\in E_{0,i}\setminus E_{0,i-1}$,
$i=1,\ldots,r_{0}-1$ with exceptional divisor $E_{0,i+1}$. We let
$S(a,b,;r_{0})=X\setminus\{E_{0,0}\cup\cdots\cup E_{0,r_{0}-1}\cup L_{z}\}$
where we identified a curve in $\mathbb{P}_{\mathbb{R}}^{2}$ with
its proper transform in $X$. A minimal resolution $\alpha:V\rightarrow X(a,b;r_{0})$
of the induced rational pencil $\beta_{r_{0}}\circ q:X(a,b;r_{0})\dashrightarrow\mathbb{P}_{\mathbb{R}}^{1}$
is isomorphic to a surface $\tau:V\rightarrow\mathbb{P}_{\mathbb{R}}^{2}$
obtained by the construction of $\S$ \ref{par:Kod1-Construction}
with $\mathbf{k=}\mathbb{R}$, $n=2$ and multiplicities $(\mu_{1,-},a)$
and $(\mu_{2,-},b)$, where $1\leq\mu_{1,-}<a$ and $1\leq\mu_{2,-}<b$
are uniquely determined in terms of $a$ and $b$ (see \cite[(2.7)]{tDPe90}
and \cite[(5.3)]{tDPe93} for the computation). 

\begin{figure}[ht] 
\centering 
\input{logkod1homotop-f.tex} 
\caption{}        
\label{fig:logkod1EX} 
\end{figure} 

Via this isomorphism, the boundary $B$ coincides with the total transform
of $E_{0,0}\cup\cdots\cup E_{0,r_{0}-1}\cup L_{z}$, the sections
$C_{0}$ and $C_{1}$ coincide respectively with the last exceptional
divisors of $\alpha$ of the points $q_{0}$ and $q_{1}$ and the
curves $A_{1}$ and $A_{2}$ are the proper transforms of $L_{x}$
and $L_{y}$ respectively. The surface $S(a,b;r_{0})_{\mathbb{C}}$
is thus $\mathbb{Z}$-acylic with $S(a,b;r_{0})(\mathbb{R})\approx\mathbb{R}^{2}$.
In fact, $S(a,b;r_{0})_{\mathbb{C}}(\mathbb{C})$ is even contractible
\cite{tDPe90}, and, using the same method as in the proof of Theorem
\ref{thm:Fake-Kodaira-1} above, one can deduce from the classification
of smooth complex contractible surfaces of Kodaira dimension $1$\emph{
}given in \emph{loc. cit.} that every homology euclidean plane $S$
of Kodaira dimension $1$ such that $S_{\mathbb{C}}(\mathbb{C})$
is contractible is isomorphic over $\mathbb{R}$ to $S(a,b;r_{0})$
for some parameters $a,b,r_{0}$ as above. 
\end{example}
\noindent
\begin{example}
Specializing the values $(a,b,r_{0})$ to $(2,3,1)$ in the previous
example, $E_{0,0}$ is the cuspidal cubic $\{x^{2}z+y^{3}=0\}\subset\mathbb{P}_{\mathbb{R}}^{2}$
and the fact that the real locus of the corresponding surface $S=S(2,3,1)$
is homeomorphic to $\mathbb{R}^{2}$ can be seen directly as follows.
Since $\beta_{1}:X\rightarrow\mathbb{P}_{\mathbb{R}}^{2}$ consists
only of the blow-up of the point $p_{0}=[1:-1:1]$, $X(\mathbb{R})$
is a Klein bottle which we view as a circle bundle $\theta\colon X(\mathbb{R})\rightarrow S_{1}$
with fibers equal to the set of $\mathbb{R}$-rational point of the
lines through $p_{0}$ in $\mathbb{P}_{\mathbb{R}}^{2}$. The sets
$E_{0,1}(\mathbb{R})$ and $L_{z}(\mathbb{R})$ are two sections of
$\theta$ which do no intersect each other. On the other hand $E_{0,0}(\mathbb{R})$
is a connected closed curve which intersects $E_{0,1}(\mathbb{R})$
and $L_{z}(\mathbb{R})$ transversally in one point and in one point
with multiplicity $3$ respectively. 

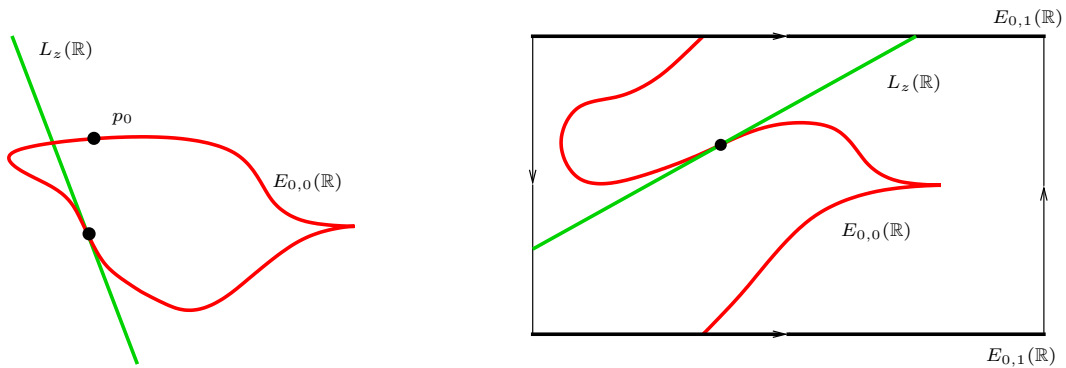
\begin{figure}[!htb]
\input{cubic-f.tex} 
\hspace{2cm}
\input{cubic-Klein-f.tex} 
\caption{The initial arrangement in $\mathbb{P}^2_{\mathbb{R}}$ and the corresponding curves in the Klein bottle $X(\mathbb{R})$  }  
\label{fig:logkod1EX2}
\end{figure}

The pair $(X(\mathbb{R}),E_{0,0}(\mathbb{R})\cup L_{z}(\mathbb{R}))$
is thus homotopically equivalent to $(X(\mathbb{R}),\ell\cup E_{0,1}(\mathbb{R}))$
where $\ell$ is a fiber of $\theta$. So $S(\mathbb{R})$ is homotopically
equivalent to $X(\mathbb{R})$ minus a fiber and a section of $\theta$
whence to a disc, implying that $S(\mathbb{R})$ is homeomorphic to
$\mathbb{R}^{2}$.
\end{example}

\subsection{Homology euclidean planes of general type. }

In this subsection, we establish the real counterpart of a refined
procedure to construct homology euclidean homology planes of general
type due to tom Dieck and Petrie \cite{tDPe93} in the complex case.
We then study the possible real forms of certain known complex families.

\subsubsection{\label{sub:Cycle-cutting-General-Type}Cycle-cutting construction}

Let again $\mathbf{k}=\mathbb{R}$ or $\mathbb{C}$, let $D\subset\mathbb{P}_{\mathbb{\mathbf{k}}}^{2}$
be a reduced curve defined over $\mathbf{k}$, with geometrically
rational irreducible components, and let $\beta:V_{0}\rightarrow\mathbb{P}_{\mathbf{k}}^{2}$
be a minimal log-resolution of the pair $(V,D)$. Given a partition
$\mathcal{E}(\beta)=\mathcal{E}_{0}\sqcup\mathcal{E}_{1}$ of the
set $\mathcal{E}(\beta)$ of irreducible exceptional divisors of $\beta$,
with associated indicator function $\chi:\mathcal{E}(\beta)\rightarrow\{0,1\}$,
we let $R_{0}=\sum_{E\in\mathcal{E}_{0}}E$, $R_{1}=\sum_{E\in\mathcal{E}_{1}}E$
and we let $D(\chi)$ be the SNC divisor on $V_{0}$ defined by 
\[
D(\chi)=\beta_{*}^{-1}(D)+R_{1}\subset\beta^{-1}(D)=\beta_{*}^{-1}(D)+R_{1}+R_{0}.
\]

\begin{defn}
A \emph{cutting datum} for a pair $(\mathbb{P}_{\mathbb{\mathbf{k}}}^{2},D)$
as above consists of

a) A partition of $\mathcal{E}(\beta)$ with indicator function $\chi:\mathcal{E}(\beta)\rightarrow\{0,1\}$
such that $D(\chi)_{\mathbb{C}}$ is connected and 
\[
\mathrm{rk}\mathbb{Z}\langle(R_{0})_{\mathbb{C}}\rangle+s(D(\chi)_{\mathbb{C}})=\mathrm{rk}\mathbb{Z}\langle D_{\mathbb{C}}\rangle-1,
\]
where $s(D(\chi)_{\mathbb{C}})$ denote the number of independent
cycles of the dual graph $\Gamma(D(\chi)_{\mathbb{C}})=(\Gamma_{0}(D(\chi)_{\mathbb{C}}),\Gamma_{1}(D(\chi)_{\mathbb{C}}))$
of $D(\chi)_{\mathbb{C}}$. 

b) A subset $\Phi$ of the set of double points of $\mathrm{Supp}(D(\chi))$
such that the subgraph $(\Gamma_{0}(D(\chi)_{\mathbb{C}}),\Gamma_{1}(D(\chi)_{\mathbb{C}})\setminus\Phi_{\mathbb{C}})$
of $\Gamma(D(\chi)_{\mathbb{C}})$ is a tree. 
\end{defn}
\begin{parn} Given a cutting datum $(\chi,\Phi)$ for a pair $(\mathbb{P}_{\mathbb{\mathbf{k}}}^{2},D)$,
we denote by $\mathcal{B}(\mathbb{P}_{\mathbf{k}}^{2},D,\chi,\Phi)$
the set of isomorphy classes of birational morphisms $\alpha:V=V(\alpha)\rightarrow V_{0}$
restricting to isomorphisms over $V_{0}\setminus\Phi$ and such that
for every $p\in\Phi$, there exists an open neighborhood $V_{0,p}$
of $p$ over which $\alpha$ restricts to a subdivisional expansion
of $V_{0,p}$ with center at $p$ (see $\S$ \ref{ex:subdiv_exp}).
For every $(\alpha:V(\alpha)\rightarrow V_{0})\in\mathcal{B}(\mathbb{P}_{\mathbf{k}}^{2},D,\chi,\Phi)$,
we let $B(\alpha)=\alpha^{-1}(D(\chi))-\sum_{p\in\Phi}A_{0}(p)$ and
$S(\alpha)=V(\alpha)\setminus B(\alpha)$. 

\end{parn}
\begin{example}
\label{Ex:Ramanujam-Surf} (Ramanujam Surfaces \cite{Ram71}, \cite[Example 3.15]{tDPe93}).
Let $D\subset\mathbb{P}_{\mathbb{R}}^{2}=\mathrm{Proj}(\mathbb{R}[x,y,z])$
be the union of the cuspidal cubic $C=\{x^{2}z+y^{3}=0\}$ with its
osculating conic $Q$ at a general $\mathbb{R}$-rational point $q\in C(\mathbb{R})$.
So $Q$ is a smooth $\mathbb{R}$-rational conic intersecting $C$
at $q$ with multiplicity $5$ and transversally at a second $\mathbb{R}$-rational
point $p$.

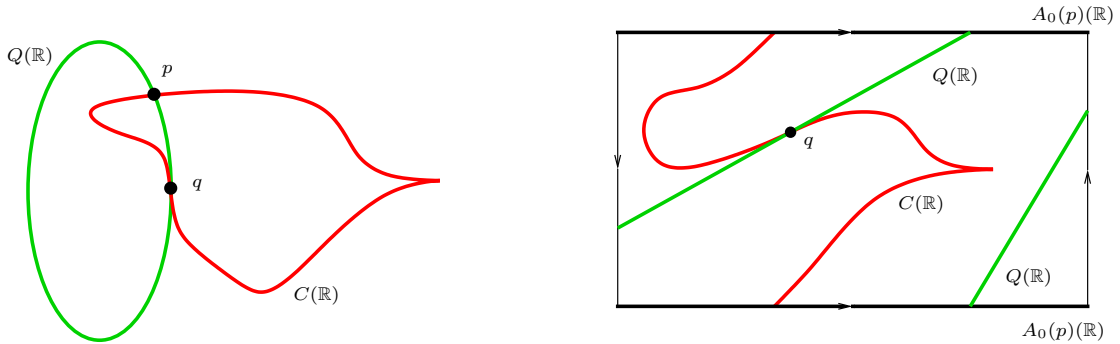
\begin{figure}[!htb]
\input{Ramanujam-f.tex} 
\hspace{2cm}
\input{Ramanujam-Klein-f.tex} 
\caption{The Ramanujam surface for the choice $(\mu_{-},\mu_{+})=(1,1)$}  
\label{fig:Ramanujam}
\end{figure}

Let $1\leq\mu_{-}\leq\mu_{+}$ be a pair of integers such that $2\mu_{-}-3\mu_{+}=\pm1$,
let $\gamma:V'\rightarrow\mathbb{P}_{\mathbb{R}}^{2}$ be the subdivisional
expansion with center at the $\mathbb{R}$-rational point $(C\cap Q)_{p}$
with multiplicities $(\mu_{-},\mu_{+})$ and last exceptional divisor
$A_{0}(p)$, let $B'=\gamma^{-1}(D)-A_{0}(p)$ and let $S=V'\setminus B'$.
Choosing $r=2C-3Q$ as the generator of the kernel $R$ of the surjective
homomorphism $d:\mathbb{Z}\langle D\rangle\rightarrow\mathrm{Cl}(\mathbb{P}_{\mathbb{R}}^{2})$,
the choice of $(\mu_{-},\mu_{+})$ guarantees that the coefficient
of $A_{0}(p)$ in $\gamma^{*}(r)$ is equal to $2\mu_{-}-3\mu_{+}=\pm1$,
whence that the induced homomorphism $\varphi:R\rightarrow\mathbb{Z}\langle A_{0}(p)\rangle$
(see $\S$ \ref{par:curve_arrangement_setup}) is an isomorphism.
Since $D(\mathbb{R})$ is not empty, we deduce from Lemma \ref{lem:plane_arrangement_top_charac}
that $S_{\mathbb{C}}(\mathbb{C})$ is $\mathbb{Z}$-acyclic and that
$S(\mathbb{R})\approx\mathbb{R}^{2}$. In fact it is known that $S_{\mathbb{C}}(\mathbb{C})$
is even contractible. The reduced total transforms of $D$ and $B'$
in the minimal log-resolutions $\beta:V_{0}\rightarrow\mathbb{P}_{\mathbb{R}}^{2}$
and $\beta':V_{0}'\rightarrow V'$ of the pairs $(\mathbb{P}_{\mathbb{R}}^{2},D)$
and $(V',B')$ respectively have the following structures: 

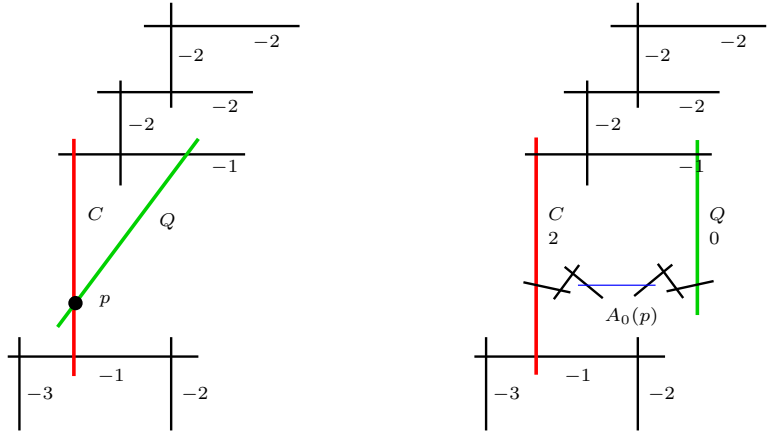
\begin{figure}[!htb]
\input{Ramanujam-log-f.tex} 
\hspace{2cm}
\input{Ramanujam-subdiv-f.tex} 
\caption{Resolved boundaries}  
\label{fig:Ramanujam-reso}
\end{figure}

So $S$ belongs to $\mathcal{B}(\mathbb{P}_{\mathbb{R}}^{2},D,\chi,\Phi)$
where $(\chi,\Phi)=(\mathbf{1}_{\mathcal{E}(\beta)},\{q\})$. Clearly,
the pair $(V_{0}',(\beta')^{-1}(B')$ cannot be birationally equivalent
to either $(\mathbb{P}_{\mathbb{R}}^{2},\mathrm{Line})$ or a pair
$(V,B)$ described in $\S$\ref{par:Kod1-Construction} via a birational
map restricting to an isomorphism on $S$. So Theorem \ref{thm:Fake-Kodaira-1}
and the fact that $\mathbb{A}_{\mathbb{R}}^{2}$ is the only homology
euclidean plane of Kodaira dimension $\leq0$ imply that $S$ is a
homology euclidean plane of general type. 
\end{example}
\noindent Theorem A in \cite{tDPe93} admits the following real counterpart: 
\begin{thm}
\label{thm:Kod2-arrangement} Let $\mathbf{k}=\mathbb{R}$ or $\mathbb{C}$
and let $S$ be a smooth $\mathbf{k}$-rational surface of general
type such that $S_{\mathbb{C}}(\mathbb{C})$ is $\mathbb{Q}$-acyclic.
In the case where $\mathbf{k}=\mathbb{R}$, assume further that $S(\mathbb{R})$
is non compact. Then there exists an arrangement $D$ of reduced geometrically
rational curves in $\mathbb{P}_{\mathbf{k}}^{2}$, such that $S$
is isomorphic over $\mathbf{k}$ to the surface $S(\alpha)$ associated
with a birational morphism $(\alpha:V(\alpha)\rightarrow W)\in\mathcal{B}(\mathbb{P}_{\mathbf{k}}^{2},D,\chi,\Phi)$
for a suitable cutting datum $(\chi,\Phi)$. \end{thm}
\begin{proof}
By virtue of Proposition \ref{thm:plan_arrangement_structure}, there
exists an arrangement $D$ of geometrically reduced and geometrically
rational curves in $\mathbb{P}_{\mathbf{k}}^{2}$ and a birational
morphism $\alpha:V\rightarrow\mathbb{P}_{\mathbf{k}}^{2}$ with the
property that $\tau^{-1}(D)$ is SNC and that $B$ is a subtree of
$\tau^{-1}(D)$ containing the proper transform $\tau_{*}^{-1}D$
of $D$. Furthermore, the image of the exceptional locus of $\tau$
is supported on $D$ and, because $S_{\mathbb{Q}}$ is $\mathbb{Z}$-acyclic,
we have $\mathrm{rk}(\mathbb{Z}\langle\tau^{-1}(D)_{\mathbb{C}}\rangle)-\mathrm{rk}(\mathbb{Z}\langle B_{\mathbb{C}}\rangle)=\mathrm{rk}(\mathbb{Z}\langle D_{\mathbb{C}}\rangle)-1$.
Since $\tau:V\rightarrow\mathbb{P}_{\mathbf{k}}^{2}$ is a log-resolution
of the pair $(\mathbb{P}_{\mathbf{k}}^{2},D)$, there exists a unique
birational morphism $\alpha:V\rightarrow V_{0}$ such that $\alpha_{*}\tau^{-1}(D)=\beta^{-1}(D)$.
The function $\chi:\mathcal{E}(\beta)\rightarrow\{0,1\}$ defined
by $\chi(E)=1$ if and only if $E\in\alpha_{*}B$ defines a partition
of $\mathcal{E}(\beta)$ and letting $\Phi$ be the image of the exceptional
locus of $\alpha$, it is enough to check that $(\chi,\Phi)$ is a
cutting datum for the pair $(\mathbb{P}_{\mathbb{R}}^{2},D)$ for
which $(\alpha:V\rightarrow V_{0})$ belongs to $\mathcal{B}(\mathbb{P}_{\mathbb{R}}^{2},D,\chi,\Phi)$.
The proof is a verbatim of that of Proposition 2.3 in \cite{tDPe93}.
\end{proof}

\subsubsection{\label{sub:Real-forms-Kod2} Real forms of homology euclidean planes
of general type}

It follows from the proof of Theorem \ref{thm:Fake-Kodaira-1} that a homology
euclidean plane $S$ of Kodaira dimension $1$ does not admit non trivial
real forms. Furthermore, it always admits a smooth projective completion
$(V,B)$ obtained from $\mathbb{P}_{\mathbb{R}}^{2}$
by blowing-up $\mathbb{R}$-rational points only and whose rational
boundary tree consists of $\mathbb{R}$-rational curves only. Here
we construct examples of homology euclidean plane of general type
for which both properties fail. 

\begin{parn} \label{par:form-construction} Consider the nodal cubic
curves $C_{1},C_{2}\subset\mathbb{P}_{\mathbb{R}}^{2}$ with respective
equations $(x-y)(x^{2}+y^{2})-xyz=0$ and $(x-y)(x^{2}-4y^{2})-xyz=0$.
Both have the $\mathbb{R}$-rational point $[1:1:0]$ as a flex but
$C_{1}$ has a second $\mathbb{C}$-rational flex $C_{1}\cap\{x^{2}+y^{2}=0\}$
while $C_{2}$ possesses two other $\mathbb{R}$-rational flexes $[2:1:0]$
and $[-2:1:0]$. So $C_{1}$ and $C_{2}$ are not $\mathbb{R}$-isomorphic
but their complexifications are both projectively equivalent over
$\mathbb{C}$ to the curve with equation $x^{3}+y^{2}-xyz=0$. The
projective duals $\Gamma_{1}$ and $\Gamma_{2}$ of $C_{1}$ and $C_{2}$
respectively are cuspidal quartics, which are nontrivial $\mathbb{R}$-forms
of each other, with projectively equivalent complexifications. They
both have an ordinary $\mathbb{R}$-rational cusp $p_{0}$ corresponding
to the common $\mathbb{R}$-rational flex of $C_{1}$ and $C_{2}$,
and either a second $\mathbb{C}$-rational ordinary cusp $q$ for
$\Gamma_{1}$, or a pair of additional $\mathbb{R}$-rational ordinary
cusps $q_{1}$ and $q_{2}$ for $\Gamma_{2}$. The tangent $L$ to
$\Gamma_{1}$ (resp. $\Gamma_{2})$ at $p_{0}$ intersects $\Gamma_{1}$
(resp. $\Gamma_{2}$) transversally in a second $\mathbb{R}$-rational
point $p$.

\begin{figure}[!htb]
\input{tricusp-quartics-f.tex} 
\caption{Real forms of a tricuspidal quartic}  
\label{fig:tricusp-quartics}
\end{figure}
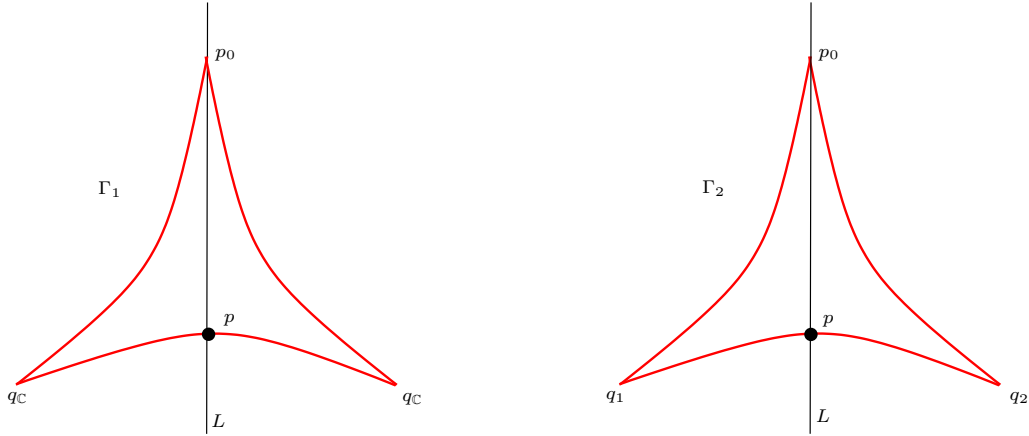

Let $D_{i}=\Gamma_{i}\cup L$, $i=1,2$ and let $1\leq\mu_{-}\leq\mu_{+}$
be a pair of integers such that $\mu_{-}-4\mu_{+}=\pm1$. Then let
$\gamma_{i}:V_{i}'\rightarrow\mathbb{P}_{\mathbb{R}}^{2}$, $i=1,2$,
be the projective surface obtained from $\mathbb{P}_{\mathbb{R}}^{2}$
by the subdivisional expansion with center at the $\mathbb{R}$-rational
point $(\Gamma_{i}\cap L)_{p}$ with multiplicities $(\mu_{-},\mu_{+})$
and last exceptional divisor $A_{0}(p)$, let $B_{i}'=\gamma_{i}^{-1}(D_{i})-A_{0}(p)$
and let $S_{i}=V_{i}'\setminus B_{1}'$. Letting $W_{i}\rightarrow V_{i}'$
be a minimal resolution of the pair $(V'_{i},B_{i}')$, the reduced
total transform $B_{2}$ of $B_{2}'$ consists of $\mathbb{R}$-rational
curves only, while the reduced total transform $B_{1}$ of $B_{1}'$
contains a chain of $\mathbb{C}$-rational curves arising from the
resolution of the $\mathbb{C}$-rational cuspidal point $q$. 

\begin{figure}[!htb]
\input{tricusp-resol-f.tex} 
\caption{Total transforms of the boundaries in the miminal resolution}  
\label{fig:tricusp-reso}
\end{figure}
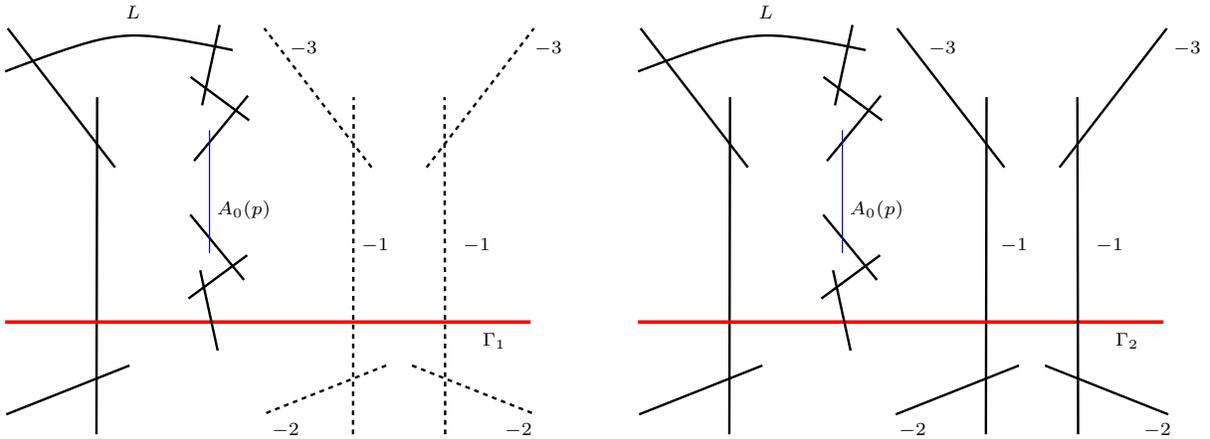

\noindent The complexifications $B_{i,\mathbb{C}}$ of $B_{i}$ have
the same structure: every irreducible component of $B_{2,\mathbb{C}}$
is invariant under the real structure $\sigma$, while $\sigma$ acts
on $B_{1,\mathbb{C}}$ by permuting the two ``cuspidal branches''
$[-3,-1,-2]$ and leaving all other irreducible component invariant. 

\end{parn}
\begin{prop}
\label{prop:Forms-Kod2}With the notation above, the following hold:

1) The surfaces $S_{1}$ and $S_{2}$ are non isomorphic homology
euclidean planes of general type, with isomorphic complexifications
$(S_{1})_{\mathbb{C}}$ and $(S_{2})_{\mathbb{C}}$.

2) The surface $S_{1}$ does not admit any smooth SNC-minimal completion
$(V,B)$ defined over $\mathbb{R}$ for which
$B$ consists of $\mathbb{R}$-rational curves only. \end{prop}
\begin{proof}
Choosing $r_{i}=\Gamma_{i}-4L$ as the generator of the kernel $R_{i}$
of the surjective homomorphism $\mathbb{Z}\langle(D_{i})\rangle\rightarrow\mathrm{Cl}(\mathbb{P}_{\mathbb{R}}^{2})$,
the choice of $(\mu_{-},\mu_{+})$ guarantees that the coefficient
of $A_{0}(p)$ in $\gamma_{i}^{*}(r_{i})$ is equal to $\mu_{-}-4\mu_{+}=\pm1$,
whence that the induced homomorphism $\varphi_{i}:R_{i}\rightarrow\mathbb{Z}\langle A_{0}(p)\rangle$(see
$\S$ \ref{par:curve_arrangement_setup}) is an isomorphism. Since
$D_{i}(\mathbb{R})$ is not empty, we deduce from Lemma \ref{lem:plane_arrangement_top_charac}
that $S_{i,\mathbb{C}}(\mathbb{C})$ is $\mathbb{Z}$-acyclic and
that $S_{i}(\mathbb{R})\approx\mathbb{R}^{2}$. The fact that $S_{i}$
is of general type follows from the same argument as in Example \ref{Ex:Ramanujam-Surf},
by comparing the structure of the minimal rational boundary tree $B_{i}$
in Figure \ref{fig:tricusp-reso} above with those described in $\S$
\ref{par:Kod1-Construction}. Since $\Gamma_{1,\mathbb{C}}$ and $\Gamma_{2,\mathbb{C}}$
are projectively equivalent, $S_{1,\mathbb{C}}$ and $S_{2,\mathbb{C}}$
are isomorphic by construction. Now suppose that $S_{1}$ admits smooth
projective completion $(V,B)$ defined over $\mathbb{R}$
for which $B$ is SNC-minimal and consists of $\mathbb{R}$-rational curves only. Then
there would exists a birational map of pairs $\varphi:(V_{1,\mathbb{C}},B_{1,\mathbb{C}})\dashrightarrow(V_{\mathbb{C}},B_{\mathbb{C}})$
defined over $\mathbb{R}$ and restricting to an isomorphism $V_{1,\mathbb{C}}\setminus B_{1,\mathbb{C}}\stackrel{\simeq}{\rightarrow}V_{\mathbb{C}}\setminus B_{\mathbb{C}}$.
Since every irreducible component of $B$ is $\mathbb{R}$-rational,
the real structure $\sigma$ on $V_{\mathbb{C}}$ acts trivially on
the set of irreducible components $B_{\mathbb{C}}$. So $\varphi$
cannot be an isomorphism of pairs because, as observed before, the
real structure $\sigma$ on $V_{1,\mathbb{C}}$ acts non trivially
on the set of irreducible components of $B_{1,\mathbb{C}}$. So $\varphi$
must be strictly birational and, letting $V_{1,\mathbb{C}}\stackrel{\alpha}{\leftarrow}X\stackrel{\alpha'}{\rightarrow}V_{\mathbb{C}}$
be a minimal resolution of $\varphi$ defined over $\mathbb{R}$,
the morphism $\alpha':X\rightarrow V_{\mathbb{C}}$ would consists
of a sequence of blow-downs of either $\sigma$-invariant $(-1)$-curves
or pairs of disjoint $(-1)$-curves exchanged by $\sigma$ supported
on the strict transform of $B_{1,\mathbb{C}}$ by $\alpha$. The structure
of $B_{1,\mathbb{C}}$ depicted in Figure \ref{fig:tricusp-reso}
above implies that the only possible such curves are the proper transforms
of the last exceptional divisors of the minimal log-resolution of
the pair $(V'_{1,\mathbb{C}},B'_{1,\mathbb{C}})$ over the three singular
points of $\Gamma_{1,\mathbb{C}}$. But the image of $\alpha^{-1}(B_{1,\mathbb{C}})$
after their contraction would no longer be SNC, which is excluded
since $B_{\mathbb{C}}$ is an SNC divisor by hypothesis. So the boundary
$B$ of every smooth projective SNC-minimal completion $(V,B)$ of $S_{1}$
must contain at least one non $\mathbb{R}$-rational component, which
shows 2). Since the pair $(V_{2},B_{2})$ constructed in $\S$ \ref{par:form-construction}
is a smooth projective SNC-minimal completion of $S_{2}$ for which $B_{2}$
consists of $\mathbb{R}$-rational curves only, we deduce in turn
that $S_{1}$ and $S_{2}$ are not isomorphic as schemes over $\mathbb{R}$. \end{proof}
\begin{rem}
The surfaces $S_{1}$ and $S_{2}$ above can also be obtained by a
cutting-cycle construction from arrangements $\Delta_{1}$ and $\Delta_{2}$
in $\mathbb{P}_{\mathbb{R}}^{2}$ consisting of lines and conics and
whose complexifications are both projectively equivalent to the following
arrangement $\Delta$ of $7$ lines with $3$ double points and $6$
triple points (see \cite{tDPe90-2}). 

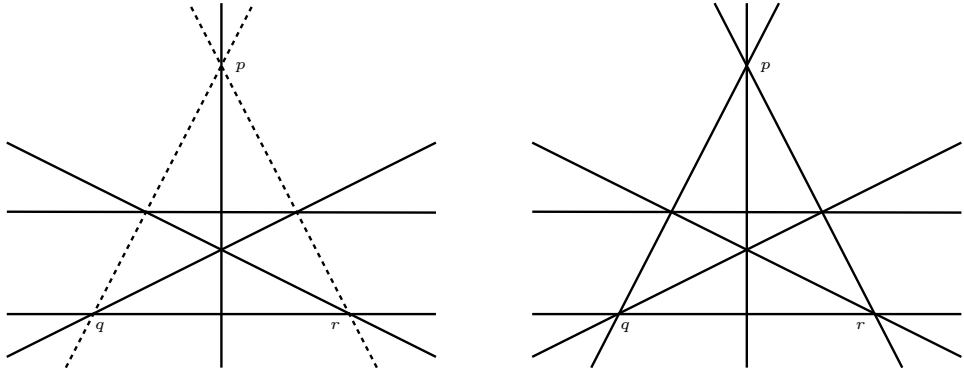
\begin{figure}[!htb]
\input{tricusp-configu-f.tex} 
\caption{Arrangement of lines associated to tricuspidal quartics}  
\label{fig:tricusp-line-arrangement}
\end{figure}

\noindent In the case of $S_{1}$ the real structure $\sigma$ on
$\mathbb{P}_{\mathbb{C}}^{2}$ acts on $\Delta$ by exchanging the
lines $pq$ and $pr$ and fixing the others while in the case of $S_{2}$,
the real structure acts trivially on $\Delta$. 
\end{rem}

\section{$\mathbb{Q}$-acyclic euclidean planes of negative Kodaira dimension }

This section is devoted to the study of $\mathbb{Q}$-acyclic euclidean
planes $S$ of negative Kodaira dimension. We first give a complete
classification of these up to isomorphisms of schemes over $\mathbb{R}$.
More precisely, we show that they all admit an $\mathbb{A}^{1}$-fibration
$\pi:S\rightarrow\mathbb{A}_{\mathbb{R}}^{1}$ defined over $\mathbb{R}$,
that is, a surjective morphism with generic fiber isomorphic
to the affine line of the function field of $\mathbb{A}_{\mathbb{R}}^{1}$
and we characterize the $\mathbb{Q}$-acyclicity of $S_{\mathbb{C}}(\mathbb{C})$
and the property that $S(\mathbb{R})\approx\mathbb{R}^{2}$ in terms
of the degenerate fibers of $\pi$. 

We then consider the classification of such euclidean planes up to
$\mathbb{R}$-biregular birational equivalence. Recall that a rational
map $\varphi:X\dashrightarrow X'$ between real algebraic varieties
is called $\mathbb{R}$-\emph{regular} if it is regular at every $\mathbb{R}$-rational
point of $X$, equivalently, the real locus of $X$ is contained in
the domain of definition of $\varphi$. We say that $\varphi$ is
$\mathbb{R}$-\emph{biregular} if it admits an $\mathbb{R}$-regular
birational inverse $\psi$. If this holds the induced morphisms $\varphi(\mathbb{R}):X(\mathbb{R})\rightarrow X'(\mathbb{R})$
and $\psi(\mathbb{R}):X'(\mathbb{R})\rightarrow X(\mathbb{R})$ are
diffeomorphisms for the euclidean topology, and are inverse to each
other. 

We establish that a large class of $\mathbb{Q}$-acyclic euclidean
plane of negative Kodaira dimension are $\mathbb{R}$-regularly birationally equivalent to the affine plane $\mathbb{A}_{\mathbb{R}}^{2}$. 

\subsection{Structure of $\mathbb{Q}$-acyclic euclidean planes of negative Kodaira
dimension}

A deep result of Miyanishi-Sugie and Fujita (see e.g. \cite[Chapter 2, Theorem 2.1.1]{MiyBook}) asserts that every smooth complex affine surface of negative Kodaira dimension admits an $\mathbb{A}^{1}$-fibration over a smooth curve. But it is wrong in general that a smooth real affine surface of negative Kodaira dimension admits such an $\mathbb{A}^{1}$-fibration defined over $\mathbb{R}$. For instance, the complement $S$ of a smooth conic $Q\subset\mathbb{P}^2_{\mathbb{R}}$ without $\mathbb{R}$-rational point has negative Kodaira dimension but no $\mathbb{A}^1$-fibration defined over $\mathbb{R}$. Indeed if $S$ admitted such an $\mathbb{A}^{1}$-fibration $\pi:S\rightarrow C$ over a smooth curve defined over $\mathbb{R}$ then $C$ would be geometrically rational, hence rational since $S$ has $\mathbb{R}$-rational points. But then the closure in $\mathbb{P}^2_{\mathbb{R}}$ of a fiber of $\pi$ over a general $\mathbb{R}$-rational point of $C$ would intersect $Q$ in a unique point, necessarily $\mathbb{R}$-rational, which is impossible. 

This subsection is devoted to the proof of the following characterization which shows in particular that $\mathbb{Q}$-acyclic euclidean planes of negative Kodaira dimension do admit $\mathbb{A}^{1}$-fibrations defined over $\mathbb{R}$. 

\begin{thm}
\label{thm:Q-acyclic-Neg-Desc} For a smooth geometrically integral
surface $S$ defined over $\mathbb{R}$ the following are equivalent:

1) $S$ is a $\mathbb{Q}$-acyclic euclidean plane of negative Kodaira dimension,

2) $S$ admits an $\mathbb{A}^{1}$-fibration $\pi:S\rightarrow\mathbb{A}_{\mathbb{R}}^{1}$
defined over $\mathbb{R}$, whose closed degenerate fibers are all
isomorphic to the affine line over the corresponding residue fields
when equipped with their reduced structure and whose degenerate fibers
over $\mathbb{R}$-rational points of $\mathbb{A}_{\mathbb{R}}^{1}$
all have odd multiplicities. 
\end{thm}

The proof requires several steps which occupy the next subsections. After recalling basic properties of $\mathbb{A}^1$-fibrations and their completions into $\mathbb{P}^1$-fibrations on smooth projective surfaces in $\S$ \ref{par:A1-fib-completion}, we establish in Proposition \ref{prop:NegKod-FiberStruct} of $\S$ \ref{par:Top-A1-fib} a characterization of $\mathbb{Q}$-acyclic euclidean plane among real surfaces $S$ admitting an $\mathbb{A}^{1}$-fibration $\pi:S\rightarrow\mathbb{A}_{\mathbb{R}}^{1}$ defined over $\mathbb{R}$. The proof of the existence of an $\mathbb{A}^1$-fibration defined over $\mathbb{R}$ on every $\mathbb{Q}$-acyclic euclidean plane (Proposition \ref{prop:A1-fib-existence} in $\S$ \ref{par:A1-Fib-FakePlane} below) then follows from a careful analysis of the structure of smooth real surfaces admitting a smooth projective completion $(V,B)$ defined over $\mathbb{R}$ for which $B$ is a rational chain which is done in $\S$ \ref{par:Rat-Chains-real}. 

\subsubsection{Basic properties of $\mathbb{A}^{1}$-fibrations and their completions} \label{par:A1-fib-completion} 
The geometry of smooth complex affine surfaces admitting an $\mathbb{A}^1$-fibration is well-understood through the study of their completions into $\mathbb{P}^1$-fibered surfaces (see e.g. \cite[Chapter 3, $\S$ 1]{MiyBook}). Let us recollect basic properties of these fibrations and their completions in a form which also covers the case of real affine surface.

\begin{lem}\label{lem:A1-fibComp} 
Let $\mathbf{k}=\mathbb{R}$ or $\mathbb{C}$, let $S$ be a smooth geometrically integral surface
defined over $\mathbf{k}$ and let $\pi:S\rightarrow C$ be an $\mathbb{A}^{1}$-fibration over
a smooth curve, either affine or projective, defined over $\mathbf{k}$.
Then there exists a smooth projective completion $(V,B)$ of $S$ defined over $\mathbf{k}$ with the following properties:

a) $V$ admits a $\mathbb{P}^1$-fibration $\overline{\pi}:V\rightarrow \overline{C}$ defined over $\mathbf{k}$ over a smooth projective completion $\overline{C}$ of $C$ and there is a commutative diagram \[\xymatrix{ S \ar[r] \ar[d]_{\pi} & V \ar[d]^{\overline{\pi}} \\ C \ar[r] & \overline{C}.}\].

b) If $S$ is affine, then the divisor $B=V\setminus S$ is a tree which can be written
in the form \[B=\bigcup_{c\in\overline{C}\setminus C}F_{c}\cup\overline{C}_{0}\cup\bigcup_{p\in C}H_{p},\]
where $\overline{C}_{0}$ is a section of $\overline{\pi}$, $F_{c}=\overline{\pi}^{-1}(c)\simeq\mathbb{P}_{\kappa(c)}^{1}$
and $H_{p}$ is either empty or a proper strictly SNC-minimal rational subtree of $\overline{\pi}^{-1}(p)$
containing a $\kappa(p)$-rational component intersecting
$\overline{C}_{0}$, the full fiber $\overline{\pi}^{-1}(p)$ being
equal to the union of $H_{p}$ and of the closure in $V$ of $\pi^{-1}(p)$.

Furthermore, the closure in $V$ of every irreducible component of $\pi^{-1}(p)$
is isomorphic to the projective line over a finite extension $\kappa'$
of $\kappa(p)$ and it intersects $H_{p}$ transversally in a unique
$\kappa'$-rational point. 
\end{lem}

\begin{figure}[!htb]
\input{logkod-infty-f.tex} 
\caption{Structure of degenerate fibers in a minimal completion}  
\label{fig:negKod-A1-fibstruct}
\end{figure}
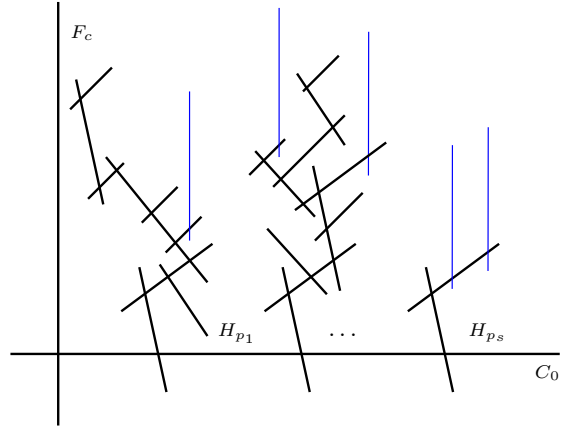

\begin{proof} Letting $V$ be any smooth projective completion of $S$ defined over $\mathbf{k}$, the fibration
$\pi:S\rightarrow C$ extends to a rational map $\overline{\pi}:V\dashrightarrow\overline{C}$
over the smooth projective model $\overline{C}$ or $C$ over $\mathbf{k}$,
and, after taking a log-resolution of the indeterminacies of $\overline{\pi}$
and of the boundary divisor $V\setminus S$, we may assume that $V$
is a smooth projective completion of $S$ with SNC boundary $B$ on
which $\pi$ extends to a morphism $\overline{\pi}:V\rightarrow\overline{C}$.
Up to performing a sequence of blow-downs defined over $\mathbf{k}$,
we may further assume that $B$ does not contain any irreducible component
$B_{i}\simeq\mathbb{P}_{\kappa}^{1}$ where $\kappa=\mathbb{R}$ or
$\mathbb{C}$, with self-intersection $-\mathrm{deg}(\kappa/\mathbf{k})$
contained in a fiber of $\overline{\pi}$ and intersecting at most
two other irreducible components of $B$. Since $V$ is smooth and
the generic fiber of $\pi$ is isomorphic to the affine line over
the function field of $C$, the generic fiber of $\overline{\pi}$
is isomorphic to the projective line over the function field of $\overline{C}$,
and the boundary divisor $B$ contains precisely one irreducible component,
say $\overline{C}_{0}$, which is a section of $\overline{\pi}$.
So there exists a birational morphism $\tau:V\rightarrow\mathbb{P}(E)$
defined over $\mathbf{k}$ to a ruled surface $\mathbb{P}(E)\rightarrow\overline{C}$
for a certain rank $2$ vector bundle $E$ over $\overline{C}$. Arguing for instance as in the proof of 
\cite[Lemma 1.0.7]{BD11}, we may, and do, further assume up to changing $\tau$ for a different birational morphism, 
that $\tau$ restricts to an isomorphism in a open neighborhood
of $\overline{C}_{0}$. As a consequence, every fiber of $\overline{\pi}$
over a closed point $c\in\overline{C}$ is a geometrically rational
tree defined over of the residue field $\kappa(c)$ of $c$, containing
at least one irreducible component isomorphic to $\mathbb{P}_{\kappa(c)}^{1}$. 
In the case where $S$ is affine, assertion b) then immediately follows from our minimality assumptions. 
\end{proof}

\subsubsection{Topology of $\mathbb{A}^1$-fibered affine surfaces}\label{par:Top-A1-fib} 
The following proposition provides in particular a characterization of $\mathbb{Q}$-acyclic euclidean plane among real surfaces $S$ admitting an $\mathbb{A}^{1}$-fibration $\pi:S\rightarrow\mathbb{A}_{\mathbb{R}}^{1}$ defined over $\mathbb{R}$. 

\begin{prop}
\label{prop:NegKod-FiberStruct} Let $S$ be a smooth geometrically
integral surface defined over $\mathbf{k}=\mathbb{R}$ or $\mathbb{C}$
and let $\pi:S\rightarrow C$ be an $\mathbb{A}^{1}$-fibration defined
over $\mathbf{k}$.

1) The surface $S_{\mathbb{C}}(\mathbb{C})$ is $\mathbb{Q}$-acyclic
if and only if $C\simeq\mathbb{A}_{\mathbf{k}}^{1}$ and every closed
fiber of $\pi$ is isomorphic to the affine line when equipped with
its reduced structure.

2) If $\mathbf{k}=\mathbb{R}$, $S$ is affine and $C\simeq\mathbb{A}_{\mathbb{R}}^{1}$,
then $S(\mathbb{R})\approx\mathbb{R}^{2}$ if and only if the scheme theoretic
fiber of $\pi$ over every $\mathbb{R}$-rational point is of the
form $mR+R'$, where $R\simeq\mathbb{A}_{\mathbb{R}}^{1}$, $m\geq1$
is odd, and $R'$ is an effective divisor whose support is disjoint
from $R$ and consists of a disjoint union of affine lines defined
over $\mathbb{C}$. \end{prop}
\begin{proof}
The first assertion follows from \cite[Theorem 4.3.1 p. 231]{MiyBook} and its proof, together with the fact that there is no nontrivial real form of $\mathbb{A}^1_{\mathbb{C}}$. 
 
For the second assertion, we consider a smooth projective completion $(V,B)$ of $S$ as in Lemma \ref{lem:A1-fibComp} above. Since $S$ is affine and rational, $V$ is obtained from a Hirzebruch surface $\rho_{n}:\mathbb{F}_{n}\rightarrow\mathbb{P}_{\mathbb{R}}^{1}$
by a sequence of blow-ups $\tau:V\rightarrow\mathbb{F}_{n}$ defined over $\mathbb{R}$
mapping the irreducible component $\overline{C}_0$ of
$B$ isomorphically onto a section of $\rho_{n}$. Up to changing $V$ for another smooth projective completion obtained by making elementary transformation with center
on the fiber $F_{\infty}\simeq \mathbb{P}^1_{\mathbb{R}}$ of $\overline{\pi}$ over $\infty=\mathbb{P}_{\mathbb{R}}^{1}\setminus\mathbb{A}_{\mathbb{R}}^{1}$, we can assume that $n=1$ and that $\tau_*(\overline{C}_0)$ is the exceptional section of $\rho_{1}$ with self-intersection $-1$.  Let $p_{1},\ldots,p_{s}$ 
and  $q_{1},\ldots,q_{r}$ be respectively the $\mathbb{R}$-rational and $\mathbb{C}$-rational points
of $\mathbb{A}_{\mathbb{R}}^{1}$ over which the fiber
of $\pi:S\rightarrow\mathbb{A}_{\mathbb{R}}^{1}$ is degenerate. Let
$F_{i}\simeq \mathbb{P}^1_{\mathbb{R}}$ and $G_{j}\simeq \mathbb{P}^1_{\mathbb{C}}$ be the fibers of $\rho_{1}$ over these points.
The images of $F_{\infty}$, $F_{i}$ and $G_{j}$ in $\mathbb{P}_{\mathbb{R}}^{2}$
by the contraction of $\tau_{*}(\overline{C}_{0})$ form an arrangement
$D$ of lines and geometrically reducible conics meeting each other
in a unique $\mathbb{R}$-rational point. This yields a presentation
of $(V_{\mathbb{C}},B_{\mathbb{C}})$ as the blow-up of an arrangement
$D_{\mathbb{C}}$ of $s+2r+1$ lines meeting each other in a unique
point. The image of $F_{\infty}\simeq\mathbb{P}_{\mathbb{R}}^{1}$
in $\mathbb{P}_{\mathbb{\mathbb{R}}}^{2}$ is a generator of $\mathrm{Cl}(\mathbb{P}_{\mathbb{R}}^{2})$,
and a basis of the kernel $R$ of the homomorphism $j_{\mathbb{C}}:\mathbb{Z}\langle D_{\mathbb{C}}\rangle\rightarrow\mathrm{Cl}(\mathbb{P}_{\mathbb{C}}^{2})$ consists for instance of the classe of the images in $\mathbb{P}_{\mathbb{C}}^{2}$
of the divisors $F_{i,\mathbb{C}}-F_{\infty,\mathbb{C}}$, $T_{j}-F_{\infty,\mathbb{C}}$
and $\overline{T}_{j}-F_{\infty,\mathbb{C}}$, where $T_{j}$ and
$\overline{T}_{j}$ denotes the two connected components of $G_{j,\mathbb{C}}$,
exchanged by the real structure $\sigma$. Let $\mathcal{E}_{0}$ be the set consisting of the complexifications of the closures $A_{i,\ell}(p_{i})$, $\ell=1,\ldots,n_{i}$ and $A_{j,\ell}(q_{j})$,
$\ell=1,\ldots,m_{j}$ in $V$ of the irreducible components of the fibers of $\pi$ over the points
$p_{i}$ and $q_{j}$. We let $\mu_{i,\ell}\geq1$ and $\nu_{j,\ell}\geq1$
be the multiplicities $A_{i,\ell}(p_{i})$ and $A_{j,\ell}(q_{j})$
in the fibers $\overline{\pi}^{*}(p_{i})$ and $\overline{\pi}^{*}(q_{j})$
respectively. 

Since $B(\mathbb{R})$ is not empty, it follows
from c) in Lemma \ref{lem:plane_arrangement_top_charac} that $S(\mathbb{R})\approx\mathbb{R}^{2}$
if and only if the map $H^{2}(\varphi):H^{2}(\mathbb{Z}_{2},R)\rightarrow H^{2}(\mathbb{Z}_{2},\mathbb{Z}\langle\mathcal{E}_{0}\rangle)$
is an isomorphism. The Galois cohomology groups $H^{2}(\mathbb{Z}_{2},R)$ and
$H^{2}(\mathbb{Z}_{2},\mathbb{Z}\langle\mathcal{E}_{0}\rangle)$ have
bases consisting respectively of the classes of the $\sigma$-invariant
curves $F_{i,\mathbb{C}}-F_{\infty,\mathbb{C}}$ and of the classes
of the complexifications of the $\mathbb{R}$-rational curves among
the $A_{i,\ell}(p_{i})$. By construction, $H^{2}(\varphi)([F_{i,\mathbb{C}}-F_{\infty,\mathbb{C}}])=\sum_{\ell=1}^{n_{i}}\mu_{i,\ell}[A_{i,\ell}(p_{i})_{\mathbb{C}}]$
where the sum is taken over these $\mathbb{R}$-rational irreducible
components. So $H^{2}(\varphi)$ is an isomorphism if and only if
for every $i=1,\ldots,s,$ there exists exactly one $\mathbb{R}$-rational
curve among the $A_{i,\ell}(p_{i})$, say $A_{i,1}(p_{i})$, and the
residue class of $\mu_{i,1}$ modulo $2$ is nonzero. This proves
2). 
\end{proof}

\subsubsection{Affine surfaces completable by a rational chain} \label{par:Rat-Chains-real}
Here we establish auxiliary results concerning the existence of "normal forms" for boundary chains of smooth affine surfaces $S$ defined over $\mathbf{k}=\mathbb{R}$ or $\mathbb{C}$ admitting a smooth completion $(V,B)$ whose boundary $B$ is a rational chain. 

\begin{lem}
\label{lem:Chains-normal-forms} Let $\mathbf{k}=\mathbb{R}$ or $\mathbb{C}$,
let $(V,B)$ be a pair defined over $\mathbf{k}$ consisting of a
smooth projective surface $V$ and a geometrically rational chain
$B$ supporting a effective ample divisor on $V$ and let $S=V\setminus B$.
Then the following hold:

1) If every irreducible component of $B$ is $\mathbf{k}$-rational,
then there exists a smooth projective completion $(V_{1},B_{1})$ of $S$
defined over $\mathbf{k}$ whose boundary $B_{1}$ is a chain of $\mathbf{k}$-rational
curves of the form $B_{1}=F\vartriangleright C\vartriangleright E$
where $F^{2}=0$, $C^{2}=-1$ and $E$ is either empty, or an irreducible
curve with self-intersection $0$, or a chain of rational curves with
self-intersections $\leq-2$. 

2) If $B$ has a non $\mathbf{k}$-rational irreducible component
then there exists a smooth projective completion $(V_{1},B_{1})$ of $S$
defined over $\mathbf{k}$ whose boundary $B_{1}$ is a geometrically
rational chain such that $B_{1,\mathbb{C}}$ has one of the following
forms:

$\quad$a) \textup{$B_{1,\mathbb{C}}=H\vartriangleright\overline{H}$
where $H$ is irreducible with self-intersection $1$ and $\overline{H}$
is its image by the real structure $\sigma$ on $V_{1,\mathbb{C}}$. }

$\quad$ b) $B_{1,\mathbb{C}}=H\vartriangleright\overline{H}$ where
$H=E\vartriangleright G$ is a chain consisting of an irreducible
curve $G$ with self-intersection $0$ and a possible empty chain
$E$ of curves with self-intersections $\leq-2$ and $\overline{H}$
is the image of $H$ by the real structure $\sigma$ on $V_{1,\mathbb{C}}$. 

$\quad$ b') $B_{1,\mathbb{C}}=H\vartriangleright\overline{H}$ where
$H=E\vartriangleright G$ is a chain consisting of an irreducible
curve $G$ with self-intersection $-1$ and a nonempty chain of curves
with self-intersections $\leq-2$ and $\overline{H}$ is the image
of $H$ by the real structure $\sigma$ on $V_{1,\mathbb{C}}$. 

$\quad$ c) \textup{$B_{1,\mathbb{C}}=H\vartriangleright C\vartriangleright\overline{H}$
}where $C$ is an irreducible curve of self-intersection $0$ invariant
by the real structure $\sigma$ on $V_{1,\mathbb{C}}$, $H$ is either
an irreducible curve with self-intersection $0$ or a chain of curves
with self-intersections $\leq-2$, except maybe its right boundary
which is a $(-1)$-curve, and $\overline{H}$ is the image of $H$
by the real structure $\sigma$ on $V_{1,\mathbb{C}}$. \end{lem}
\begin{proof}
We may assume from the very beginning that $B$ is minimal over $\mathbf{k}$,
i.e. that it does not contain any irreducible component $B_{i}\simeq\mathbb{P}_{\kappa}^{1}$,
where $\kappa=\mathbb{R}$ or $\mathbb{C}$ with self-intersection
$-\deg(\kappa/\mathbf{k})$. Being the support of an effective ample
divisor on $V_{\mathbb{C}}$, $B_{\mathbb{C}}$ must then contain
an irreducible component with self-intersection $\geq-1$. 

The first assertion is well-known, so we only sketch the argument,
referring the reader for instance to \cite{Gi-Da1} for the detail.
If every irreducible component of $B$ is $\mathbf{k}$-rational,
then the previous observation together implies that $B$ contains
at least one irreducible component of nonnegative self-intersection.
Fixing an orientation on $B$, we let $D$ be the leftmost irreducible
component of $B$ with this property and we let $D_{\leq}\subset B$
be the subchain of $B$ consisting of $D$ and the components on the
left of it. Then by a sequence of birational transformations, consisting
of blow-ups of double points of the support of the successive total
transforms of $D{}_{\leq}$ and contractions of irreducible components
of them, we can transform $(V,B)$ into a smooth projective completion
$(V',B')$ of $S$ defined over $\mathbf{k}$, whose boundary $B'$
consists of $\mathbf{k}$-rational curves, in such a way that the
left boundary $F$ of $B'$ has self-intersection $0$. (e.g. see
\cite[Lemma 2.7]{Du05} for a description of such kind of birational
transformations). The surface $V'\setminus B'\simeq V\setminus B$
being affine, $B'$ has at least a second irreducible component and
up to making additional elementary transformations with centers at
$\mathbf{k}$-rational points of $F$, we may assume that the irreducible
component of $B'$ intersecting $F$, say $C$, has self-intersection
$-1$. It follows that the complete linear system $|F|$ generates
a $\mathbb{P}^{1}$-fibration $\overline{\pi}:V'\rightarrow\mathbb{P}_{\mathbf{k}}^{1}$
having $F$ as a full fiber, $C$ as a section, and the remaining
irreducible components of $B'$ form a possibly empty chain $T$ contained
in a single other fiber of $\overline{\pi}$. If $T$ is not empty
then after contracting all successive $(-1)$-curves contained in
its support and performing additional elementary transformations with
centers at $\mathbf{k}$-rational points of $F$ if necessary, we
reach a smooth projective completion $(V_{1},B_{1})$ defined over
$\mathbf{k}$, whose boundary is a chain $B_{1}=F\vartriangleright C\vartriangleright T$
of $\mathbf{k}$-rational curves with the desired properties. Note
for further use that the initial boundary $B$ contained two disjoint
irreducible components with non negative self-intersections if and
only it consisted of a chain of three irreducible components $D\vartriangleright C\vartriangleright E$
with $D^{2}=(E)^{2}=0$ and $C^{2}$ arbitrary. Indeed, since the
birational transformations involved in the construction the pair $(V',B')$
restricts to isomorphisms outside $D_{\leq}$ and its successive total
transforms, the proper transform $E'$ in $V'$ of an irreducible
component $E\subset B$ disjoint from $D$ is disjoint from $F$ and
has self-intersection $(E')^{2}=E^{2}$. The only possibility is thus
that $E'$ is a fiber of $\overline{\pi}:V'\rightarrow\mathbb{P}_{\mathbf{k}}^{1}$.
Thus $E$ has self-intersection $0$ and is necessarily the right
boundary of $B$, and the same argument implies that $B=D\vartriangleright C\vartriangleright E$
as desired. 

Now suppose that at least one of the irreducible component of $B$
is not $\mathbf{k}$-rational. So $\mathbf{k}=\mathbb{R}$ and we have the following
alternative: either $B(\mathbb{R})$ consists of a unique point $p$
and then $B_{\mathbb{C}}$ is a chain of the form $H\vartriangleright\overline{H}$
where $H$ and $\overline{H}$ are chains intersecting each others
in $p$ and exchanged by the real structure $\sigma$ on $V_{\mathbb{C}}$,
or $B$ contains a unique geometrically irreducible component with
empty real locus, or a unique $\mathbb{R}$-rational irreducible component,
say $C_{0}$, and $B_{\mathbb{C}}$ is a chain of the form $H\vartriangleright C_{0,\mathbb{C}}\vartriangleright\overline{H}$
where $C_{0,\mathbb{C}}\simeq\mathbb{P}_{\mathbb{C}}^{1}$ and $H$
and $\overline{H}$ are possibly empty chains exchanged by the real
structure $\sigma$ on $V_{\mathbb{C}}$. We consider two sub-cases:

1) If $B_{\mathbb{C}}$ is SNC-minimal, then by the observation at
the beginning of the proof, there exists an irreducible component
$D_{0}$ of $B_{\mathbb{C}}$ with non negative self-intersection.
If $B(\mathbb{R})=\{p\}$ then $B_{\mathbb{C}}=H\cup\overline{H}$
and since $B$ defined over $\mathbb{R}$, it follows that $B_{\mathbb{C}}$
contains at least two irreducible components with non negative self-intersection,
$D_{0}$ and its image $\overline{D}_{0}$ by the real structure $\sigma$
on $V_{\mathbb{C}}$. If $D_{0}$ and $\overline{D}_{0}$ are disjoint
then $B_{\mathbb{C}}=D_{0}\cup C\cup\overline{D}_{0}$ where $C\simeq\mathbb{P}_{\mathbb{C}}^{1}$.
But then it would follow that $B(\mathbb{R})$ is either empty or
homeomorphic to $S^{1}$, a contradiction. So up to the choice of
an ordering of $B_{\mathbb{C}}$ and the exchange of $D_{0}$ and
$\overline{D}_{0}$, we may assume that $B_{\mathbb{C}}=G\vartriangleright D_{0}\vartriangleright\overline{D}_{0}\vartriangleright\overline{G}$
where $D_{0}$ and $\overline{D}_{0}$ intersect in $\{p\}$ and $G$
and $\overline{G}$ are possibly empty chains of rational curves with
self-intersection $\leq-2$ exchanged by the real structure $\sigma$.
Furthermore $D_{0}^{2}\leq1$ for otherwise, by blowing-up $\{p\}$
with exceptional $E$, we would obtain a new smooth projective completion
$(V',B')$ of $S$ defined over $\mathbb{R}$ whose boundary chain
$B'$ would have the property that $B'_{\mathbb{C}}=G\vartriangleright D_{0}\vartriangleright E\vartriangleright\overline{D}_{0}\vartriangleright\overline{G}$
contains two disjoint irreducible components with positive self-intersection.
For the same reason we conclude that either $D_{0}^{2}=1$ and then
$B=D_{0}\vartriangleright\overline{D}_{0}$ or $D_{0}^{2}=\overline{D}_{0}^{2}=0$
and then $B_{\mathbb{C}}=G\vartriangleright D_{0}\vartriangleright\overline{D}_{0}\vartriangleright\overline{G}$
where $G$ and $\overline{G}$ are possibly empty chains of rational
curves with self-intersection $\leq-2$ exchanged by the real structure
$\sigma$. This corresponds to cases a) and b) respectively.

Otherwise, if $B(\mathbb{R})=\emptyset$ or $S^{1}$ then $B_{\mathbb{C}}=H\vartriangleright C_{0,\mathbb{C}}\vartriangleright\overline{H}$.
If $H$ contains an irreducible component with non negative self-intersection
then by the same argument as above, we conclude that $B=D_{0}\vartriangleright C_{0,\mathbb{C}}\vartriangleright\overline{D}_{0}$
with $D_{0}^{2}=\overline{D}_{0}^{2}=0$ and by elementary
transformations defined over $\mathbb{R}$ with centers on $D_{0}\cup\overline{D}_{0}$,
we may assume that $C_{0,\mathbb{C}}^{2}=0$ or $-1$, the second
case being then reduced further to the one $B=D_{0}\cup\overline{D}_{0}$
with $D_{0}^{2}=\overline{D}_{0}^{2}=1$ by contracting $C_{0,\mathbb{C}}$.
The only other possibility is that $C_{0,\mathbb{C}}^{2}\geq0$ and
that $H$ and $\overline{H}$ consists of chains of curves with self-intersections
$\leq-2$ exchanged by the real structure $\sigma$. By blowing-up
pairs of double points on $C_{0,\mathbb{C}}$ exchanged by the real
structure $\sigma$, we may reduce to either case c), that is, $B_{\mathbb{C}}=H\vartriangleright C_{0,\mathbb{C}}\vartriangleright\overline{H}$
where $C_{0,\mathbb{C}}^{2}=0$, $H$ consists of a chain of curves
with self-intersection $\leq-2$ except maybe its right boundary which
is a $(-1)$-curve, and $\overline{H}$ is the image of $H$ by the
real structure $\sigma$, or to the situation that $C_{0,\mathbb{C}}^{2}=-1$
from which we reach case b) by contracting $C_{0}$. 

2) If $B_{\mathbb{C}}$ is not minimal over $\mathbb{C}$, then the
hypothesis that $B$ is minimal over $\mathbb{R}$ implies that $B_{\mathbb{C}}=E\vartriangleright D_{0}\vartriangleright\overline{D}_{0}\vartriangleright\overline{E}$
where $D_{0}$ and $\overline{D}_{0}$ are irreducible with self-intersection
$-1$ and $E$ is a chain of curves with self-intersection different
from $-1$. Since $B_{\mathbb{C}}$ is the support of an ample divisor
on $V_{\mathbb{C}}$, $E$ cannot be empty. Furthermore, it cannot
contain any irreducible component with nonnegative self-intersection
for otherwise $B_{\mathbb{C}}$ would contain two disjoint such components.
This yields case b'). \end{proof}

\begin{cor}
\label{cor:Cplx-Normal-Form} Let $(V,B)$ be a pair defined over
$\mathbb{R}$ consisting of a smooth projective surface $V$ and a
geometrically rational chain $B$ supporting a effective ample divisor
on $V$. If the irreducible component of $B$ are not all $\mathbb{R}$-rational
then there exists a smooth projective completion $(V_{2},B_{2})$ of $S_{\mathbb{C}}=V_{\mathbb{C}}\setminus B_{\mathbb{C}}$
defined over $\mathbb{C}$ whose boundary $B_{2}$ is a chain of $\mathbb{C}$-rational
curves of the form $B_{2}=F\vartriangleright C\vartriangleright E$
where $F^{2}=0$, $C^{2}=-1$ and where $E$ is either empty, or an
irreducible curve with self-intersection $0$, or a chain of the one
of the following types 

i) $[-e_{1},-e_{2},\cdots,-e_{n},-e_{n},\ldots,-e_{1}]$, where $n\geq1$
and $e_{i}\geq2$ for every $i=1,\ldots,n$.

ii) \textup{$[-e_{1},\ldots,-e_{n-1},-2e_{n},-e_{n-1},\ldots,-e_{1}]$
}where $n\geq2$ and $e_{i}\geq2$ for every $i=1,\ldots,n$.

iii) $[-e_{1},\ldots,-e_{n-1},-2e_{n}+1,-e_{n-1},\ldots,-e_{1}]$
where $n\geq2$ and $e_{i}\geq2$ for every $i=1,\ldots,n$.\end{cor}
\begin{proof}
Let $(V_{1},B_{1})$ be the smooth projective completion $(V_{1},B_{1})$
of $S$ defined over $\mathbb{R}$ with geometrically  rational chain
boundary $B_{1}$ constructed in the previous lemma. In case a) we
reach a chain of three irreducible components with self-intersection
$(0,-1,0)$ by blowing-up the intersection point of $H\vartriangleright\overline{H}$.
In case b) and $E$ is empty, then $B_{1,\mathbb{C}}$ consists of
a pair of irreducible curves $G$ and $\overline{G}$ with self-intersection
$0$ which can be transformed into a pair of curves with self-intersection
$0$ and $-1$ by performing an elementary transformation at their
intersection point. Otherwise, if $E=E_{1}\vartriangleright\cdots\vartriangleright E_{n}$
is not empty, we let $e_{i}=-E_{i}^{2}=-(\overline{E}_{i})^{2}\geq2$
for every $i=1,\ldots,n$. We desired smooth projective completion of $S_{\mathbb{C}}$
is obtained from $(V_{1,\mathbb{C}},B_{1,\mathbb{C}})$ by performing
the following sequence of birational transformations with centers
and exceptional curves all supported on the successive total transforms
of $B_{1,\mathbb{C}}$. We first blow-up the point $G\cap\overline{G}$
and contract the proper transform of $G$. The self-intersection of
$E_{n}$ increased by $1$, the self-intersection of $\overline{G}$
decreased by one and the proper transform of the exceptional divisor
of the blow-up has self-intersection $0$. We repeat the same operation
again $e_{n}-2$ times until the proper transform of $E_{n}$ has
self-intersection $-1$. Then we blow-up again the intersection point
of the proper transform of the last exceptional divisor $E$ with
$\overline{G}$ to get a chain of the form $E_{1}\vartriangleright E_{2}\vartriangleright\cdots\vartriangleright E_{n}\vartriangleright E\vartriangleright E'\vartriangleright\overline{G}\vartriangleright\overline{E}_{n}\vartriangleright\cdots\vartriangleright\overline{E}_{1}$
with $E_{n}^{2}=E^{2}=-1$ and $\overline{G}^{2}=-e_{n}$. Then we
contract $E_{n}$ to get a chain of the form $E_{1}\vartriangleright E_{2}\vartriangleright\cdots\vartriangleright E_{n-1}\vartriangleright E\vartriangleright E'\vartriangleright\overline{G}\vartriangleright\overline{E}_{n}\cup\cdots\vartriangleright\overline{E}_{1}$
where $E_{n-1}^{2}=-e_{n-1}+1$, $E^{2}=0$ and $(E')^{2}=-1$. We
continue by induction until we reach a chain of the form $E_{1}\vartriangleright E\vartriangleright E'\vartriangleright\tilde{T}$
where $E_{1}^{2}=E^{2}=-1$, $(E')^{2}=-e_{1}$ and $\tilde{T}$ is
a chain of irreducible curves of type $[-e_{2},\cdots,-e_{n},-e_{n},\ldots,-e_{1}]$.
By contracting $E_{1}$, we eventually reach the desired smooth completion
with boundary chain of type i). 

The remaining two cases, corresponding respectively to smooth projective
completions $(V_{1},B_{1})$ of the form c) and b') in Lemma \ref{lem:Chains-normal-forms},
follow from similar arguments. We leave the detail to the reader. \end{proof}

\begin{example}
\label{Ex:Q-acyclic-boundary-chains} ($\mathbb{Q}$-acyclic surfaces
completable by a chain of rational curves). 
Let $S$ be smooth complex $\mathbb{Q}$-acyclic surface non isomorphic to $\mathbb{A}^2_{\mathbb{C}}$, with an $\mathbb{A}^{1}$-fibration $\pi:S\rightarrow\mathbb{A}_{\mathbb{C}}^{1}$ and admitting a smooth projective completion $(V,B)$ as in Lemma \ref{lem:A1-fibComp} for which $B=F_{\infty} \vartriangleright \overline{C}_0 \vartriangleright E$ is a chain, where $\infty=\mathbb{P}^{1}_{\mathbb{C}}\setminus \mathbb{A}^{1}_{\mathbb{C}}$. Since $S\not\simeq \mathbb{A}^2_{\mathbb{C}}$, $E$ is not empty. Letting $q=\overline{\pi}(E)$, it follows from Lemma \ref{lem:A1-fibComp} and Proposition \ref{prop:NegKod-FiberStruct} that $\pi^{-1}(q)$ is the unique degenerate fiber of $\pi$. Its closure in $V$ is irreducible, of multiplicity $\mu\geq2$ as a component of $\overline{\pi}^{*}(q)$
and is the unique $(-1)$-curve contained in $\overline{\pi}^{-1}(q)$.
It follows that $V$ is obtained from the Hirzebruch surface $\rho_{n}:\mathbb{F}_{n}\rightarrow\mathbb{P}_{\mathbb{C}}^{1}$, where $n=\overline{C}^2_0$, by a sequence of blow-ups $\tau:V\rightarrow \mathbb{F}_{n}$ of the following type: the
first step is the blow-up of a point of $F_{q}=\rho_{n}^{-1}(q)$ distinct from its
intersection with the exceptional section of $\rho_{n}$, say with
exceptional divisor $E_{0}$. The second step consists of a subdivisional
expansion with center at the point $p=(F_{q}\cap E_{0})$ with last
exceptional divisor $A_{0}(p)$ and multiplicities $(\mu,\nu)$ for
a certain integer $\nu\geq1$. The final step consists of the blow-up
of a simple point $r$ of the total transform of $F_{q}$ supported on
$A_{0}(p)$, with exceptional divisor $D$. We then have $E=\tau^{-1}(F_{q})\setminus D$. 
\begin{figure}[!htb]
\input{logkod-infty-chain-f.tex} 
\caption{Boundary chain of a $\mathbb{Q}$-acyclic surface with negative Kodaira dimension}  
\label{fig:negKod-chains-boundary}
\end{figure}
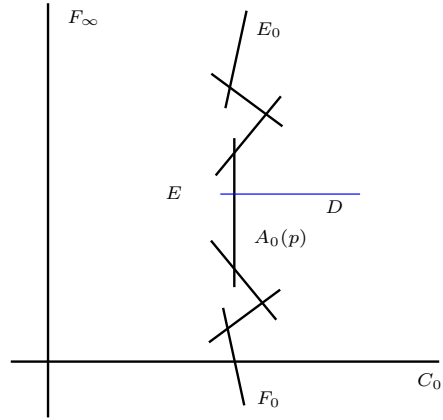

\noindent Note in particular that the left and right boundary curves
of the chain $E$ have distinct self-intersections, except if $E$
consists of three irreducible components with self-intersection $-2$.
In this case, the contraction of the chain $\overline{C}_{0}\vartriangleright E$
defines a birational morphism $\beta:V\rightarrow\mathbb{P}_{\mathbb{C}}^{2}$
which maps $S$ isomorphically onto the complement of the image $\beta_{*}(F_{\infty})$
of $F_{\infty}$, which is a smooth conic. 

In the case where $q$ is a real point of $\mathbb{P}^1_{\mathbb{C}}$ equipped with the standard real structure and $r$ is a real point of $A_{0}(p)$, the pair $(V,B)$ is defined over $\mathbb{R}$ and so is the $\mathbb{A}^1$-fibration $\pi:S\rightarrow \mathbb{A}^1_{\mathbb{C}}$. By virtue of Proposition \ref{prop:NegKod-FiberStruct}, $S$ is a then a $\mathbb{Q}$-acyclic euclidean plane if and only if $\mu$ is odd. 
\end{example}

\subsubsection{Existence of real $\mathbb{A}^1$-fibrations}\label{par:A1-Fib-FakePlane}
\noindent The next proposition completes the proof of Theorem \ref{thm:Q-acyclic-Neg-Desc}. 
\begin{prop}
\label{prop:A1-fib-existence} Let $S$ be a smooth geometrically
integral surface defined over $\mathbb{R}$ of negative Kodaira dimension
such that $S_{\mathbb{C}}(\mathbb{C})$ is $\mathbb{Q}$-acyclic and
$S(\mathbb{R})$ is non compact. Then $S$ admits an $\mathbb{A}^{1}$-fibration
$\pi:S\rightarrow\mathbb{A}_{\mathbb{R}}^{1}$ defined over $\mathbb{R}$. \end{prop}
\begin{proof}
By \cite[Chapter 4, Theorem 4.3.1]{MiyBook} $S_{\mathbb{C}}$ admits an $\mathbb{A}^{1}$-fibration
$q:S_{\mathbb{C}}\rightarrow \mathbb{A}^1_{\mathbb{C}}$.  If $q$ is the
unique such $\mathbb{A}^{1}$-fibration on $S_{\mathbb{C}}$ up to
composition by automorphisms of $\mathbb{A}^1_{\mathbb{C}}$, it must be the complexification of a morphism 
$\pi:S\rightarrow C$ defined over $\mathbb{R}$. Since the affine line does not have nontrivial forms over fields of
characteristic zero, we conclude that $C=\mathbb{A}^1_{\mathbb{R}}$ and that the generic fiber of $\pi$ is isomorphic to the affine line
over the function field of $C$. So $\pi:S\rightarrow C\simeq\mathbb{A}_{\mathbb{R}}^{1}$
is the desired $\mathbb{A}^{1}$-fibration defined over $\mathbb{R}$.

So it remains to consider the case where $S_{\mathbb{C}}$ admits
at least two $\mathbb{A}^{1}$-fibrations over $\mathbb{A}_{\mathbb{C}}^{1}$  
with distinct general fibers. By virtue of \cite[Theorem 2.4]{Du05} this holds
if and only if $S_{\mathbb{C}}$ is not isomorphic to $\mathbb{A}_{\mathbb{C}}^{1}\times(\mathbb{A}_{\mathbb{C}}^{1}\setminus\{0\})$
and admits a smooth projective completion whose boundary divisor consists
of a chain of smooth proper rational curves. Furthermore, the boundary
in any other smooth projective completion of $S_{\mathbb{C}}$ with SNC-minimal boundary is
then again a chain of rational curves. Now let $(V,B)$ be a smooth
projective completion of $S$ defined over $\mathbb{R}$ and such that $B$
is a geometrically rational tree minimal over $\mathbb{R}$. The tree
$B_{\mathbb{C}}$ is then minimal over $\mathbb{C}$ unless it contains
pairs of non-disjoint $(-1)$-curves exchanged by the real structure
$\sigma$, in which case a minimal smooth projective completion $(V_{0},B_{0})$
over $\mathbb{C}$ is obtained from $(V_{\mathbb{C}},B_{\mathbb{C}})$
by blowing-down all possible successive $(-1)$-curves in $B_{\mathbb{C}}$.
Since $B_{0}$ must be a chain, it follows that $B_{\mathbb{C}}$
is a chain too, and so, $B$ is a geometrically  rational chain. Indeed,
if $B_{\mathbb{C}}$ contains an irreducible component $R$ intersecting
at least three other components, then because $B_{0}$ is a chain,
at least one the connected component of $B_{\mathbb{C}}\setminus R$,
say $D$, is contracted to a smooth point by the above sequence of
blow-downs. So it must contain at least a $(-1)$-curve $E$ intersecting
at most two other irreducible components of $B_{\mathbb{C}}$, one
of which being, by the minimality assumption, its image $\overline{E}$
by the real structure $\sigma$. But then $D$ contains a pair of
non-disjoint $(-1)$-curves, contradicting the negative definiteness
of its intersection matrix. 

a) If $B$ consists of $\mathbb{R}$-rational components only then
by virtue of 1) in Lemma \ref{lem:Chains-normal-forms}, there
exists a smooth projective completion $(V',B')$ of $S$ defined over $\mathbb{R}$
whose boundary $B'$ is a chain of $\mathbb{R}$-rational curves of
the form $B'=F\vartriangleright C\vartriangleright T$ where $F^{2}=0$,
$C^{2}=-1$ and $T$ is either empty, or an $\mathbb{R}$-rational
curve with self-intersection $0$, or a chain of $\mathbb{R}$-rational
curves with self-intersections $\leq-2$. Since $V'$ is $\mathbb{R}$-rational,
the complete linear system $|F|$ defines a $\mathbb{P}^{1}$-fibration
$V'\rightarrow\mathbb{P}_{\mathbb{R}}^{1}$ having $F$ as a fiber
and $C$ as a section and whose restriction to $S$ is an $\mathbb{A}^{1}$-fibration
$\pi:S\rightarrow\mathbb{A}_{\mathbb{R}}^{1}$ if $T$ does not consists
of a unique curve, or an $\mathbb{A}^{1}$-fibration $\pi:S\rightarrow\mathbb{A}_{\mathbb{R}}^{1}\setminus\{0\}$
otherwise, the second case being excluded by the fact that $S_{\mathbb{C}}(\mathbb{C})$
is $\mathbb{Q}$-acyclic by hypothesis. 

b) If $B$ contains at least one non $\mathbb{R}$-rational component
then by virtue of Corollary \ref{cor:Cplx-Normal-Form}, there exists
a smooth projective completion $(V',B')$ of $S_{\mathbb{C}}$ whose boundary
$B'$ is a chain of $\mathbb{C}$-rational curves $B'=F\vartriangleright C\vartriangleright E$
where $F^{2}=0$, $C^{2}=-1$ and where $E$ is either empty, or an
irreducible curve with self-intersection $0$ or a chain of curves
with negative self-intersections listed in the corollary. In the first
two cases, it follows that $S_{\mathbb{C}}$ is either isomorphic
to $\mathbb{A}_{\mathbb{C}}^{2}$, in which case $S\simeq\mathbb{A}_{\mathbb{R}}^{2}$
and we are done, or it admits an $\mathbb{A}^{1}$-fibration over $\mathbb{A}_{\mathbb{C}}^{1}\setminus\{0\}$.
But the second possibility is again excluded by the hypothesis that
$S_{\mathbb{C}}(\mathbb{C})$ is $\mathbb{Q}$-acyclic. In the remaining
cases, it follows from the description given in Corollary \ref{cor:Cplx-Normal-Form}
that the sequences of self-intersections of the irreducible components
of $E$ are all symmetric. By virtue of \cite[Corollary 2]{Gi-Da1}
(see also \cite[Corollary 3.2.3]{BD11}), such sequences up to reversion
of the ordering are invariants of the isomorphism type of $S_{\mathbb{C}}$.
Comparing with the sequences obtained in Example \ref{Ex:Q-acyclic-boundary-chains}
above for $\mathbb{Q}$-acyclic complex surfaces admitting a smooth
projective completion whose boundary is a chain, we conclude that the only
possibility is that $E$ consists of a chain of three curves with
self-intersection $-2$, whence that $S_{\mathbb{C}}$ is isomorphic
to the complement of a smooth conic in $\mathbb{P}_{\mathbb{C}}^{2}$.
This implies in turn that $S$ is isomorphic to the complement of
a smooth conic $D$ in $\mathbb{P}_{\mathbb{R}}^{2}$, and since $S(\mathbb{R})$
is not compact, $D$ has an $\mathbb{R}$-rational point $p$. It
follows that $S$ admits an $\mathbb{A}^{1}$-fibration $\pi:S\rightarrow\mathbb{A}_{\mathbb{R}}^{1}$
defined over $\mathbb{R}$, induced by the restriction of pencil of
conics osculating $D$ at $p$, i.e the pencil generated by $D$ and
twice its tangent line at $p$. 
\end{proof}

\begin{rem}
The proof of Proposition \ref{prop:A1-fib-existence} shows more generally
that a smooth affine geometrically integral surface $S$ defined over
$\mathbb{R}$ and of negative Kodaira dimension admits an $\mathbb{A}^{1}$-fibration
$\pi:S\rightarrow\mathbb{A}_{\mathbb{R}}^{1}$ defined over $\mathbb{R}$
provided that $S_{\mathbb{C}}$ does not admit any smooth projective completion
$(V,B)$ not defined over $\mathbb{R}$ and whose boundary is a chain
of rational curves of the form $B=F\vartriangleright C\vartriangleright E$
where $F^{2}=0$, $C^{2}=-1$ and $E$ is a nonempty ordered chain
of rational curves whose type $[-a_{1},\ldots,-a_{n}]$, $a_{i}\geq2$,
is symmetric, in the sense that the sequences $(-a_{1},\ldots,-a_{n})$
and $(-a_{n},\ldots,-a_{1})$ are equal. Chains with this symmetry
property are called palindromes in \cite{BD11}. 

The $\mathbb{Q}$-acyclicity of $S_{\mathbb{C}}(\mathbb{C})$ and the non compactness of $S(\mathbb{R})$
play a crucial in the characterization obtained in this proposition.
Indeed, as already observed in the proof of Proposition \ref{prop:A1-fib-existence},
the complexification of the complement $S$ of a smooth conic $D\subset\mathbb{P}_{\mathbb{R}}^{2}$
without $\mathbb{R}$-rational point admits a smooth projective completion
$(V,B)$ not defined over $\mathbb{R}$ whose boundary $B$ has the
form $B=F\vartriangleright C\vartriangleright E$ where $E$ is a
palindrome of type $[-2,-2,-2]$, and this surface $S$ does not admit
any $\mathbb{A}^{1}$-fibration defined over $\mathbb{R}$. 

One can show along the same lines that there even exists smooth surfaces
$S$ of negative Kodaira dimension with $S(\mathbb{R})\approx\mathbb{R}^{2}$
but $S_{\mathbb{C}}(\mathbb{C})$ not $\mathbb{Q}$-acyclic which
do not admit any $\mathbb{A}^{1}$-fibration $\pi:S\rightarrow\mathbb{A}_{\mathbb{R}}^{1}$
defined over $\mathbb{R}$. We leave to the reader to check that this
holds for instance for the nontrivial real form $S=\{x^{2}+y^{2}=z^{3}-1\}\subset\mathbb{A}_{\mathbb{R}}^{3}$
of the surface $S'=\{uv=z^{3}-1\}\subset\mathbb{A}_{\mathbb{R}}^{3}$,
whose complexification has $H_{2}(S_{\mathbb{C}}(\mathbb{C});\mathbb{Z})\simeq\mathbb{Z}^{2}$.
The real loci of $S$ and $S'$ are both homeomorphic to $\mathbb{R}^{2}$
via the maps $\mathbb{R}^{2}\rightarrow S(\mathbb{R})$, $(x,y)\mapsto(x,y,\sqrt[3]{x^{2}+y^{2}+1})$
and $\mathbb{R}^{2}\rightarrow S'(\mathbb{R})$, $(u,v)\mapsto(u,v,\sqrt[3]{uv+1})$
respectively. The surface $S'$ admits an $\mathbb{A}^{1}$-fibration
$\pi'=\mathrm{pr}_{u}\mid_{S'}:S'\rightarrow\mathbb{A}_{\mathbb{R}}^{1}$
defined over $\mathbb{R}$ whose unique degenerate fiber $(\pi')^{-1}(0)$
consists of the disjoint union of the curves $\mathbb{A}_{\mathbb{R}}^{1}=\mathrm{Spec}(\mathbb{R}[v])$
and $\mathbb{A}_{\mathbb{C}}^{1}=\mathrm{Spec}(\mathbb{R}[z]/((z^{2}+z+1)[v])$.
So $\kappa(S)=\kappa(S_{\mathbb{C}})=\kappa(S')=-\infty$. Furthermore,
$S'$ admits a smooth projective completion $(V',B')$ defined over $\mathbb{R}$
whose boundary $B$ is a chain of three $\mathbb{R}$-rational curves
$F\vartriangleright C\vartriangleright E$ as above where $E$ is
palindrome consisting of a unique curve with self-intersection $-3$
(see e.g. \cite[§ 5.4]{BD11}). 
\end{rem}

\subsection{$\mathbb{R}$-biregular birational rectification of $\mathbb{Q}$-acyclic
euclidean planes of negative Kodaira dimension}

In this subsection, we consider the question of classification of
$\mathbb{Q}$-acyclic euclidean planes of negative Kodaira dimension
up to $\mathbb{R}$-biregular birational equivalence. We say for short
that such an euclidean plane $S$ is $\mathbb{R}$-\emph{biregularly
birationally
 rectifiable} if there exists an $\mathbb{R}$-biregular
birational map $\varphi:S\dashrightarrow\mathbb{A}_{\mathbb{R}}^{2}$,
i.e. a birational map defined over $\mathbb{R}$, containing the real
locus of $S$ in its domain of definition and inducing a diffeomorphism
$\varphi(\mathbb{R}):S(\mathbb{R})\stackrel{\simeq}{\rightarrow}\mathbb{R}^{2}=\mathbb{A}_{\mathbb{R}}^{2}(\mathbb{R})$.
The following theorem implies in particular that a large class of
$\mathbb{Q}$-acyclic euclidean planes of negative Kodaira are indeed
$\mathbb{R}$-biregularly birationally
 rectifiable.
\begin{thm}
\label{thm:Bir-rectif} Let $S$ be a smooth affine geometrically
integral surface defined over $\mathbb{R}$ with an $\mathbb{A}^{1}$-fibration
$\pi:S\rightarrow\mathbb{A}_{\mathbb{R}}^{1}$. Suppose that $S(\mathbb{R})\approx\mathbb{R}^{2}$
and that all but at most one fibers of $\pi$ over $\mathbb{R}$-rational
points of $\mathbb{A}_{\mathbb{R}}^{1}$ contain a reduced $\mathbb{R}$-rational
irreducible component. Then $S$ is $\mathbb{R}$-biregularly birationally rectifiable.
\end{thm}

The assumptions of Theorem \ref{thm:Bir-rectif} being satisfied by any $\mathbb{Q}$-acyclic euclidean plane $S$ admitting an $\mathbb{A}^{1}$-fibration $\pi:S\rightarrow\mathbb{A}_{\mathbb{R}}^{1}$ with at most one degenereate, we obtain:
 
\begin{cor} Every $\mathbb{Q}$-acyclic euclidean plane $S$ with $S(\mathbb{R})\approx\mathbb{R}^{2}$ admitting
an $\mathbb{A}^{1}$-fibration $\pi:S\rightarrow\mathbb{A}_{\mathbb{R}}^{1}$ with at most one degenerate fiber is $\mathbb{R}$-biregularly birationally rectifiable.
\end{cor}

The proof of Theorem \ref{thm:Bir-rectif} consists of two steps,
given in $\S$ \ref{sub:Standard--models}-\ref{sub:Proof-of-Theorem}
below. We first reduce via suitable $\mathbb{R}$-biregular birational
maps to the case of surfaces $S$ equipped with an $\mathbb{A}_{\mathbb{R}}^{1}$-fibration
$\pi:S\rightarrow\mathbb{A}_{\mathbb{R}}^{1}$ defined over $\mathbb{R}$
with irreducible fibers and at most one degenerate $\mathbb{R}$-rational
fiber. Then we show by induction on the number of irreducible components
in the boundary $B$ of a smooth projective SNC-minimal completion 
$(V,B)$ of $S$ 
defined over $\mathbb{R}$ that every such surface is $\mathbb{R}$-biregularly
birationally rectifiable. 

\subsubsection{\label{sub:Standard--models} Standard $r$-models}

The simplest surfaces $\pi:S\rightarrow\mathbb{A}_{\mathbb{R}}^{1}$
satisfying the hypotheses of Theorem \ref{thm:Bir-rectif} are those
for which $\pi:S\rightarrow\mathbb{A}_{\mathbb{R}}^{1}$ restricts
to a trivial $\mathbb{A}^{1}$-bundle over $\mathbb{A}_{\mathbb{R}}^{1}\setminus\{0\}$
and $\pi^{*}(\{0\})$ is geometrically irreducible, of odd multiplicity
$m\geq1$. Indeed, the fact that $S(\mathbb{R})\approx\mathbb{R}^{2}$
is then guaranteed by 2) in Proposition \ref{prop:NegKod-FiberStruct}. 

\begin{parn}  \label{par:Standard-Comp} When specialized to such
surfaces, the general description given in $\S$ \ref{par:A1-fib-completion}
provides a smooth projective completion $(V,B)$ of $S$ defined
over $\mathbb{R}$ into a surface obtained from $\rho_{1}:\mathbb{F}_{1}\rightarrow\mathbb{P}_{\mathbb{R}}^{1}$
with exceptional section $C_{0}\simeq\mathbb{P}_{\mathbb{R}}^{1}$
and a pair of fixed $\mathbb{R}$-rational fibers $E_{-1}=\rho_{1}^{-1}(\{0\})$
and $F_{\infty}=\rho_{1}^{-1}(\mathbb{P}_{\mathbb{R}}^{1}\setminus\mathbb{A}_{\mathbb{R}}^{1})$
via a birational morphism $\tau:V\rightarrow\mathbb{P}_{\mathbb{R}}^{1}$
defined over $\mathbb{R}$ of the following form:

- If $m=1$ then $\pi:S\rightarrow\mathbb{A}_{\mathbb{R}}^{1}$ is
isomorphic to the trivial $\mathbb{A}^{1}$-bundle $\mathrm{pr}_{1}:S\simeq\mathbb{A}_{\mathbb{R}}^{2}\rightarrow\mathbb{A}_{\mathbb{R}}^{1}$,
and we have an isomorphism $(V,B)=(\mathbb{F}_{1},F_{\infty}\cup C_{0})$. 

- Otherwise, if $m\geq2$ then $\tau=\tau_{0}\circ\cdots\circ\tau_{n}$
is a sequence of blow-ups of $\mathbb{R}$-rational points, starting
with the blow-up $\tau_{0}:V_{1}\rightarrow V_{0}=\mathbb{F}_{1}$
of a point $p_{0}\in E_{-1}\setminus C_{0}$, say with exceptional
divisor $E_{0}$, followed by the blow-up $\tau_{1}:V_{2}\rightarrow V_{1}$
of the intersection point $p_{1}$ of $E_{0}$ with the proper transform
of $E_{-1}$, with exceptional divisor $E_{1}$, and continuing with
a sequence of blow-ups $\tau_{i}:V_{i+1}\rightarrow V_{i}$ of $\mathbb{R}$-rational
points $p_{i}\in E_{i-1}$ with exceptional divisor $E_{i}$. The
last step $\tau_{n}:V=V_{n+1}\rightarrow V_{n}$ is the blow-up of
an $\mathbb{R}$-rational point $p_{n}\in E_{n-1}$ with exceptional
divisor $A_{0}$. The surface $S$ is then isomorphic to the complement
in $V$ of the SNC divisor $B=F_{\infty}\cup C_{0}\cup E$, where
$E=\bigcup_{i=0}^{n}E_{i-1}$ is a tree of $\mathbb{R}$-rational
curves, the $\mathbb{A}^{1}$-fibration $\pi:S\rightarrow\mathbb{A}_{\mathbb{R}}^{1}$
coincides with the restriction to $S$ of the $\mathbb{P}^{1}$-fibration
$\overline{\pi}:\rho_{1}\circ\tau:V\rightarrow\mathbb{P}_{\mathbb{R}}^{1}$
and $\pi^{-1}(\{0\})=A_{0}\cap S$. Note that by construction $E_{i}^{2}\leq-2$
for every $i=-1,\ldots,n-1$. 

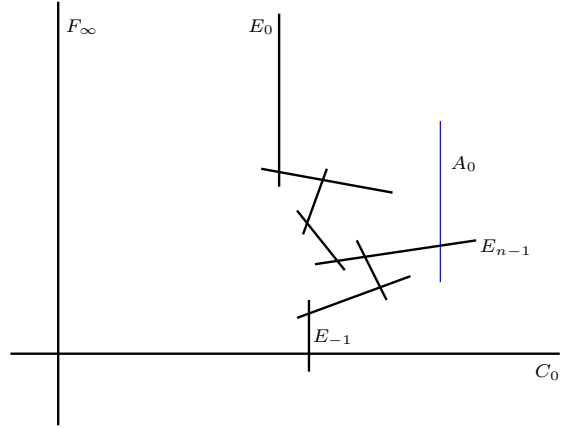
\begin{figure}[!htb]
\input{rectif-0-f.tex} 
\caption{Structure of the divisor $B\cup A_0$}  
\label{fig:rectif-0}
\end{figure}

Since $m\geq2$ is odd, the sequence $\tau_{i}$ is not completely
arbitrary: for instance, the branching components
of the tree $E$, i.e. the irreducible components of $E$ intersecting
at least two other irreducible components, must have odd multiplicities
as irreducible components of the degenerate fiber $\overline{\pi}^{*}(\{0\})$
of the $\mathbb{P}^{1}$-fibration $\overline{\pi}:V\rightarrow\mathbb{P}_{\mathbb{R}}^{1}$.
More precisely, we have the following description 

\end{parn}
\begin{lem}
\label{lem:Standard-Comp-Chain} For a pair $(V,B=F_{\infty}\cup C_{0}\cup E)$
as in $\S$ \ref{par:Standard-Comp}, the following holds: 

1) Every branching component of $E$ has odd multiplicity as an irreducible
component of the degenerate fiber $\overline{\pi}^{*}(\{0\})$. 

2) Let $L\subset C_{0}\cup E$ be the unique minimal subchain containing
$C_{0}$ and $E_{0}$ and let $E_{i}$ be the unique branching component
of $E$ contained in $L$. If $E_{i}\cap E_{0}\neq\emptyset$ then
$i=2p$ for some $p\geq1$ and $L=C_{0}\vartriangleright E_{-1}\vartriangleright\cdots\vartriangleright E_{2p}\vartriangleright E_{0}$
is a chain of type $(-1,-2,\ldots,-2,E_{2p}^{2},-(2p+1))$. \end{lem}
\begin{proof}
If $E$ has a branching component with even multiplicity, say $E_{i}$
for some $i\geq2$, then since $\tau_{i+1}:V_{i+2}\rightarrow V_{i+1}$
necessarily consists of the blow-up of a simple point of $E_{-1}\cup\cdots\cup E_{i}$
supported on $E_{i}$, the proper transform in $V$ of its exceptional
divisor $E_{i+1}$ has the same multiplicity as $E_{i}$ in $\overline{\pi}^{*}(\{0\})$.
Since the center of the next blow-up $\tau_{i+2}$ is either the point
$E_{i}\cap E_{i+1}$ or a simple point supported on $E_{i+1}$, it
follows by induction that the proper transform in $V$ of every divisor
$E_{j}$, $j\geq i$ arises with even multiplicity in $\overline{\pi}^{*}(\{0\})$.
As a consequence, $A_{0}$ would have even multiplicity as an irreducible
component of $\overline{\pi}^{*}(\{0\})$, in contradiction with the
fact that $A_{0}\cap S$ has odd multiplicity $m$ as a scheme theoretic
fiber of $\pi$. The second assertion follows immediately from the
observation that if $E_{i}\cap E_{0}\neq\emptyset$ then the multiplicity
of $E_{i}$ as a component of $\overline{\pi}^{*}(\{0\})$ is equal
to $-E_{0}^{2}\geq2$. \end{proof}
\begin{defn}
\label{def:Standar-r-pair} An $r$\emph{-standard} $\mathbb{A}^{1}$-\emph{fibered
surface} is a smooth geometrically integral affine surface $S$ with
an $\mathbb{A}^{1}$-fibration $\pi:S\rightarrow\mathbb{A}_{\mathbb{R}}^{1}$
defined over $\mathbb{R}$, restricting to a trivial $\mathbb{A}^{1}$-bundle
over $\mathbb{A}_{\mathbb{R}}^{1}\setminus\{0\}$ and such $\pi^{-1}(\{0\})$
is geometrically irreducible, of odd multiplicity $m\geq1$\emph{.}
An \emph{$r$-standard pair} is a pair $(V,B)$ consisting of a smooth
geometrically integral projective surface $V$ and a geometrically
rational tree $B$ both defined over $\mathbb{R}$ isomorphic to the
completion of an $r$\emph{-standard} $\mathbb{A}^{1}$-\emph{fibered
surface }constructed in $\S$ \ref{par:Standard-Comp}. 
\end{defn}
The next proposition reduces the study of the $\mathbb{R}$-biregular
birational rectifiability of surfaces considered in Theorem \ref{thm:Bir-rectif}
to the case of $r$-standard surfaces:
\begin{prop}
\label{prop:R-standard-reduction} Let $S$ be a smooth affine geometrically
integral surface defined over $\mathbb{R}$ with an $\mathbb{A}^{1}$-fibration
$\pi:S\rightarrow\mathbb{A}_{\mathbb{R}}^{1}$. Suppose that $S(\mathbb{R})\approx\mathbb{R}^{2}$
and that all but at most one fiber of $\pi$ over $\mathbb{R}$-rational
points of $\mathbb{A}_{\mathbb{R}}^{1}$ contain a reduced $\mathbb{R}$-rational
irreducible component. Then there exists an $r$-standard affine $\mathbb{A}^{1}$-fibered
surface $\pi_{0}:S_{0}\rightarrow\mathbb{A}_{\mathbb{R}}^{1}$ and
an $\mathbb{R}$-biregular birational map $\varphi:S\dashrightarrow S_{0}$
such that the following diagram commutes 
\begin{eqnarray*}
S & \stackrel{\varphi}{\dashrightarrow} & S_{0}\\
\pi\downarrow &  & \downarrow\pi_{0}\\
\mathbb{A}_{\mathbb{R}}^{1} & \stackrel{\simeq}{\rightarrow} & \mathbb{A}_{\mathbb{R}}^{1}.
\end{eqnarray*}
 \end{prop}
\begin{proof}
The strategy is of course to eliminate on the one hand all degenerate
fibers of $\pi$ over $\mathbb{C}$-rational points of $\mathbb{A}_{\mathbb{R}}^{1}$
and on the other hand all non $\mathbb{R}$-rational irreducible components
of the degenerate fibers of $\pi$ over $\mathbb{R}$-rational points.\textbf{
}Since $S(\mathbb{R})\approx\mathbb{R}^{2}$ it follows from 2) in
Proposition \ref{prop:NegKod-FiberStruct} that for every $\mathbb{R}$-rational
point $p\in\mathbb{A}_{\mathbb{R}}^{1}$, the fiber $\pi^{*}(p)$
has the form $mR+R'$, where $R\simeq\mathbb{A}_{\mathbb{R}}^{1}$,
$m\geq1$ is odd, and $R'$ is an effective divisor whose support
is disjoint from $R$ and consists of a disjoint union of affine lines
defined over $\mathbb{C}$. The complement $S'$ in $S$ of all non
$\mathbb{R}$-rational irreducible components of $\pi^{-1}(p)$ where
$p$ runs through the finitely many $\mathbb{R}$-rational points
of $\mathbb{A}_{\mathbb{R}}^{1}$ over which the fiber of $\pi$ is
degenerate is an affine open subset of $S$ on which $\pi$ restricts
to an $\mathbb{A}^{1}$-fibration $\pi':S'\rightarrow\mathbb{A}_{\mathbb{R}}^{1}$
whose fibers over $\mathbb{R}$-rational point of $\mathbb{A}_{\mathbb{R}}^{1}$
are all isomorphic to $\mathbb{A}_{\mathbb{R}}^{1}$ when equipped
with their reduced structure. Furthermore, the hypotheses imply that
there exists at most one $\mathbb{R}$-rational point of $\mathbb{A}_{\mathbb{R}}^{1}$
over which the scheme theoretic fiber of $\pi'$, say $(\pi')^{*}(\{0\})$
is degenerate. By construction, the inclusion $S'\hookrightarrow S$
is an $\mathbb{R}$-biregular birational map. Now let $(V',B')$
be a smooth projective completion of $S'$ defined over $\mathbb{R}$ with geometrically
rational boundary tree $$B'=F'_{\infty}\cup\overline{C}_{0}'\cup\bigcup_{p\in\mathbb{A}_{\mathbb{R}}^{1}}H_{p}'$$
as in $\S$\ref{par:A1-fib-completion}, where $F_{\infty}\simeq\mathbb{P}_{\mathbb{R}}^{1}$
and $\overline{C}_{0}\simeq\mathbb{P}_{\mathbb{R}}^{1}$ are respectively
the fiber over the $\mathbb{R}$-rational closed point $\infty=\mathbb{P}_{\mathbb{R}}^{1}\setminus\mathbb{A}_{\mathbb{R}}^{1}$
and a section of the $\mathbb{P}^{1}$-fibration $\overline{\pi}':V'\rightarrow\mathbb{P}_{\mathbb{R}}^{1}$
extending $\pi'$. Since $(\pi')^{-1}(\{0\})$ is the unique possibly
degenerate fiber of $\pi'$ over an $\mathbb{R}$-rational point of
$\mathbb{A}_{\mathbb{R}}^{1}$, the divisor $\bigcup_{p\in\mathbb{A}_{\mathbb{R}}^{1}}H_{p}$
can be decomposed into the disjoint $\bigcup H_{0}\bigsqcup\bigcup_{q\in\mathbb{A}_{\mathbb{R}}^{1}(\mathbb{C})}H_{q}$
where $H_{0}$ is a possibly empty tree of $\mathbb{R}$-rational
curve supported on $(\overline{\pi}')^{-1}(\{0\})$. For every $\mathbb{C}$-rational
point $q$ of $\mathbb{A}_{\mathbb{R}}^{1}$ for which $H_{q}$ is
not empty, equivalently, for every $\mathbb{C}$-rational point $q$
of $\mathbb{A}_{\mathbb{R}}^{1}$ over which the fiber $(\pi')^{*}(q)$
is degenerate, there exists a sequence of contractions $\beta_{q}:V'\rightarrow V_{q}'$
of curves defined over $\mathbb{R}$ supported on $(\overline{\pi}')^{-1}(q)$
such that $\beta_{q}((\overline{\pi}')^{*}(q))\simeq\mathbb{P}_{\mathbb{C}}^{1}$
is a smooth fiber of the $\mathbb{P}^{1}$-fibration $\overline{\pi}'_{q}=\overline{\pi}'\circ\beta_{q}^{-1}:V_{q}'\rightarrow\mathbb{P}_{\mathbb{R}}^{1}$.
Let $V_0$ be the $\mathbb{P}^{1}$-fibered surface obtained from $V'$ by
performing such sequences of contractions for every $\mathbb{C}$-rational
point $q\in\mathbb{A}_{\mathbb{R}}^{1}$ such that $(\pi')^{*}(q)$
is degenerate, let $\beta:V'\rightarrow V_{0}$ be the corresponding
birational morphism and let $B_{0}$ be the image of $F_{\infty}\cup\overline{C}_{0}\cup H_{0}$
by $\beta$. By construction, $\overline{\pi}_{0}^{*}(\{0\})$ is
the unique degenerate fiber of $\overline{\pi}_{0}:V_{0}\rightarrow\mathbb{P}_{\mathbb{R}}^{1}$
and the restriction of $\beta$ to $S$ induces an $\mathbb{R}$-biregular
birational map $S\dashrightarrow S_{0}=V_{0}\setminus B_{0}$ which
commutes with the $\mathbb{A}^{1}$-fibrations $\pi$ and $\pi_{0}$
induced by $\overline{\pi}$ and $\overline{\pi}_{0}$ respectively.
So $\pi_{0}:S_{0}\rightarrow\mathbb{A}^{1}$ is a standard $\mathbb{A}^{1}$-fibered
surface, provided that $S_{0}$ is affine. First note that $S_{0}$
does not contain any complete algebraic curve. Indeed, otherwise such
an irreducible curve $D$ would not intersect $F_{\infty}$ in $V_{0}$,
whence would be contained in a fiber of $\overline{\pi}_{0}$. Since
every fiber of $\overline{\pi}_{0}$ but $\overline{\pi}_{0}^{-1}(0)$
is smooth and $\overline{C}_{0}$ is a section of $\overline{\pi}_{0}$,
it would follow that $D$ is contained in $\overline{\pi}_{0}^{-1}(0)$.
But since $\beta$ restricts to an isomorphism in a neighborhood of
$(\overline{\pi}')^{-1}(0)$, it would follow that $\beta^{-1}(D)$
is a complete curve in $S'$, which is absurd since $S'$ is affine.
It remains to observe that $B_{0}$ is the support of an effective
$\mathbb{Q}$-divisor $\Delta$ on $V_{0}$ whose intersection with
every irreducible component of $B_{0}$ is strictly positive: since
$F_{0}^{2}=0$ and the dual graph of $B_{0}$ is a tree, such a $\Delta$
is obtained by assigning a positive coefficient $a_{0}\in\mathbb{Q}$
to $F_{0}$ and assigning to the other irreducible components of $B_{0}$
a sequence of positive rational coefficients decreasing rapidly with
respect to the distance to $F_{0}$ in the dual graph of $B_{0}$.
The so constructed $\Delta$ is ample by virtue of the Nakai-Moishezon
criterion and so, $S_{0}$ is affine as desired. 
\end{proof}

\subsubsection{Elementary birational links between $r$-standard pairs. }

Let $(V,B=F_{\infty}\cup C_{0}\cup E)$ be an $r$-standard pair with
non empty tree $E$ and let $\tau=\tau_{0}\circ\cdots\circ\tau_{m}:V\rightarrow\mathbb{F}_{1}$
be the birational morphism constructed in $\S$ \ref{par:Standard-Comp}.
Since the proper base point $p_{0}$ of $\tau^{-1}$ belongs to $E_{-1}\setminus C_{0}$,
the pencil $\mathbb{F}_{1}\dashrightarrow\mathbb{P}_{\mathbb{R}}^{1}$
lifting the projection $\mathbb{P}_{\mathbb{R}}^{2}\dashrightarrow\mathbb{P}_{\mathbb{R}}^{1}$
from $p_{0}$ via the contraction $\mathbb{F}_{1}\rightarrow\mathbb{P}_{\mathbb{R}}^{2}$
of $C_{0}$ lifts to a $\mathbb{P}^{1}$-fibration $\xi_{p_{0}}:V\rightarrow\mathbb{P}_{\mathbb{R}}^{1}$
with a unique degenerate fiber supported by the closure in $V$ of
$(C_{0}\cup E\setminus E_{0})\cup A_{0}$ and having the proper transforms
of $F_{\infty}$ and $E_{0}$ as cross-sections. The restriction of
$\xi_{p_{0}}$ to $S=V\setminus B$ is thus a surjective fibration
$S\rightarrow\mathbb{P}_{\mathbb{R}}^{1}$ defined over $\mathbb{R}$
whose generic fiber is isomorphic to the $1$-punctured affine line
over the function field of $\mathbb{P}_{\mathbb{R}}^{1}$. 
\begin{defn}
\label{def:Elementary-links} With the notation above, we call \emph{elementary
links} the birational transformations of pairs $\eta:(V,B)\dashrightarrow(V^{(1)},B^{(1)})$
defined as follows: : 

a) If $E_{0}^{2}=-2s$ is even, we choose $s$ distinct smooth fibers
$\ell_{i}\simeq\mathbb{P}_{\mathbb{C}}^{1}$ of $\xi_{p_{0}}:V\rightarrow\mathbb{P}_{\mathbb{R}}^{1}$
over $\mathbb{C}$-rational points of $\mathbb{P}_{\mathbb{R}}^{1}$
and we let $\eta':V\dashrightarrow V'$ be the birational map defined
over $\mathbb{R}$ consisting of the blow-up of the $\mathbb{C}$-rational
points $\ell_{i}\cap F_{\infty}$, with respective exceptional divisors
$\tilde{\ell}_{i}\simeq\mathbb{P}_{\mathbb{C}}^{1}$, followed by
the contraction of the proper transforms of the $\ell_{i}$, $i=1,\ldots,s$.
The proper transforms of $E_{0}$ and $F_{\infty}$ in $V'$ have
self-intersections $0$ and $-2s$ respectively, while the self-intersections
of the remaining irreducible components of $B$ are left unchanged. 

\begin{figure}[!htb]
\input{rectif-a-f.tex} 
\caption{Elementary link: even case}  
\label{fig:Even-rectif}
\end{figure}
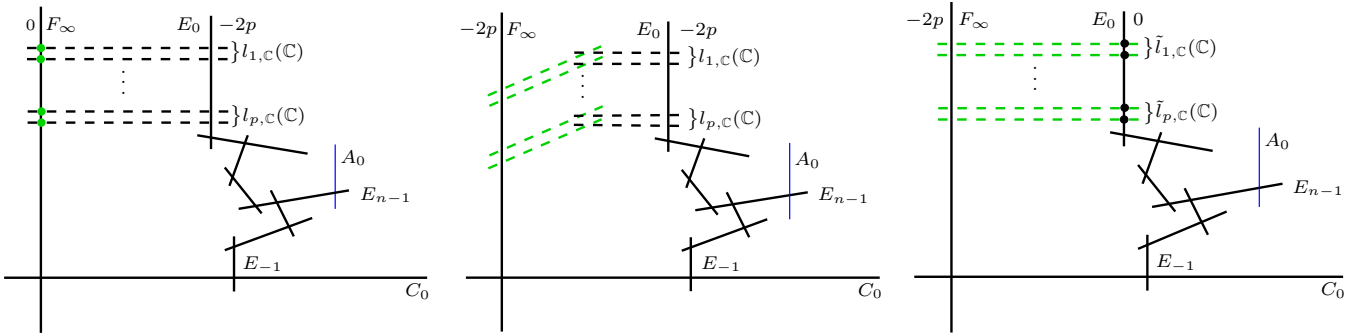

Since $V'$ is rational, the complete linear system $|E_{0}|$ generates
a $\mathbb{P}^{1}$-fibration $\overline{\pi}':V'\rightarrow\mathbb{P}_{\mathbb{R}}^{1}$
having the unique irreducible component $E_{i_{1}}$ of $E$ intersecting
$E_{0}$ as a section. Furthermore, since $E_{0}^{2}$ is even, the
multiplicity of $E_{i_{1}}$ as an irreducible component of the degenerate
fiber of $\xi_{p_{0}}$ is even and so, it follows from Lemma \ref{lem:Standard-Comp-Chain}
that $E_{i_{1}}$ is not a branching component of $E$. This implies
that the closure $T$ in $V'$ of $B\setminus(E_{i_{1}}\cup E_{0})$
is contained in a unique degenerate fiber of $\overline{\pi}'$. After
making a sequence of birational transformations consisting on the
one hand of elementary transformations with $\mathbb{R}$-rational
centers (including infinitely near ones) on $E_{0}\setminus E_{i_{1}}$
and contracting the proper transform of $E_{0}$, and on the other
hand of contracting all successive $(-1)$-curves supported on the
proper transform of $T$ starting from that of $C_{0}$, the total
transform of $E_{0}\cup E_{i_{1}}\cup T$ can be re-written in the
form $B^{(1)}=F_{\infty}^{(1)}\cup C_{0}^{(1)}\cup E^{(1)}$, where
$F_{\infty}^{(1)}\simeq\mathbb{P}_{\mathbb{R}}^{1}$ is the last exceptional
divisor of the sequence of elementary transformations, $C_{0}^{(1)}$
is the proper transform of $E_{i_{1}}$ and $E^{(1)}$ is the image
of $T$. We let $V^{(1)}$ be the so constructed smooth projective
surface and we let $\eta:(V,B)\dashrightarrow(V^{(1)},B^{(1)})$ be
the corresponding birational map. By construction, $E^{(1)}$ is either
empty or has strictly less irreducible components than $E$. 

b) If $E_{0}^{2}=-(2s+1)$ is odd, we choose $s+1$ distinct smooth
fibers $\ell_{i}\simeq\mathbb{P}_{\mathbb{C}}^{1}$ of $\xi_{p_{0}}:V\rightarrow\mathbb{P}_{\mathbb{R}}^{1}$
over $\mathbb{C}$-rationals point of $\mathbb{P}_{\mathbb{R}}^{1}$
and we let $\eta':V\dashrightarrow V'$ be the birational map defined
over $\mathbb{R}$ consisting of the blow-up of the $\mathbb{C}$-rational
points $\ell_{i}\cap F_{\infty}$, with respective exceptional divisors
$\tilde{\ell}_{i}\simeq\mathbb{P}_{\mathbb{C}}^{1}$, followed by
the contraction of the proper transforms of the $\ell_{i}$, $i=1,\ldots,s+1$.
The proper transforms of $E_{0}$ and $F_{\infty}$ in $V''$ have
self-intersections $1$ and $-2s-2$ respectively, while the self-intersections
of the remaining irreducible components of $B$ are left unchanged. 

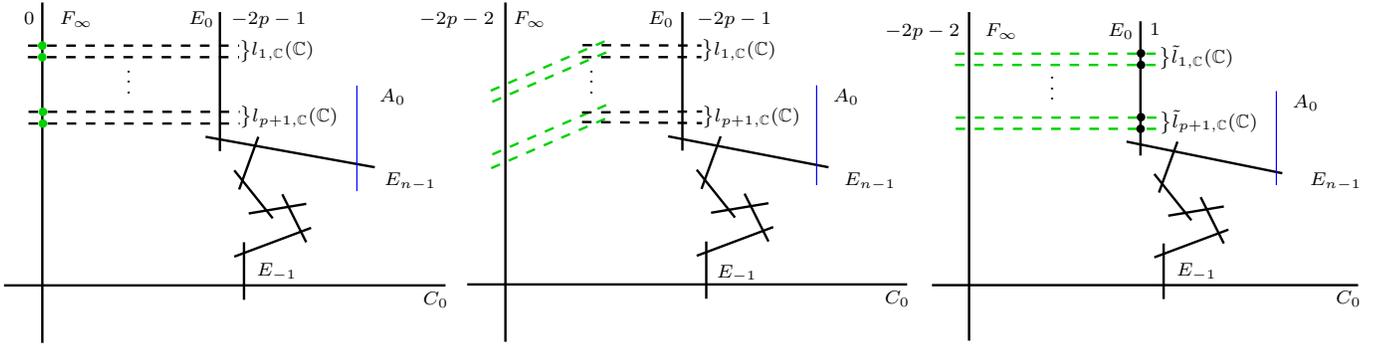
\begin{figure}[!htb]
\input{rectif-b-f.tex} 
\caption{Elementary link: odd case}  
\label{fig:Odd-rectif}
\end{figure}

Let $V''\rightarrow V'$ be the surface obtained by blowing-up the
intersection point of $E_{0}$ with the closure $T$ of the strict
transform of $B\setminus E_{0}$ and let $C_{0}^{(1)}$ be the exceptional
divisor. Then $T$ is contained in a unique degenerate fiber of the
$\mathbb{P}^{1}$-fibration $\overline{\pi}'':V''\rightarrow\mathbb{P}_{\mathbb{R}}^{1}$
generated by the complete linear system $|E_{0}|$ on $V''$ while
$C_{0}^{(1)}$ is a section of this fibration. After contracting all
successive $(-1)$-curves supported in $T$, starting from that of
$C_{0}$ and continuing with that of the successive proper transforms
of the irreducible components of the nonempty chain of $(-2)$-curves
joining $C_{0}$ to the branching component of $E$ contained in the
chain $L$ as in Lemma \ref{lem:Standard-Comp-Chain}, the total transform
of $E_{0}\cup C_{0}^{(1)}\cup T$ can be rewritten in the form $B^{(1)}=F_{\infty}^{(1)}\cup C_{0}^{(1)}\cup E^{(1)}$
where $F_{\infty}^{(1)}$ and $E^{(1)}$ are the proper transforms
of $E_{0}$ and $T$ respectively. We let $V^{(1)}$ be the so constructed
smooth projective surface and we let $\eta:(V,B)\dashrightarrow(V^{(1)},B^{(1)})$
be the corresponding birational map defined over $\mathbb{R}$. The
description given in Lemma \ref{lem:Standard-Comp-Chain} implies
again that if not empty, $E^{(1)}$ has strictly less irreducible
components than $E$. \end{defn}
\begin{prop}
\label{prop:r-standard-stability-links} Let $(V,B=F_{\infty}\cup C_{0}\cup E)$
be an $r$-standard pair such that $E$ is not empty and let $\eta:(V,B)\dashrightarrow(V^{(1)},B^{(1)})$
be an elementary link as in Definition \ref{def:Elementary-links}
above. Then $(V^{(1)},B^{(1)})$ is an $r$-standard pair and the
induced birational map $S=V\setminus B\dashrightarrow S^{(1)}=V^{(1)}\setminus B^{(1)}$
is $\mathbb{R}$-biregular. \end{prop}
\begin{proof}
By construction, the birational map $S\dashrightarrow S^{(1)}$ induces
a diffeomorphism $S(\mathbb{R})\stackrel{\simeq}{\rightarrow}S^{(1)}(\mathbb{R})$.
So it is enough to show that $S_{\mathbb{C}}^{(1)}(\mathbb{C})$ is
$\mathbb{Q}$-acyclic. Indeed, if so, then $S^{(1)}$ is affine whence
in particular does not contain any complete curve. As a consequence,
the $\mathbb{A}^{1}$-fibration $\pi^{(1)}:S^{(1)}\rightarrow\mathbb{A}_{\mathbb{R}}^{1}$
induced by the restriction of the $\mathbb{P}^{1}$-fibration $\overline{\pi}^{(1)}:V^{(1)}\rightarrow\mathbb{P}_{\mathbb{R}}^{1}$
defined by the complete linear system $|F_{\infty}^{(1)}|$ has at
most one degenerate fiber, whose closure is contained in the fiber
of $\overline{\pi}^{(1)}$ over the $\mathbb{R}$-rational point $\overline{\pi}^{(1)}(E^{(1)})\in\mathbb{A}_{\mathbb{R}}^{1}$.
Together with Proposition \ref{prop:NegKod-FiberStruct}, the $\mathbb{Q}$-acyclicity
of $S_{\mathbb{C}}^{(1)}(\mathbb{C})$ and the fact that $S^{(1)}(\mathbb{R})\approx\mathbb{R}^{2}$
imply that the unique possible degenerate fiber of $\pi^{(1)}$ is
isomorphic to $\mathbb{A}_{\mathbb{R}}^{1}$ when equipped with its
reduced structure and has odd multiplicity. So $\pi^{(1)}:S^{(1)}\rightarrow\mathbb{A}_{\mathbb{R}}^{1}$
is an $r$-standard $\mathbb{A}^{1}$-fibered surface. 

Note that since $\left(V,B\right)$ is an $r$-standard pair, $S_{\mathbb{C}}(\mathbb{C})$
is $\mathbb{Q}$-acyclic by virtue of Proposition \ref{prop:NegKod-FiberStruct}.
Furthermore, it follows from the proof of this proposition that $H_{1}(S_{\mathbb{C}}(\mathbb{C});\mathbb{Z})\simeq\mathbb{Z}/m\mathbb{Z}$,
where $m$ is the multiplicity of the unique degenerate fiber of the
$\mathbb{A}^{1}$-fibration $\pi:S\rightarrow\mathbb{A}_{\mathbb{R}}^{1}$.
By the description given in $\S$ \ref{par:Standard-Comp} and Figure
\ref{fig:rectif-0}, we may view the pair $(V,B)$ as being constructed
from the surface $\tau_{0}:V_{1}\rightarrow\mathbb{F}_{1}$ obtained
by blowing-up the point $p_{0}\in E_{-1}\setminus C_{0}$ by blowing-up
further a sequence of $\mathbb{R}$-rational points supported on the
successive total transforms on $E_{-1}$. Contracting $C_{0}$ from
$V_{1}$, we may also view $(V,B)$ as being obtained from another
Hirzebruch surface $\rho_{1}':\mathbb{F}_{1}\rightarrow\mathbb{P}_{\mathbb{R}}^{1}$
having $E_{0}$ and $F_{\infty}$ as sections with self-intersections
$-1$ and $+1$ respectively by a sequence $\alpha:V\rightarrow\mathbb{F}_{1}$
of blow-ups of $\mathbb{R}$-rational points in in such a way that
the $\mathbb{P}^{1}$-fibration $\xi_{p_{0}}:V\rightarrow\mathbb{P}_{\mathbb{R}}^{1}$
coincides with $\rho_{1}'\circ\tau'$. The image $D=\alpha{}_{*}(B)$
of $B$ consists of the union of $E_{0}$, $F_{\infty}$ and $E_{-1}$,
which is a fiber of $\tau'$. With the notation of \ref{par:curve_arrangement_setup},
the kernel $R$ of the surjective map $j_{\mathbb{C}}:\mathbb{Z}\langle D_{\mathbb{C}}\rangle\rightarrow\mathrm{Cl}(\mathbb{F}_{1})$
is generated by $F_{\infty}-E_{0}-E_{-1}$ while $\mathbb{Z}\langle\mathcal{E}_{0}\rangle$
is the free abelian group generated by $A_{0,\mathbb{C}}$. Letting
$f_{\infty}$, $e_{0}$ and $e_{-1}$ be the coefficients of $A_{0}$
in the total transforms in $V$ of $F_{\infty}$, $E_{0}$ and $E_{-1}$
respectively, the homomorphism $\varphi:R\rightarrow\mathbb{Z}\langle\mathcal{E}_{0}\rangle$
with respect to the chosen bases is simply the multiplication by $f_{\infty}-e_{0}-e_{-1}$,
and since the diagram chasing in the proof of Lemma \ref{lem:plane_arrangement_top_charac} (see also Remark~\ref{Rk:Arrangement-top-carac-gen})
provides an isomorphism $H_{1}(S_{\mathbb{C}}(\mathbb{C});\mathbb{Z})\simeq\mathbb{Z}\langle\mathcal{E}_{0}\rangle/\mathrm{Im}\varphi$,
we have $f_{\infty}-e_{0}-e_{-1}=\pm m$. 

On the other hand, with the notation of Definition \ref{def:Elementary-links},
$S^{(1)}=V^{(1)}\setminus B^{(1)}$ is isomorphic to the complement
in the projective surface $V'$ of the proper transform $B'$ of $B$.
Since by construction $V'$  is obtained from $V$ by performing $r=-E_{0}^{2}$
elementary birational transformations along $\mathbb{C}$-rational
smooth fibers of $\xi_{p_{0}}:V\rightarrow\mathbb{P}_{\mathbb{R}}^{1}$
with centers on $F_{\infty}$, we have a commutative diagram 
\begin{eqnarray*}
V & \stackrel{\eta'}{\dashrightarrow} & V'\\
\alpha\downarrow &  & \downarrow\alpha'\\
\mathbb{F}_{1} & \dashrightarrow & \mathbb{F}_{1-2r}
\end{eqnarray*}
where $\mathbb{F}_{1}\dashrightarrow\mathbb{F}_{1-2r}$ consists of
$r$ elementary birational transformations along $\mathbb{C}$-rational
smooth fibers of $\rho_{1}':\mathbb{F}_{1}\rightarrow\mathbb{P}_{\mathbb{R}}^{1}$
with centers on $F_{\infty}$, $\eta'$ restricts to an isomorphism
in a open neighborhood of $B\setminus(E_{0}\cup F_{\infty})$ and
$\alpha':V'\rightarrow\mathbb{F}_{1-2r}$ is a sequence of blow-ups
of $\mathbb{R}$-rational points. It follows in particular that the
coefficients of $A_{0}$ in the total transforms in $V'$ of $F_{\infty}$,
$E_{0}$ and $E_{-1}$ are again equal to $f_{\infty}$, $e_{0}$
and $e_{-1}$ respectively. On the other hand, the proper transforms
of $E_{-1}$, $E_{0}$ and $F_{\infty}$ in $\mathbb{F}_{1-2r}$ are
respectively a fiber of the $\mathbb{P}^{1}$-bundle structure and
a pair of sections with self-intersections $-1+2r$ and $1-2r$. Thus
$E_{0}$ is linearly equivalent in $\mathbb{F}_{1-2r}$ to $F_{\infty}+(2r-1)E_{-1}$
and, letting $D'=\alpha'_{*}B'$, a basis of the kernel $R'$ of the
surjective map $j'_{\mathbb{C}}:\mathbb{Z}\langle D'_{\mathbb{C}}\rangle\rightarrow\mathrm{Cl}(\mathbb{F}_{1})$
is generated by $F_{\infty}+(2r-1)E_{-1}-E_{0}$. As a consequence,
the homomorphism $\varphi':R'\rightarrow\mathbb{Z}\langle\mathcal{E}_{0}\rangle$
coincides with the multiplication by $f_{\infty}+(2r-1)e_{-1}-e_{0}$
and we deduce from the generalization of Lemma \ref{lem:plane_arrangement_top_charac}
given in Remark \ref{Rk:Arrangement-top-carac-gen}\textbf{ }that
$S_{\mathbb{C}}^{(1)}(\mathbb{C})$ is $\mathbb{Q}$-acyclic unless
$f_{\infty}+(2r-1)e_{-1}-e_{0}=0$. But this second possibility never
occurs because $f_{\infty}-e_{0}-e_{-1}=\pm m$ is odd by hypothesis.
This completes the proof. 
\end{proof}

\subsubsection{\label{sub:Proof-of-Theorem} Proof of Theorem \ref{thm:Bir-rectif}}

\indent\newline\indent By hypothesis $S$ is a smooth affine geometrically
integral surface defined over $\mathbb{R}$, with $S(\mathbb{R})\approx\mathbb{R}^{2}$
and equipped with an $\mathbb{A}^{1}$-fibration $\pi:S\rightarrow\mathbb{A}_{\mathbb{R}}^{1}$
such that all but at most one fibers of $\pi$ over $\mathbb{R}$-rational
points of $\mathbb{A}_{\mathbb{R}}^{1}$ contain a reduced $\mathbb{R}$-rational
irreducible component. By virtue of Proposition \ref{prop:R-standard-reduction},
there exists an $\mathbb{R}$-biregular birational map $\varphi_{0}:S\dashrightarrow S^{(0)}$
onto an $r$-standard surface $\pi^{(0)}:S^{(0)}\dashrightarrow\mathbb{A}_{\mathbb{R}}^{1}$
such that $\pi=\pi^{(0)}\circ\varphi_{0}$. Letting $(V^{(0)},B^{(0)}=F_{\infty}^{(0)}\cup C_{0}^{(0)}\cup E^{(0)})$
be a smooth completion of $S^{(0)}$ defined over $\mathbb{R}$ as
in $\S$ \ref{par:Standard-Comp}, we have the following alternative:
either $E^{(0)}$ is empty and then $V^{(0)}\simeq\mathbb{F}_{1}$
and $S^{(0)}=V^{(0)}\setminus B^{(0)}\simeq\mathbb{A}_{\mathbb{R}}^{2}$
or, by virtue of Proposition \ref{prop:r-standard-stability-links},
there exists an elementary link $\eta_{0}:(V^{(0)},B^{(0)})\dashrightarrow(V^{(1)},B^{(1)}=F_{\infty}^{(1)}\cup C_{0}^{(1)}\cup E^{(1)})$
to an $r$-standard pair $(V^{(1)},B^{(1)})$ restricting to an $\mathbb{R}$-biregular
birational map $\eta_{0}:S^{(0)}\dashrightarrow S^{(1)}=V^{(1)}\setminus B^{(1)}$
between $r$-standard surfaces. Since $E^{(1)}$ is either empty or
has strictly less irreducible component then $E^{(0)}$ we conclude
by induction that there exists a finite sequence of elementary links
\[
(V^{(0)},B^{(0)})\stackrel{\eta_{0}}{\dashrightarrow}(V^{(1)},B^{(1)})\stackrel{\eta_{1}}{\dashrightarrow}\cdots\stackrel{\eta_{m-1}}{\dashrightarrow}(V^{(m)},B^{(m)})\stackrel{\eta_{m}}{\dashrightarrow}(V^{(m+1)},B^{(m+1)})
\]
terminating with an $r$-standard pair $(V^{(m+1)},B^{(m+1)})$ for
which $E^{(m+1)}$ is empty and such that the composition $\eta_{m}\circ\cdots\circ\eta_{0}:S^{(0)}\dashrightarrow S^{(m+1)}=V^{(m+1)}\setminus B^{(m+1)}\simeq\mathbb{A}_{\mathbb{R}}^{2}$
is an $\mathbb{R}$-biregular birational map. This completes the proof
of Theorem \ref{thm:Bir-rectif}. 
\begin{example}
Let $D\subset\mathbb{P}_{\mathbb{R}}^{2}$ be a geometrically integral
$\mathbb{R}$-rational curve of degree $d=2n+1\geq1$ with a unique
singular point $p$ of multiplicity $2n$, such that $D_{\mathbb{C}}$
has a unique analytic branch at $p$, and let $S=\mathbb{P}_{\mathbb{R}}^{2}\setminus D$.
Then $D(\mathbb{R})$ is connected and since $d$ is odd, $S(\mathbb{R})$
is connected, hence homeomorphic to $\mathbb{R}^{2}$. On the other
hand, $S_{\mathbb{C}}(\mathbb{C})$ is $\mathbb{Q}$-acyclic with
$H_{1}(S_{\mathbb{C}}(\mathbb{C});\mathbb{Z})\simeq\mathrm{Cl}(S_{\mathbb{C}})\simeq\mathbb{Z}_{d}$
and so, $S$ is not isomorphic to $\mathbb{A}_{\mathbb{R}}^{2}$ as
a scheme over $\mathbb{R}$. But it follows from Theorem \ref{thm:Bir-rectif}
that $S$ is $\mathbb{R}$-biregularly birationally isomorphic to $\mathbb{A}_{\mathbb{R}}^{2}$.
Indeed, the pencil $\mathbb{P}_{\mathbb{R}}^{2}\dashrightarrow\mathbb{P}_{\mathbb{R}}^{2}$
generated by $D$ and $d$ times its tangent line $L\simeq\mathbb{P}_{\mathbb{R}}^{1}$
at $p$ restricts on $S$ to an $\mathbb{A}^{1}$-fibration $\pi:S\rightarrow\mathbb{A}_{\mathbb{R}}^{1}$
with a unique degenerate fiber isomorphic to $\mathbb{A}_{\mathbb{R}}^{1}=L\cap S$,
of multiplicity $d$. 
\end{example}

\section{Complements}

In the note \cite{DM16}, we exhibit examples of real fake planes of every Kodaira dimension $\kappa\in \{-\infty,0,1,2\}$ which are birationally diffeomorphic to $\mathbb{R}^{2}$. 

\subsection{\label{sub:Rectif-kod0}Exceptional $\mathbb{Q}$-homology euclidean
planes of Kodaira dimension $0$}

By virtue of \cite[(8.64)]{Fu82} (see also \cite[Lemma 4.4.2]{MiyBook}),
a smooth affine complex $\mathbb{Q}$-acyclic surface of Kodaira dimension
$0$ is either $\mathbb{A}_{*}^{1}$-ruled over a base curve isomorphic
to $\mathbb{A}_{\mathbb{C}}^{1}$ or $\mathbb{P}_{\mathbb{C}}^{1}$,
or is isomorphic to one of the so-called exceptional surfaces $Y(3,3,3)$,
$Y(2,4,4)$ or $Y(2,3,6)$. Hereafter, we construct all real models
of these exceptional surfaces and characterize $\mathbb{Q}$-acyclic
euclidean planes among them. We show in particular that $Y(2,4,4)$
admits two real forms of very different nature: one whose real locus
is not diffeomorphic to $\mathbb{R}^{2}$ and a second one which is
$\mathbb{R}$-biregularly birationally equivalent to $\mathbb{A}_{\mathbb{R}}^{2}$.
This illustrates the fact that neither the topology of the complex
model nor the Kodaira dimension are invariant under $\mathbb{R}$-biregular
birational equivalence.

\subsubsection{Real model of $Y(3,3,3)$ }

Let $D$ be the union of four general lines $\ell_{i}\simeq\mathbb{P}_{\mathbb{R}}^{1}$,
$i=0,1,2,3$ in $\mathbb{P}_{\mathbb{R}}^{2}$ and let $\tau\colon V\rightarrow\mathbb{P}_{\mathbb{R}}^{2}$
be the projective surface obtained by first blowing-up the points
$p_{ij}=\ell_{i}\cap\ell_{j}$ with exceptional divisors $E_{ij}$,
$i,j=1,2,3$, $i\neq j$ and then blowing-up the points $\ell_{1}\cap E_{12}$,
$\ell_{2}\cap E_{23}$ and $\ell_{3}\cap E_{1,3}$ with respective
exceptional divisors $E_{1}$, $E_{2}$ and $E_{3}$. We let $B=\ell_{0}\cup\ell_{1}\cup\ell_{2}\cup\ell_{3}\cup E_{12}\cup E_{23}\cup E_{13}$.
The dual graphs of $D$, its total transform $\tau^{-1}(D)$ in $V$
and $B$ are depicted in Figure \ref{fig:logkod0-Y333}. 

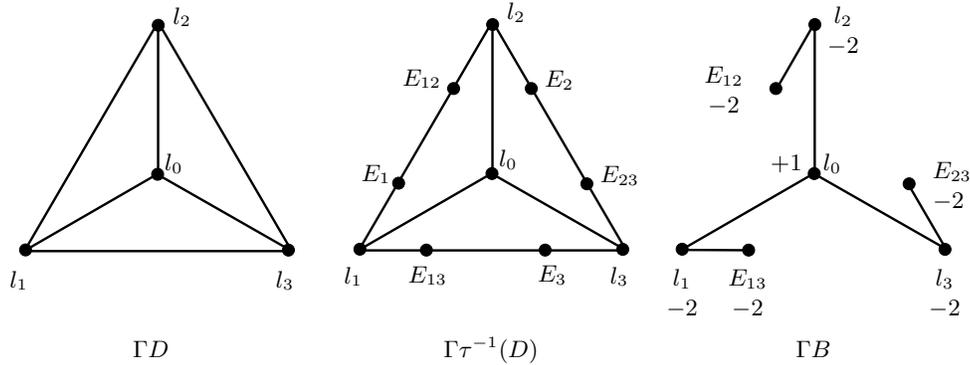
\begin{figure}[!htb]
\input{logkod0-Y333-f.tex} 
\caption{Construction of $Y(3,3,3)$}  
\label{fig:logkod0-Y333}
\end{figure}We let $Y(3,3,3)=V\setminus B$. With the notation of $\S$ \ref{par:curve_arrangement_setup},
the kernel $R$ of the surjective map $j_{\mathbb{C}}:\mathbb{Z}\langle D_{\mathbb{C}}\rangle\rightarrow\mathrm{Cl}(\mathbb{P}_{\mathbb{C}}^{2})$
is generated by the classes $\ell_{i,\mathbb{C}}-\ell_{0,\mathbb{C}}$,
$i=1,2,3$ while $\mathbb{Z}\langle\mathcal{E}_{0}\rangle$ is the
free abelian group generated by $E_{1,\mathbb{C}}$, $E_{2,\mathbb{C}}$
and $E_{3,\mathbb{C}}$. With this choice of basis, the induced homomorphism
$\varphi:R\rightarrow\mathbb{Z}\langle\mathcal{E}_{0}\rangle$ is
represented by the matrix 
\[
A=\left(\begin{array}{ccc}
2 & 1 & 0\\
0 & 2 & 1\\
1 & 0 & 2
\end{array}\right)
\]
which has determinant $\det A=9$. So by virtue of Lemma \ref{lem:plane_arrangement_top_charac},
$Y(3,3,3)_{\mathbb{C}}$ is $\mathbb{Q}$-acyclic with $H_{2}(Y(3,3,3)_{\mathbb{C}};\mathbb{Z})=0$
and $H_{1}(Y(3,3,3)_{\mathbb{C}};\mathbb{Z})\simeq\mathbb{Z}_{9}$.
Furthermore, since $j:\mathbb{Z}\langle D\rangle\rightarrow\mathrm{Cl}(\mathbb{P}_{\mathbb{R}}^{2})$
is surjective and $V$ is obtained from $\mathbb{P}_{\mathbb{R}}^{2}$
be blowing-up $\mathbb{R}$-rational points only, we deduce from c)
in the same Lemma that $Y(3,3,3)(\mathbb{R})\approx\mathbb{R}^{2}$. 
\begin{prop}
Let $S$ be a smooth surface defined over $\mathbb{R}$ such that
$S(\mathbb{R})\approx\mathbb{R}^{2}$ and $S_{\mathbb{C}}\simeq Y(3,3,3)_{\mathbb{C}}$.
Then $S$ is isomorphic to $Y(3,3,3)$ as a scheme over $\mathbb{R}$. \end{prop}
\begin{proof}
Since the automorphism group of $Y(3,3,3)_{\mathbb{C}}$ is isomorphic
to $\mathbb{Z}_{3}$, it follows that $Y(3,3,3)$ has no nontrivial
$\mathbb{R}$-form: indeed isomorphy classes of $\mathbb{R}$-forms
of $Y(3,3,3)$ are in one-to-one correspondence with elements of the
cohomology group $H^{1}(\mathbb{Z}_{2},\mathrm{Aut}(Y(3,3,3)_{\mathbb{C}}))\simeq H^{1}(\mathbb{Z}_{2},\mathbb{Z}_{3})=0$,
as every element in $\mathbb{Z}_{3}$ is a multiple of $2$. \end{proof}
\begin{question}
Is $Y(3,3,3)$ $\mathbb{R}$-biregularly birationally equivalent to
$\mathbb{A}_{\mathbb{R}}^{2}$ ? 
\end{question}

\subsubsection{Real forms of $Y(2,4,4)$ }

Starting from $\mathbb{P}_{\mathbb{R}}^{1}\times\mathbb{P}_{\mathbb{R}}^{1}$,
we can construct two non isomorphic real forms $Y_{r}(2,4,4)$ and
$Y_{c}(2,4,4)$ of the same complex surface as follows:

a) The surface $Y_{r}(2,4,4)$. We let $D_{r}$ be the union of three
distinct $\mathbb{R}$-rational fibers $\ell_{j}\simeq\mathbb{P}_{\mathbb{R}}^{1}$,
$j=1,2,3$, of the first projection and of three distinct $\mathbb{R}$-rational
fibers $M_{i}\simeq\mathbb{P}_{\mathbb{R}}^{1}$, $i=1,2,3$, of the
second projection. We let $\pi_{r}:V_{r}\rightarrow\mathbb{P}_{\mathbb{R}}^{1}\times\mathbb{P}_{\mathbb{R}}^{1}$
be the surface obtained by first blowing-up the $\mathbb{R}$-rational
points $p_{12}=M_{1}\cap\ell_{2}$, $p_{13}=M_{1}\cap\ell_{3}$, $p_{23}=M_{2}\cap\ell_{3}$
and $p_{32}=M_{3}\cap\ell_{2}$ with respective exceptional divisors
$E_{12},E_{13},E_{23}$ and $E_{32}$, and then blowing-up the $\mathbb{R}$-rational
points $M_{2}\cap E_{23}$ and $M_{3}\cap E_{32}$ with respective
exceptional divisors $F_{23}$ and $F_{32}$. We let $B_{r}=M_{1}\cup M_{2}\cup M_{3}\cup\ell_{1}\cup\ell_{2}\cup\ell_{3}\cup E_{23}\cup E_{32}$
and we let $Y_{r}(2,4,4)=V_{r}\setminus B_{r}$.

b) The surface $Y_{c}(2,4,4)$. We let $D_{c}$ be the union of a
$\mathbb{R}$-rational fibers $\ell_{1}\simeq\mathbb{P}_{\mathbb{R}}^{1}$
and $M_{1}\simeq\mathbb{P}_{\mathbb{R}}^{1}$ of the first and second
projection and of a pair of conjugate $\mathbb{C}$-rational fibers
$\ell\simeq\mathbb{P}_{\mathbb{C}}^{1}$ and $M\simeq\mathbb{P}_{\mathbb{C}}^{1}$
of the first and second projection respectively. We let $\pi_{c}:V_{c}\rightarrow\mathbb{P}_{\mathbb{R}}^{1}\times\mathbb{P}_{\mathbb{R}}^{1}$
be the surface defined over $\mathbb{R}$ obtained by blowing-up the
$\mathbb{C}$-rational points $M_{1}\cap\ell$ and $M\cap\ell$ with
respective exceptional divisors $E_{1}$ and $E$ and then blowing-up
the $\mathbb{C}$-rational point $M\cap E$ with exceptional divisor
$F$. We let $B_{c}=M_{1}\cup M\cup\ell_{1}\cup\ell\cup E$ and we
let $Y_{c}(2,4,4)=V_{c}\setminus B_{c}$. 

By construction, the surfaces $Y_{r}(2,4,4)$ and $Y_{c}(2,4,4)$
are not isomorphic over $\mathbb{R}$, but their complexifications
$Y_{r}(2,4,4)_{\mathbb{C}}$ and $Y_{c}(2,4,4)_{\mathbb{C}}$ are
isomorphic over $\mathbb{C}$. With the notation of \S \ref{par:curve_arrangement_setup},
the kernel $R$ of the homomorphism $j_{\mathbb{C}}:\mathbb{Z}\langle D_{r,\mathbb{C}}\rangle\rightarrow\mathrm{Cl}(\mathbb{P}_{\mathbb{C}}^{1}\times\mathbb{P}_{\mathbb{C}}^{1})$
is generated by the classes $\ell_{2,\mathbb{C}}-\ell_{1,\mathbb{C}}$,
$\ell_{3,\mathbb{C}}-\ell_{1,\mathbb{C}}$, $M_{2,\mathbb{C}}-M_{1,\mathbb{C}}$
and $M_{3,\mathbb{C}}-M_{1,\mathbb{C}}$ and letting $\mathbb{Z}\langle\mathcal{E}_{0}\rangle$be
the free abelian group generated by $E_{12,\mathbb{C}}$, $E_{13,\mathbb{C}}$,
$F_{23,\mathbb{C}}$ and $F_{32,\mathbb{C}}$, the induced homomorphism
$\varphi:R\rightarrow\mathbb{Z}\langle\mathcal{E}_{0}\rangle$ is
represented by the matrix 
\[
A=\left(\begin{array}{cccc}
1 & 0 & -1 & -1\\
0 & 1 & -1 & -1\\
0 & 1 & 2 & 0\\
1 & 0 & 0 & 2
\end{array}\right)
\]
which has determinant $\det A=8$. A similar argument as in the proof
of Lemma \ref{lem:plane_arrangement_top_charac}, see also Remark~\ref{Rk:Arrangement-top-carac-gen}, shows that $Y_{r}(2,4,4)_{\mathbb{C}}$
is $\mathbb{Q}$-acyclic with $H_{2}(Y_{r}(2,4,4)_{\mathbb{C}};\mathbb{Z})=0$
and $H_{1}(Y_{r}(2,4,4)_{\mathbb{C}};\mathbb{Z})\simeq\mathbb{Z}_{8}$. 
\begin{prop}
Let $S$ be a smooth surface defined over $\mathbb{R}$ such that
$S(\mathbb{R})\approx\mathbb{R}^{2}$ and $S_{\mathbb{C}}\simeq Y_{r}(2,4,4)_{\mathbb{C}}\simeq Y_{c}(2,4,4)_{\mathbb{C}}$.
Then $S$ is isomorphic to $Y_{c}(2,4,4)$ as a scheme over $\mathbb{R}$
and is $\mathbb{R}$-regularly birationally
 equivalent to $\mathbb{A}_{\mathbb{R}}^{2}$. \end{prop}
\begin{proof}
The automorphism group of $Y_{r}(2,4,4)_{\mathbb{C}}$ being isomorphic
to $\mathbb{Z}_{2}$, $Y_{r}(2,4,4)$ and $Y_{c}(2,4,4)$ are the
only two $\mathbb{R}$-forms of $Y_{r}(2,4,4)$. Since $j:\mathbb{Z}\langle D_{r}\rangle\rightarrow\mathrm{Cl}(\mathbb{P}_{\mathbb{R}}^{1}\times\mathbb{P}_{\mathbb{R}}^{1})$
is surjective and $\pi_{r}:V_{r}\rightarrow\mathbb{P}_{\mathbb{R}}^{1}\times\mathbb{P}_{\mathbb{R}}^{1}$
consists of blow-up of $\mathbb{R}$-rational points only, we infer
similarly as in the proof of c) in Lemma \ref{lem:plane_arrangement_top_charac}
that $Y_{r}(2,4,4)(\mathbb{R})\approx\mathbb{R}^{2}$ if and only if $\varphi\otimes\mathrm{id}:R\otimes_{\mathbb{Z}}\mathbb{Z}_{2}\rightarrow\mathbb{Z}\langle\mathcal{E}_{0}\rangle\otimes_{\mathbb{Z}}\mathbb{Z}_{2}$
is an isomorphism, which is not the case. Since the reduction of $A$
modulo $2$ is not invertible, we conclude that $Y_{r}(2,4,4)(\mathbb{R})\not\approx\mathbb{R}^{2}$.
On the other hand, since $\pi_{c}:V_{c}\rightarrow\mathbb{P}_{\mathbb{R}}^{1}\times\mathbb{P}_{\mathbb{R}}^{1}$
consists only of blow-ups of $\mathbb{C}$-rational points, the pair
$(V_{c},B_{c})$ is $\mathbb{R}$-regularly birationally equivalent
to $(\mathbb{P}_{\mathbb{R}}^{1}\times\mathbb{P}_{\mathbb{R}}^{1},M_{1}\cup\ell_{1})$
and so $Y_{c}(2,4,4)$ is $\mathbb{R}$-regularly birationally
 equivalent
to $\mathbb{P}_{\mathbb{R}}^{1}\times\mathbb{P}_{\mathbb{R}}^{1}\setminus(M_{1}\cup\ell_{1})\simeq\mathbb{A}_{\mathbb{R}}^{2}$. 
\end{proof}

\subsubsection{Real model of $Y(2,3,6)$}

Staring again, with two triples of $\mathbb{R}$-rational fibers $M_{i}$
and $\ell_{j}$, $i,j=1,2,3$ in $\mathbb{P}_{\mathbb{R}}^{1}\times\mathbb{P}_{\mathbb{R}}^{1}$
of the second and first projection respectively, we let $\pi:V\rightarrow\mathbb{P}_{\mathbb{R}}^{1}\times\mathbb{P}_{\mathbb{R}}^{1}$
be the smooth projective surface obtained by first blowing-up the
points $p_{13}=M_{1}\cap\ell_{3}$, $p_{31}=M_{3}\cap\ell_{1}$, $p_{22}=M_{2}\cap\ell_{2}$
and $p_{23}=M_{2}\cap\ell_{3}$ with respective exceptional divisors
$E_{13}$, $E_{31}$, $E_{22}$ and $E_{23}$ and then blowing-up
the points $p_{22}'=E_{22}\cap\ell_{2}$, $p_{13}'=M_{1}\cap E_{13}$
and $p_{31}'=E_{31}=M_{3}\cap E_{31}$ with respective exceptional
divisors $F_{22}$, $F_{13}$ and $F_{31}$. We let $B=M_{1}\cup M_{2}\cup M_{3}\cup\ell_{1}\cup\ell_{2}\cup\ell_{3}\cup E_{22}\cup E_{13}\cup E_{31}$
and we let $Y(2,3,6)=V\setminus B$. Then $Y(2,3,6)$ is a smooth
affine surface defined over $\mathbb{R}$ such that $H_{1}(Y(2,3,6)_{\mathbb{C}};\mathbb{Z})\simeq\mathbb{Z}_{6}$
and $H_{2}(Y(2,3,6)_{\mathbb{C}};\mathbb{Z})=0$. Since $V$ is obtained
from $\mathbb{P}_{\mathbb{R}}^{1}\times\mathbb{P}_{\mathbb{R}}^{1}$
by blowing-up $\mathbb{R}$-rational points only, we deduce in a similar
way as in the previous case that $Y(2,3,6)(\mathbb{R})\not\approx\mathbb{R}^{2}$.
Furthermore, since the automorphism group of $Y(2,3,6)_{\mathbb{C}}$
is trivial, there is no nontrivial $\mathbb{R}$-form of $Y(2,3,6)$.
Summing up, there is no smooth affine surface $S$ defined over $\mathbb{R}$
with $S(\mathbb{R})\approx\mathbb{R}^{2}$ and $S_{\mathbb{C}}\simeq Y(2,3,6)_{\mathbb{C}}$.

\subsection{\label{sub:Moduli-of--biregularly}Moduli of $\mathbb{R}$-biregularly birationally rectifiable surfaces of negative Kodaira dimension}

\indent\newline\noindent As seen in the introduction, in the rational
projective case, there is a unique minimal complexification or at
most one family of pairwise non isomorphic but $\mathbb{R}$-biregularly
birationally and deformation equivalent minimal complexifications.
Non minimal complexifications are obtained from these models by blowing-up
sequences of pairs of non-real conjugate points. It is natural to
expect an affine counterpart of this type of results in the form of
continuous moduli of $\mathbb{Q}$-acyclic euclidean planes of negative
Kodaira dimension all $\mathbb{R}$-biregularly birationally equivalent
to each others. For instance, starting with the standard open embedding
of $\mathbb{A}_{\mathbb{R}}^{2}$ in $\mathbb{P}_{\mathbb{R}}^{2}$
as the complement of a line $L_{\infty}\simeq\mathbb{P}_{\mathbb{R}}^{1}$
and performing a sequence of blow-ups $\tau:V\rightarrow\mathbb{P}_{\mathbb{R}}^{2}$
defined over $\mathbb{R}$ whose centers lie over $L_{\infty}$, one
obtains open embeddings $\mathbb{A}_{\mathbb{R}}^{2}\hookrightarrow V$
into various smooth projective surfaces defined over $\mathbb{R}$,
which, in restriction to the real loci correspond to smooth open embeddings
of $\mathbb{R}^{2}$ into smooth compact non-orientable surfaces of
arbitrary genus $g\geq1$. For a fixed number $g-1\geq0$ of $\mathbb{R}$-rational
points blown-up, the isomorphy type as real algebraic varieties of
the so-constructed surfaces $V$ with $g(V(\mathbb{R}))=g$ depend
on the choice of the points, giving rise in general to a continuous
moduli of such algebraic surfaces. In contrast, it follows from \cite{BiH07}
that their equivalence classes up to $\mathbb{R}$-biregular birational
isomorphisms depend only on $g$, which in this particular case coincides
simply with the number of $\mathbb{R}$-rational irreducible components
of the boundary $B=V\setminus\mathbb{A}_{\mathbb{R}}^{2}$. 

The next proposition illustrates the existence of infinitely many
deformation equivalence classes of pairwise non isomorphic $\mathbb{Q}$-acyclic
euclidean planes all $\mathbb{R}$-biregularly birationally equivalent
to $\mathbb{A}_{\mathbb{R}}^{2}$, each deformation equivalence class
having further a moduli of arbitrary positive dimension $n\geq3$.
\begin{prop}
Let $Y=\mathrm{Spec}(\mathbb{R}[a_{1},\ldots,a_{n}])$, $n\geq3$,
let $r\geq3$ be an odd integer, and let $\mathfrak{X}\subset Y\times\mathbb{A}_{\mathbb{R}}^{3}$
be the subvariety with equation $x^{n+1}z=y^{r}+\sum_{i=2}^{n}a_{i}x^{i+1}+x^{2}+x$.
Then the following hold:

1) The projection $\mathrm{pr}_{Y}:\mathfrak{X}\rightarrow Y$ is
smooth and $\mathrm{pr}_{Y}(\mathbb{R}):\mathfrak{X}(\mathbb{R})\rightarrow Y(\mathbb{R})$
is a trivial $\mathbb{R}^{2}$-bundle over $Y(\mathbb{R})\approx\mathbb{R}^{n}$. 

2) For every $\mathbb{R}$-rational point $p\in Y$, the scheme theoretic
fiber $S=\mathfrak{X}_{p}$ is a smooth connected affine surface defined
of $\mathbb{R}$, of negative Kodaira dimension with $S(\mathbb{R})\approx\mathbb{R}^{2}$
and $H_{1}(S_{\mathbb{C}}(\mathbb{C});\mathbb{Z})\simeq\mathbb{Z}_{r}$,
$H_{2}(S_{\mathbb{C}}(\mathbb{C});\mathbb{Z})=0$. The restriction
of $S$ of the projection $\mathrm{pr}_{x}$ is an $\mathbb{A}^{1}$-fibration
$q\colon S\rightarrow\mathbb{A}_{\mathbb{R}}^{1}$ with $q^{-1}(0)$ as
a unique geometrically irreducible degenerate fiber, of multiplicity
$r$. In particular, $S$ is $\mathbb{R}$-biregularly birationally equivalent to $\mathbb{A}_{\mathbb{R}}^{2}$. 

3) Every $S=\mathfrak{X}_{p}$ is deformation equivalent to $\mathfrak{X}_{0}$
via the retraction $Y\rightarrow\{0\}$, $(a_{2},\ldots,a_{n})\in\mathbb{R}^{n-1}\mapsto(ta_{2},\ldots,ta_{n})$,
$t\in\mathbb{R}$. 

4) Let $p=(a_{1},\ldots,a_{n})\in Y(\mathbb{R})$ and $p'=(a_{1}',\ldots,a_{n}')\in Y(\mathbb{R})$.
Then $\mathfrak{X}_{p}$ is isomorphic to $\mathfrak{X}_{p'}$ if
and only if $p=p'$.\end{prop}
\begin{proof}
The first assertion follows from the Jacobian criterion and the observation
that the map
\begin{eqnarray*}
\psi:Y(\mathbb{R})\times\mathbb{R}^{2}\rightarrow\mathfrak{X}(\mathbb{R}) &  & (a_{1},\ldots,a_{n},x,z)\mapsto(a_{1},\ldots,a_{n},x,\sqrt[r]{x^{n+1}z-\sum_{i=2}^{n}a_{i}x^{i+1}-x^{2}-x},z)
\end{eqnarray*}
is an homeomorphism. For every $p\in Y(\mathbb{R})$, $q:S=\mathfrak{X}_{p}\rightarrow\mathbb{A}_{\mathbb{R}}^{1}$
is an $r$-standard $\mathbb{A}^{1}$-fibered surface with $q^{-1}(\mathbb{A}_{\mathbb{R}}^{1}\setminus\{0\})\simeq\mathrm{Spec}(\mathbb{R}[x^{\pm1},y])\simeq(\mathbb{A}_{\mathbb{R}}^{1}\setminus\{0\})\times\mathbb{A}_{\mathbb{R}}^{1}$
and $q^{*}(\{0\})\simeq\mathrm{Spec}(\mathbb{R}[y]/(y^{r})[z])$.
So 2) follows from Proposition \ref{prop:NegKod-FiberStruct} and
Theorem \ref{thm:Bir-rectif}. The third assertion is clear. For the
last assertion, letting $S=\mathfrak{X}_{p}$ and $S'=\mathfrak{X}_{p'}$,
it follows from Theorem 6.1 in \cite{Po11} that $S_{\mathbb{C}}$
and $S'_{\mathbb{C}}$ are isomorphic if and only if there exists
$\lambda,\alpha,\mu\in\mathbb{C}^{*}$ and $\beta\in\mathbb{C}$ such
that 
\[
(\alpha y+\beta)^{r}+\sum_{i=2}^{n}a_{i}(\lambda x)^{i+1}+(\lambda x)^{2}+\lambda x=\mu(y^{r}+\sum_{i=2}^{n}a_{i}'x^{i+1}+x^{2}+x).
\]
The previous identity implies that $\beta=0$, $\mu=\alpha^{r}=\lambda^{2}=\lambda$
and $a_{i}\lambda^{i+1}=\mu a_{i}'$ for $i=2,\ldots,n$. Thus $\lambda=\mu=1$
necessarily and so $(a_{1},\ldots,a_{n})=(a_{1}',\ldots,a_{n}')$. 
\end{proof}

\bibliographystyle{amsplain} 

\end{document}

%% file: subdiv-f.tex
\begin{picture}(0,0)%
\includegraphics{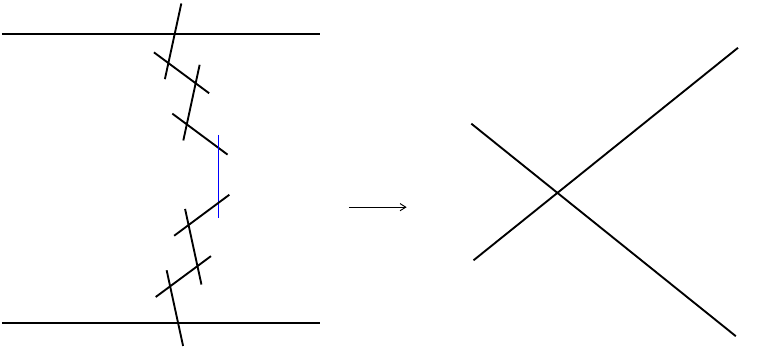}%
\end{picture}%
\setlength{\unitlength}{1973sp}%
\begingroup\makeatletter\ifx\SetFigFont\undefined%
\gdef\SetFigFont#1#2#3#4#5{%
  \reset@font\fontsize{#1}{#2pt}%
  \fontfamily{#3}\fontseries{#4}\fontshape{#5}%
  \selectfont}%
\fi\endgroup%
\begin{picture}(12246,5546)(-7174,-5909)
\put(-1124,-3511){\makebox(0,0)[b]{\smash{{\SetFigFont{7}{8.4}{\familydefault}{\mddefault}{\updefault}{\color[rgb]{0,0,0}$\tau$}%
}}}}
\put(-2031,-5807){\makebox(0,0)[b]{\smash{{\SetFigFont{7}{8.4}{\familydefault}{\mddefault}{\updefault}{\color[rgb]{0,0,0}$C_+$}%
}}}}
\put(-3152,-3136){\makebox(0,0)[b]{\smash{{\SetFigFont{7}{8.4}{\rmdefault}{\mddefault}{\updefault}{\color[rgb]{0,0,0}$A_0(p)$}%
}}}}
\put(-1973,-1159){\makebox(0,0)[b]{\smash{{\SetFigFont{7}{8.4}{\familydefault}{\mddefault}{\updefault}{\color[rgb]{0,0,0}$C_-$}%
}}}}
\put(2340,-3483){\makebox(0,0)[b]{\smash{{\SetFigFont{7}{8.4}{\familydefault}{\mddefault}{\updefault}{\color[rgb]{0,0,0}$p$}%
}}}}
\put(5057,-5795){\makebox(0,0)[b]{\smash{{\SetFigFont{7}{8.4}{\familydefault}{\mddefault}{\updefault}{\color[rgb]{0,0,0}$C_+$}%
}}}}
\put(5057,-1170){\makebox(0,0)[b]{\smash{{\SetFigFont{7}{8.4}{\familydefault}{\mddefault}{\updefault}{\color[rgb]{0,0,0}$C_-$}%
}}}}
\end{picture}%

%% file: logkod1-f.tex
\begin{picture}(0,0)%
\includegraphics{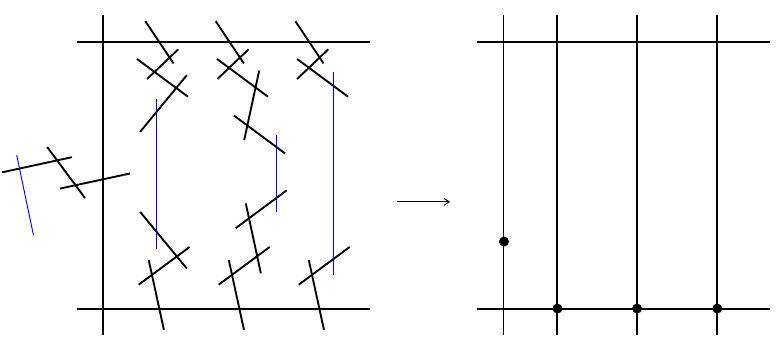}%
\end{picture}%
\setlength{\unitlength}{1973sp}%
\begingroup\makeatletter\ifx\SetFigFont\undefined%
\gdef\SetFigFont#1#2#3#4#5{%
  \reset@font\fontsize{#1}{#2pt}%
  \fontfamily{#3}\fontseries{#4}\fontshape{#5}%
  \selectfont}%
\fi\endgroup%
\begin{picture}(12499,5648)(-8324,-5838)
\put(-3619,-385){\makebox(0,0)[b]{\smash{{\SetFigFont{7}{8.4}{\familydefault}{\mddefault}{\updefault}{\color[rgb]{0,0,0}$E_{n,0}$}%
}}}}
\put(-3576,-2921){\makebox(0,0)[b]{\smash{{\SetFigFont{7}{8.4}{\familydefault}{\mddefault}{\updefault}{\color[rgb]{0,0,0}$\dots$}%
}}}}
\put(-215,-5748){\makebox(0,0)[b]{\smash{{\SetFigFont{7}{8.4}{\familydefault}{\mddefault}{\updefault}{\color[rgb]{0,0,0}$E_{0,0}$}%
}}}}
\put(585,-5748){\makebox(0,0)[b]{\smash{{\SetFigFont{7}{8.4}{\familydefault}{\mddefault}{\updefault}{\color[rgb]{0,0,0}$E_{1,0}$}%
}}}}
\put(1919,-5748){\makebox(0,0)[b]{\smash{{\SetFigFont{7}{8.4}{\familydefault}{\mddefault}{\updefault}{\color[rgb]{0,0,0}$\dots$}%
}}}}
\put(3199,-5748){\makebox(0,0)[b]{\smash{{\SetFigFont{7}{8.4}{\familydefault}{\mddefault}{\updefault}{\color[rgb]{0,0,0}$E_{n,0}$}%
}}}}
\put(4160,-5374){\makebox(0,0)[b]{\smash{{\SetFigFont{7}{8.4}{\familydefault}{\mddefault}{\updefault}{\color[rgb]{0,0,0}$C_1$}%
}}}}
\put(4160,-1108){\makebox(0,0)[b]{\smash{{\SetFigFont{7}{8.4}{\familydefault}{\mddefault}{\updefault}{\color[rgb]{0,0,0}$C_0$}%
}}}}
\put(-2399,-2836){\makebox(0,0)[b]{\smash{{\SetFigFont{7}{8.4}{\familydefault}{\mddefault}{\updefault}{\color[rgb]{0,0,0}$A_0(p_n)$}%
}}}}
\put(-5249,-2836){\makebox(0,0)[b]{\smash{{\SetFigFont{7}{8.4}{\familydefault}{\mddefault}{\updefault}{\color[rgb]{0,0,0}$A_0(p_1)$}%
}}}}
\put(-7274,-3736){\makebox(0,0)[b]{\smash{{\SetFigFont{7}{8.4}{\familydefault}{\mddefault}{\updefault}{\color[rgb]{0,0,0}$A_0(p_0)$}%
}}}}
\put( 61,-4119){\makebox(0,0)[b]{\smash{{\SetFigFont{7}{8.4}{\familydefault}{\mddefault}{\updefault}{\color[rgb]{0,0,0}$p_0$}%
}}}}
\put(1014,-4936){\makebox(0,0)[b]{\smash{{\SetFigFont{7}{8.4}{\familydefault}{\mddefault}{\updefault}{\color[rgb]{0,0,0}$p_{1,0}$}%
}}}}
\put(3563,-4944){\makebox(0,0)[b]{\smash{{\SetFigFont{7}{8.4}{\familydefault}{\mddefault}{\updefault}{\color[rgb]{0,0,0}$p_{n,0}$}%
}}}}
\put(-2290,-5397){\makebox(0,0)[b]{\smash{{\SetFigFont{7}{8.4}{\familydefault}{\mddefault}{\updefault}{\color[rgb]{0,0,0}$C_1$}%
}}}}
\put(-2372,-1108){\makebox(0,0)[b]{\smash{{\SetFigFont{7}{8.4}{\familydefault}{\mddefault}{\updefault}{\color[rgb]{0,0,0}$C_0$}%
}}}}
\put(-6605,-5756){\makebox(0,0)[b]{\smash{{\SetFigFont{7}{8.4}{\familydefault}{\mddefault}{\updefault}{\color[rgb]{0,0,0}$E_{0,0}$}%
}}}}
\put(-5903,-408){\makebox(0,0)[b]{\smash{{\SetFigFont{7}{8.4}{\familydefault}{\mddefault}{\updefault}{\color[rgb]{0,0,0}$E_{1,0}$}%
}}}}
\end{picture}%

%% file: logkod1homotop-f.tex
\begin{picture}(0,0)%
\includegraphics{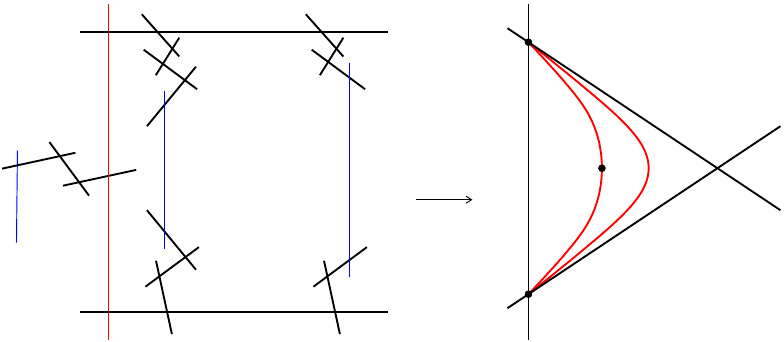}%
\end{picture}%
\setlength{\unitlength}{1973sp}%
\begingroup\makeatletter\ifx\SetFigFont\undefined%
\gdef\SetFigFont#1#2#3#4#5{%
  \reset@font\fontsize{#1}{#2pt}%
  \fontfamily{#3}\fontseries{#4}\fontshape{#5}%
  \selectfont}%
\fi\endgroup%
\begin{picture}(12520,5442)(-8342,-5826)
\put(561,-5177){\makebox(0,0)[b]{\smash{{\SetFigFont{7}{8.4}{\familydefault}{\mddefault}{\updefault}{\color[rgb]{0,0,0}$b_0$}%
}}}}
\put(505,-977){\makebox(0,0)[b]{\smash{{\SetFigFont{7}{8.4}{\familydefault}{\mddefault}{\updefault}{\color[rgb]{0,0,0}$b_1$}%
}}}}
\put(2465,-3945){\makebox(0,0)[b]{\smash{{\SetFigFont{7}{8.4}{\familydefault}{\mddefault}{\updefault}{\color[rgb]{0,0,0}$L_x$}%
}}}}
\put(-7906,-4539){\makebox(0,0)[b]{\smash{{\SetFigFont{7}{8.4}{\familydefault}{\mddefault}{\updefault}{\color[rgb]{0,0,0}$E_{0,r_0}$}%
}}}}
\put(-7223,-2433){\makebox(0,0)[b]{\smash{{\SetFigFont{7}{8.4}{\familydefault}{\mddefault}{\updefault}{\color[rgb]{0,0,0}$E_{0,0}$}%
}}}}
\put(-5151,-2993){\makebox(0,0)[b]{\smash{{\SetFigFont{7}{8.4}{\familydefault}{\mddefault}{\updefault}{\color[rgb]{0,0,0}$L_x$}%
}}}}
\put(-2295,-2993){\makebox(0,0)[b]{\smash{{\SetFigFont{7}{8.4}{\familydefault}{\mddefault}{\updefault}{\color[rgb]{0,0,0}$L_y$}%
}}}}
\put(-223,-2545){\makebox(0,0)[b]{\smash{{\SetFigFont{7}{8.4}{\familydefault}{\mddefault}{\updefault}{\color[rgb]{0,0,0}$L_y$}%
}}}}
\put(785,-2881){\makebox(0,0)[b]{\smash{{\SetFigFont{7}{8.4}{\familydefault}{\mddefault}{\updefault}{\color[rgb]{0,0,0}$E_{0,0}$}%
}}}}
\put(1569,-3273){\makebox(0,0)[b]{\smash{{\SetFigFont{7}{8.4}{\familydefault}{\mddefault}{\updefault}{\color[rgb]{0,0,0}$p_0$}%
}}}}
\put(2353,-2153){\makebox(0,0)[b]{\smash{{\SetFigFont{7}{8.4}{\familydefault}{\mddefault}{\updefault}{\color[rgb]{0,0,0}$L_z$}%
}}}}
\end{picture}%

%% file: cubic-f.tex
\begin{picture}(0,0)%
\includegraphics{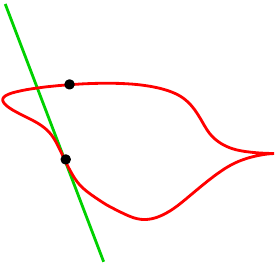}%
\end{picture}%
\setlength{\unitlength}{1973sp}%
\begingroup\makeatletter\ifx\SetFigFont\undefined%
\gdef\SetFigFont#1#2#3#4#5{%
  \reset@font\fontsize{#1}{#2pt}%
  \fontfamily{#3}\fontseries{#4}\fontshape{#5}%
  \selectfont}%
\fi\endgroup%
\begin{picture}(4431,4213)(15164,-9037)
\put(18960,-6763){\makebox(0,0)[b]{\smash{{\SetFigFont{7}{8.4}{\familydefault}{\mddefault}{\updefault}{\color[rgb]{0,0,0}$E_{0,0}(\mathbb{R})$}%
}}}}
\put(16640,-5931){\makebox(0,0)[b]{\smash{{\SetFigFont{7}{8.4}{\familydefault}{\mddefault}{\updefault}{\color[rgb]{0,0,0}$p_0$}%
}}}}
\put(15922,-5106){\makebox(0,0)[b]{\smash{{\SetFigFont{7}{8.4}{\familydefault}{\mddefault}{\updefault}{\color[rgb]{0,0,0}$L_z(\mathbb{R})$}%
}}}}
\end{picture}%

%% file: cubic-Klein-f.tex
\begin{picture}(0,0)%
\includegraphics{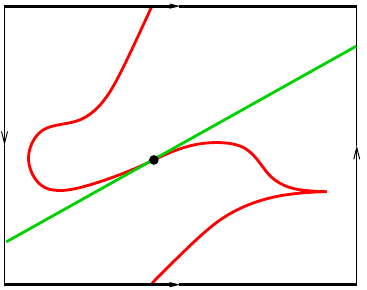}%
\end{picture}%
\setlength{\unitlength}{1973sp}%
\begingroup\makeatletter\ifx\SetFigFont\undefined%
\gdef\SetFigFont#1#2#3#4#5{%
  \reset@font\fontsize{#1}{#2pt}%
  \fontfamily{#3}\fontseries{#4}\fontshape{#5}%
  \selectfont}%
\fi\endgroup%
\begin{picture}(5863,4614)(13169,-7734)
\put(17217,-6905){\makebox(0,0)[b]{\smash{{\SetFigFont{7}{8.4}{\familydefault}{\mddefault}{\updefault}{\color[rgb]{0,0,0}$E_{0,0}(\mathbb{R})$}%
}}}}
\put(19011,-3349){\makebox(0,0)[lb]{\smash{{\SetFigFont{7}{8.4}{\familydefault}{\mddefault}{\updefault}{\color[rgb]{0,0,0}$E_{0,1}(\mathbb{R})$}%
}}}}
\put(19017,-7651){\makebox(0,0)[lb]{\smash{{\SetFigFont{7}{8.4}{\familydefault}{\mddefault}{\updefault}{\color[rgb]{0,0,0}$E_{0,1}(\mathbb{R})$}%
}}}}
\put(17433,-4315){\makebox(0,0)[b]{\smash{{\SetFigFont{7}{8.4}{\familydefault}{\mddefault}{\updefault}{\color[rgb]{0,0,0}$L_z(\mathbb{R})$}%
}}}}
\end{picture}%

%% file: Ramanujam-f.tex
\begin{picture}(0,0)%
\includegraphics{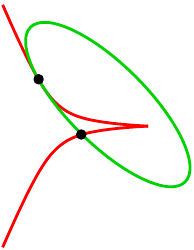}%
\end{picture}%
\setlength{\unitlength}{1973sp}%
\begingroup\makeatletter\ifx\SetFigFont\undefined%
\gdef\SetFigFont#1#2#3#4#5{%
  \reset@font\fontsize{#1}{#2pt}%
  \fontfamily{#3}\fontseries{#4}\fontshape{#5}%
  \selectfont}%
\fi\endgroup%
\begin{picture}(3072,3967)(3530,-8232)
\put(4193,-7951){\makebox(0,0)[b]{\smash{{\SetFigFont{7}{8.4}{\familydefault}{\mddefault}{\updefault}{\color[rgb]{0,0,0}$C(\mathbb{R})$}%
}}}}
\put(4824,-6728){\makebox(0,0)[b]{\smash{{\SetFigFont{7}{8.4}{\familydefault}{\mddefault}{\updefault}{\color[rgb]{0,0,0}$p$}%
}}}}
\put(5805,-5093){\makebox(0,0)[b]{\smash{{\SetFigFont{7}{8.4}{\familydefault}{\mddefault}{\updefault}{\color[rgb]{0,0,0}$Q(\mathbb{R})$}%
}}}}
\put(4035,-5791){\makebox(0,0)[b]{\smash{{\SetFigFont{7}{8.4}{\familydefault}{\mddefault}{\updefault}{\color[rgb]{0,0,0}$q$}%
}}}}
\end{picture}%

%% file: Ramanujam-Klein-f.tex
\begin{picture}(0,0)%
\includegraphics{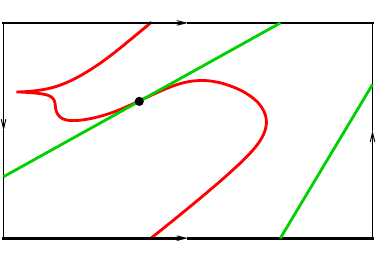}%
\end{picture}%
\setlength{\unitlength}{1973sp}%
\begingroup\makeatletter\ifx\SetFigFont\undefined%
\gdef\SetFigFont#1#2#3#4#5{%
  \reset@font\fontsize{#1}{#2pt}%
  \fontfamily{#3}\fontseries{#4}\fontshape{#5}%
  \selectfont}%
\fi\endgroup%
\begin{picture}(6020,4266)(13203,-8168)
\put(18858,-8081){\makebox(0,0)[b]{\smash{{\SetFigFont{7}{8.4}{\familydefault}{\mddefault}{\updefault}{\color[rgb]{0,0,0}$A_{0}(p)(\mathbb{R})$}%
}}}}
\put(18427,-7404){\makebox(0,0)[b]{\smash{{\SetFigFont{7}{8.4}{\familydefault}{\mddefault}{\updefault}{\color[rgb]{0,0,0}$Q(\mathbb{R})$}%
}}}}
\put(15660,-5683){\makebox(0,0)[b]{\smash{{\SetFigFont{7}{8.4}{\familydefault}{\mddefault}{\updefault}{\color[rgb]{0,0,0}$q$}%
}}}}
\put(17505,-4883){\makebox(0,0)[b]{\smash{{\SetFigFont{7}{8.4}{\familydefault}{\mddefault}{\updefault}{\color[rgb]{0,0,0}$Q(\mathbb{R})$}%
}}}}
\put(18982,-4109){\makebox(0,0)[b]{\smash{{\SetFigFont{7}{8.4}{\familydefault}{\mddefault}{\updefault}{\color[rgb]{0,0,0}$A_{0}(p)(\mathbb{R})$}%
}}}}
\put(17408,-6750){\makebox(0,0)[b]{\smash{{\SetFigFont{7}{8.4}{\familydefault}{\mddefault}{\updefault}{\color[rgb]{0,0,0}$C(\mathbb{R})$}%
}}}}
\end{picture}%

%% file: Ramanujam-log-f.tex
\begin{picture}(0,0)%
\includegraphics{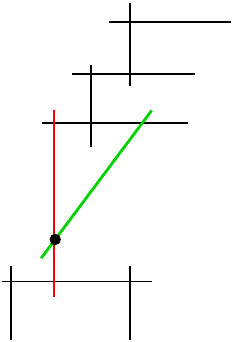}%
\end{picture}%
\setlength{\unitlength}{1973sp}%
\begingroup\makeatletter\ifx\SetFigFont\undefined%
\gdef\SetFigFont#1#2#3#4#5{%
  \reset@font\fontsize{#1}{#2pt}%
  \fontfamily{#3}\fontseries{#4}\fontshape{#5}%
  \selectfont}%
\fi\endgroup%
\begin{picture}(3732,5443)(14417,-9420)
\put(14851,-8986){\makebox(0,0)[b]{\smash{{\SetFigFont{7}{8.4}{\familydefault}{\mddefault}{\updefault}{\color[rgb]{0,0,0}$-3$}%
}}}}
\put(16747,-4743){\makebox(0,0)[b]{\smash{{\SetFigFont{7}{8.4}{\familydefault}{\mddefault}{\updefault}{\color[rgb]{0,0,0}$-2$}%
}}}}
\put(17701,-4561){\makebox(0,0)[b]{\smash{{\SetFigFont{7}{8.4}{\familydefault}{\mddefault}{\updefault}{\color[rgb]{0,0,0}$-2$}%
}}}}
\put(17176,-5386){\makebox(0,0)[b]{\smash{{\SetFigFont{7}{8.4}{\familydefault}{\mddefault}{\updefault}{\color[rgb]{0,0,0}$-2$}%
}}}}
\put(16126,-5611){\makebox(0,0)[b]{\smash{{\SetFigFont{7}{8.4}{\familydefault}{\mddefault}{\updefault}{\color[rgb]{0,0,0}$-2$}%
}}}}
\put(17176,-6136){\makebox(0,0)[b]{\smash{{\SetFigFont{7}{8.4}{\familydefault}{\mddefault}{\updefault}{\color[rgb]{0,0,0}$-1$}%
}}}}
\put(15451,-6736){\makebox(0,0)[lb]{\smash{{\SetFigFont{7}{8.4}{\familydefault}{\mddefault}{\updefault}{\color[rgb]{0,0,0}$C$}%
}}}}
\put(15751,-8761){\makebox(0,0)[b]{\smash{{\SetFigFont{7}{8.4}{\familydefault}{\mddefault}{\updefault}{\color[rgb]{0,0,0}$-1$}%
}}}}
\put(16801,-8986){\makebox(0,0)[b]{\smash{{\SetFigFont{7}{8.4}{\familydefault}{\mddefault}{\updefault}{\color[rgb]{0,0,0}$-2$}%
}}}}
\put(16351,-6811){\makebox(0,0)[lb]{\smash{{\SetFigFont{7}{8.4}{\familydefault}{\mddefault}{\updefault}{\color[rgb]{0,0,0}$Q$}%
}}}}
\put(15601,-7786){\makebox(0,0)[lb]{\smash{{\SetFigFont{7}{8.4}{\familydefault}{\mddefault}{\updefault}{\color[rgb]{0,0,0}$p$}%
}}}}
\end{picture}%

%% file: Ramanujam-subdiv-f.tex
\begin{picture}(0,0)%
\includegraphics{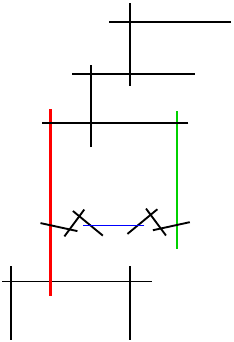}%
\end{picture}%
\setlength{\unitlength}{1973sp}%
\begingroup\makeatletter\ifx\SetFigFont\undefined%
\gdef\SetFigFont#1#2#3#4#5{%
  \reset@font\fontsize{#1}{#2pt}%
  \fontfamily{#3}\fontseries{#4}\fontshape{#5}%
  \selectfont}%
\fi\endgroup%
\begin{picture}(3732,5443)(14417,-9420)
\put(16426,-8011){\makebox(0,0)[b]{\smash{{\SetFigFont{7}{8.4}{\familydefault}{\mddefault}{\updefault}{\color[rgb]{0,0,0}$A_0(p)$}%
}}}}
\put(16747,-4743){\makebox(0,0)[b]{\smash{{\SetFigFont{7}{8.4}{\familydefault}{\mddefault}{\updefault}{\color[rgb]{0,0,0}$-2$}%
}}}}
\put(17701,-4561){\makebox(0,0)[b]{\smash{{\SetFigFont{7}{8.4}{\familydefault}{\mddefault}{\updefault}{\color[rgb]{0,0,0}$-2$}%
}}}}
\put(17176,-5386){\makebox(0,0)[b]{\smash{{\SetFigFont{7}{8.4}{\familydefault}{\mddefault}{\updefault}{\color[rgb]{0,0,0}$-2$}%
}}}}
\put(16126,-5611){\makebox(0,0)[b]{\smash{{\SetFigFont{7}{8.4}{\familydefault}{\mddefault}{\updefault}{\color[rgb]{0,0,0}$-2$}%
}}}}
\put(17176,-6136){\makebox(0,0)[b]{\smash{{\SetFigFont{7}{8.4}{\familydefault}{\mddefault}{\updefault}{\color[rgb]{0,0,0}$-1$}%
}}}}
\put(15751,-8761){\makebox(0,0)[b]{\smash{{\SetFigFont{7}{8.4}{\familydefault}{\mddefault}{\updefault}{\color[rgb]{0,0,0}$-1$}%
}}}}
\put(16801,-8986){\makebox(0,0)[b]{\smash{{\SetFigFont{7}{8.4}{\familydefault}{\mddefault}{\updefault}{\color[rgb]{0,0,0}$-2$}%
}}}}
\put(14851,-8986){\makebox(0,0)[b]{\smash{{\SetFigFont{7}{8.4}{\familydefault}{\mddefault}{\updefault}{\color[rgb]{0,0,0}$-3$}%
}}}}
\put(17401,-6736){\makebox(0,0)[lb]{\smash{{\SetFigFont{7}{8.4}{\familydefault}{\mddefault}{\updefault}{\color[rgb]{0,0,0}$Q$}%
}}}}
\put(17401,-7036){\makebox(0,0)[lb]{\smash{{\SetFigFont{7}{8.4}{\familydefault}{\mddefault}{\updefault}{\color[rgb]{0,0,0}$0$}%
}}}}
\put(15376,-6736){\makebox(0,0)[lb]{\smash{{\SetFigFont{7}{8.4}{\familydefault}{\mddefault}{\updefault}{\color[rgb]{0,0,0}$C$}%
}}}}
\put(15376,-7036){\makebox(0,0)[lb]{\smash{{\SetFigFont{7}{8.4}{\familydefault}{\mddefault}{\updefault}{\color[rgb]{0,0,0}$2$}%
}}}}
\end{picture}%

%% file: tricusp-quartics-f.tex
\begin{picture}(0,0)%
\includegraphics{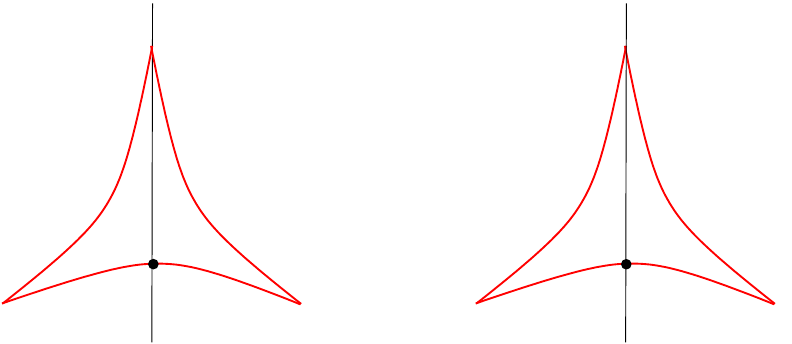}%
\end{picture}%
\setlength{\unitlength}{1973sp}%
\begingroup\makeatletter\ifx\SetFigFont\undefined%
\gdef\SetFigFont#1#2#3#4#5{%
  \reset@font\fontsize{#1}{#2pt}%
  \fontfamily{#3}\fontseries{#4}\fontshape{#5}%
  \selectfont}%
\fi\endgroup%
\begin{picture}(12659,5468)(-4234,-7001)
\put(-1519,-5570){\makebox(0,0)[b]{\smash{{\SetFigFont{7}{8.4}{\familydefault}{\mddefault}{\updefault}{\color[rgb]{0,0,0}$p$}%
}}}}
\put(-4183,-6555){\makebox(0,0)[b]{\smash{{\SetFigFont{7}{8.4}{\familydefault}{\mddefault}{\updefault}{\color[rgb]{0,0,0}$q_{\mathbb{C}}$}%
}}}}
\put(6017,-2236){\makebox(0,0)[b]{\smash{{\SetFigFont{7}{8.4}{\familydefault}{\mddefault}{\updefault}{\color[rgb]{0,0,0}$p_0$}%
}}}}
\put(5941,-6831){\makebox(0,0)[b]{\smash{{\SetFigFont{7}{8.4}{\familydefault}{\mddefault}{\updefault}{\color[rgb]{0,0,0}$L$}%
}}}}
\put(4570,-3955){\makebox(0,0)[b]{\smash{{\SetFigFont{7}{8.4}{\familydefault}{\mddefault}{\updefault}{\color[rgb]{0,0,0}$\Gamma_2$}%
}}}}
\put(6006,-5583){\makebox(0,0)[b]{\smash{{\SetFigFont{7}{8.4}{\familydefault}{\mddefault}{\updefault}{\color[rgb]{0,0,0}$p$}%
}}}}
\put(8410,-6554){\makebox(0,0)[b]{\smash{{\SetFigFont{7}{8.4}{\familydefault}{\mddefault}{\updefault}{\color[rgb]{0,0,0}$q_2$}%
}}}}
\put(3337,-6542){\makebox(0,0)[b]{\smash{{\SetFigFont{7}{8.4}{\familydefault}{\mddefault}{\updefault}{\color[rgb]{0,0,0}$q_1$}%
}}}}
\put(-1564,-2237){\makebox(0,0)[b]{\smash{{\SetFigFont{7}{8.4}{\familydefault}{\mddefault}{\updefault}{\color[rgb]{0,0,0}$p_0$}%
}}}}
\put(-1654,-6887){\makebox(0,0)[b]{\smash{{\SetFigFont{7}{8.4}{\familydefault}{\mddefault}{\updefault}{\color[rgb]{0,0,0}$L$}%
}}}}
\put(-3018,-3963){\makebox(0,0)[b]{\smash{{\SetFigFont{7}{8.4}{\familydefault}{\mddefault}{\updefault}{\color[rgb]{0,0,0}$\Gamma_1$}%
}}}}
\put(780,-6566){\makebox(0,0)[b]{\smash{{\SetFigFont{7}{8.4}{\familydefault}{\mddefault}{\updefault}{\color[rgb]{0,0,0}$q_{\mathbb{C}}$}%
}}}}
\end{picture}%

%% file: tricusp-resol-f.tex
\begin{picture}(0,0)%
\includegraphics{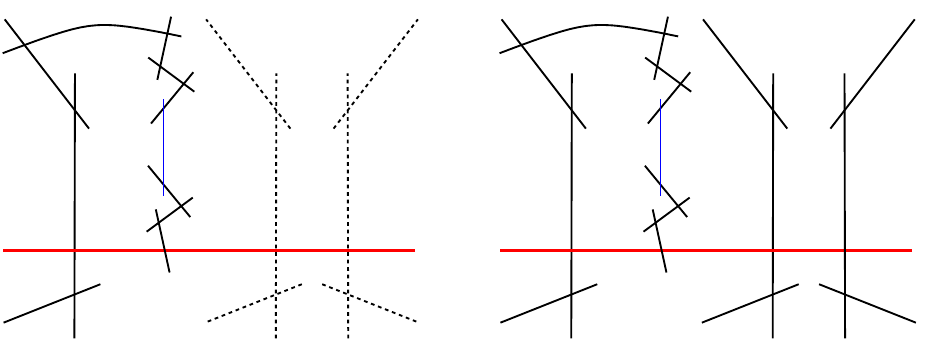}%
\end{picture}%
\setlength{\unitlength}{1973sp}%
\begingroup\makeatletter\ifx\SetFigFont\undefined%
\gdef\SetFigFont#1#2#3#4#5{%
  \reset@font\fontsize{#1}{#2pt}%
  \fontfamily{#3}\fontseries{#4}\fontshape{#5}%
  \selectfont}%
\fi\endgroup%
\begin{picture}(14909,5486)(15557,-6434)
\put(29476,-4036){\makebox(0,0)[b]{\smash{{\SetFigFont{7}{8.4}{\familydefault}{\mddefault}{\updefault}{\color[rgb]{0,0,0}$-1$}%
}}}}
\put(22051,-6361){\makebox(0,0)[b]{\smash{{\SetFigFont{7}{8.4}{\familydefault}{\mddefault}{\updefault}{\color[rgb]{0,0,0}$-2$}%
}}}}
\put(27001,-6361){\makebox(0,0)[b]{\smash{{\SetFigFont{7}{8.4}{\familydefault}{\mddefault}{\updefault}{\color[rgb]{0,0,0}$-2$}%
}}}}
\put(30076,-6361){\makebox(0,0)[b]{\smash{{\SetFigFont{7}{8.4}{\familydefault}{\mddefault}{\updefault}{\color[rgb]{0,0,0}$-2$}%
}}}}
\put(21526,-4036){\makebox(0,0)[b]{\smash{{\SetFigFont{7}{8.4}{\familydefault}{\mddefault}{\updefault}{\color[rgb]{0,0,0}$-1$}%
}}}}
\put(28276,-4036){\makebox(0,0)[b]{\smash{{\SetFigFont{7}{8.4}{\familydefault}{\mddefault}{\updefault}{\color[rgb]{0,0,0}$-1$}%
}}}}
\put(18601,-3586){\makebox(0,0)[b]{\smash{{\SetFigFont{7}{8.4}{\familydefault}{\mddefault}{\updefault}{\color[rgb]{0,0,0}$A_0(p)$}%
}}}}
\put(26551,-3586){\makebox(0,0)[b]{\smash{{\SetFigFont{7}{8.4}{\familydefault}{\mddefault}{\updefault}{\color[rgb]{0,0,0}$A_0(p)$}%
}}}}
\put(19351,-1561){\makebox(0,0)[b]{\smash{{\SetFigFont{7}{8.4}{\familydefault}{\mddefault}{\updefault}{\color[rgb]{0,0,0}$-3$}%
}}}}
\put(22426,-1561){\makebox(0,0)[b]{\smash{{\SetFigFont{7}{8.4}{\familydefault}{\mddefault}{\updefault}{\color[rgb]{0,0,0}$-3$}%
}}}}
\put(27376,-1561){\makebox(0,0)[b]{\smash{{\SetFigFont{7}{8.4}{\familydefault}{\mddefault}{\updefault}{\color[rgb]{0,0,0}$-3$}%
}}}}
\put(30451,-1561){\makebox(0,0)[b]{\smash{{\SetFigFont{7}{8.4}{\familydefault}{\mddefault}{\updefault}{\color[rgb]{0,0,0}$-3$}%
}}}}
\put(19126,-6361){\makebox(0,0)[b]{\smash{{\SetFigFont{7}{8.4}{\familydefault}{\mddefault}{\updefault}{\color[rgb]{0,0,0}$-2$}%
}}}}
\put(20251,-4036){\makebox(0,0)[b]{\smash{{\SetFigFont{7}{8.4}{\familydefault}{\mddefault}{\updefault}{\color[rgb]{0,0,0}$-1$}%
}}}}
\put(25068,-1119){\makebox(0,0)[lb]{\smash{{\SetFigFont{7}{8.4}{\familydefault}{\mddefault}{\updefault}{\color[rgb]{0,0,0}$L$}%
}}}}
\put(29551,-5236){\makebox(0,0)[lb]{\smash{{\SetFigFont{7}{8.4}{\familydefault}{\mddefault}{\updefault}{\color[rgb]{0,0,0}$\Gamma_2$}%
}}}}
\put(21601,-5236){\makebox(0,0)[lb]{\smash{{\SetFigFont{7}{8.4}{\familydefault}{\mddefault}{\updefault}{\color[rgb]{0,0,0}$\Gamma_1$}%
}}}}
\put(17118,-1119){\makebox(0,0)[lb]{\smash{{\SetFigFont{7}{8.4}{\familydefault}{\mddefault}{\updefault}{\color[rgb]{0,0,0}$L$}%
}}}}
\end{picture}%

%% file: tricusp-configu-f.tex
\begin{picture}(0,0)%
\includegraphics{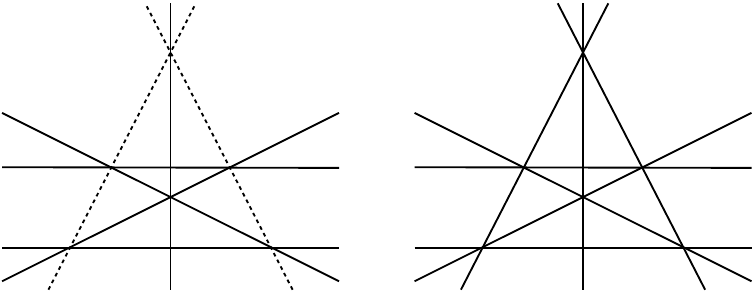}%
\end{picture}%
\setlength{\unitlength}{1973sp}%
\begingroup\makeatletter\ifx\SetFigFont\undefined%
\gdef\SetFigFont#1#2#3#4#5{%
  \reset@font\fontsize{#1}{#2pt}%
  \fontfamily{#3}\fontseries{#4}\fontshape{#5}%
  \selectfont}%
\fi\endgroup%
\begin{picture}(12058,4648)(-3609,-6376)
\put(7092,-5858){\makebox(0,0)[lb]{\smash{{\SetFigFont{6}{7.2}{\familydefault}{\mddefault}{\updefault}{\color[rgb]{0,0,0}$r$}%
}}}}
\put(4134,-5858){\makebox(0,0)[lb]{\smash{{\SetFigFont{6}{7.2}{\familydefault}{\mddefault}{\updefault}{\color[rgb]{0,0,0}$q$}%
}}}}
\put(5896,-2586){\makebox(0,0)[lb]{\smash{{\SetFigFont{6}{7.2}{\familydefault}{\mddefault}{\updefault}{\color[rgb]{0,0,0}$p$}%
}}}}
\put(-704,-2586){\makebox(0,0)[lb]{\smash{{\SetFigFont{6}{7.2}{\familydefault}{\mddefault}{\updefault}{\color[rgb]{0,0,0}$p$}%
}}}}
\put(492,-5858){\makebox(0,0)[lb]{\smash{{\SetFigFont{6}{7.2}{\familydefault}{\mddefault}{\updefault}{\color[rgb]{0,0,0}$r$}%
}}}}
\put(-2466,-5858){\makebox(0,0)[lb]{\smash{{\SetFigFont{6}{7.2}{\familydefault}{\mddefault}{\updefault}{\color[rgb]{0,0,0}$q$}%
}}}}
\end{picture}%

%% file: logkod-infty-f.tex
\begin{picture}(0,0)%
\includegraphics{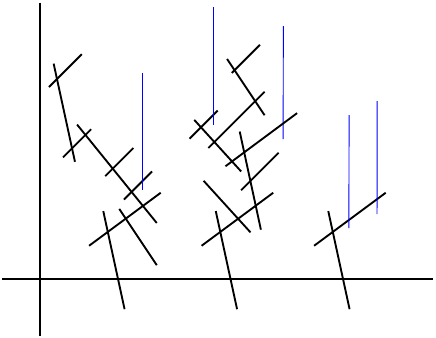}%
\end{picture}%
\setlength{\unitlength}{1973sp}%
\begingroup\makeatletter\ifx\SetFigFont\undefined%
\gdef\SetFigFont#1#2#3#4#5{%
  \reset@font\fontsize{#1}{#2pt}%
  \fontfamily{#3}\fontseries{#4}\fontshape{#5}%
  \selectfont}%
\fi\endgroup%
\begin{picture}(6966,5391)(-7232,-7969)
\put(-1199,-6811){\makebox(0,0)[b]{\smash{{\SetFigFont{7}{8.4}{\familydefault}{\mddefault}{\updefault}{\color[rgb]{0,0,0}$H_{p_s}$}%
}}}}
\put(-4349,-6811){\makebox(0,0)[b]{\smash{{\SetFigFont{7}{8.4}{\familydefault}{\mddefault}{\updefault}{\color[rgb]{0,0,0}$H_{p_1}$}%
}}}}
\put(-2999,-6811){\makebox(0,0)[b]{\smash{{\SetFigFont{9}{10.8}{\familydefault}{\mddefault}{\updefault}{\color[rgb]{0,0,0}$\dots$}%
}}}}
\put(-449,-7336){\makebox(0,0)[b]{\smash{{\SetFigFont{7}{8.4}{\familydefault}{\mddefault}{\updefault}{\color[rgb]{0,0,0}$C_0$}%
}}}}
\put(-6299,-3061){\makebox(0,0)[b]{\smash{{\SetFigFont{7}{8.4}{\familydefault}{\mddefault}{\updefault}{\color[rgb]{0,0,0}$F_c$}%
}}}}
\end{picture}%

%% file: logkod-infty-chain-f.tex
\begin{picture}(0,0)%
\includegraphics{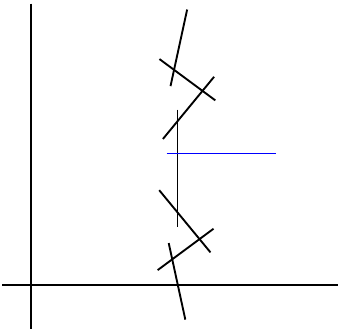}%
\end{picture}%
\setlength{\unitlength}{1973sp}%
\begingroup\makeatletter\ifx\SetFigFont\undefined%
\gdef\SetFigFont#1#2#3#4#5{%
  \reset@font\fontsize{#1}{#2pt}%
  \fontfamily{#3}\fontseries{#4}\fontshape{#5}%
  \selectfont}%
\fi\endgroup%
\begin{picture}(5448,5273)(-5973,-9260)
\put(-2549,-7036){\makebox(0,0)[b]{\smash{{\SetFigFont{7}{8.4}{\familydefault}{\mddefault}{\updefault}{\color[rgb]{0,0,0}$A_0(p)$}%
}}}}
\put(-3893,-6477){\makebox(0,0)[b]{\smash{{\SetFigFont{7}{8.4}{\rmdefault}{\mddefault}{\updefault}{\color[rgb]{0,0,0}$E$}%
}}}}
\put(-5024,-4261){\makebox(0,0)[b]{\smash{{\SetFigFont{7}{8.4}{\familydefault}{\mddefault}{\updefault}{\color[rgb]{0,0,0}$F_\infty$}%
}}}}
\put(-2699,-4411){\makebox(0,0)[b]{\smash{{\SetFigFont{7}{8.4}{\rmdefault}{\mddefault}{\updefault}{\color[rgb]{0,0,0}$E_0$}%
}}}}
\put(-1874,-6661){\makebox(0,0)[b]{\smash{{\SetFigFont{7}{8.4}{\rmdefault}{\mddefault}{\updefault}{\color[rgb]{0,0,0}$D$}%
}}}}
\put(-2699,-9061){\makebox(0,0)[b]{\smash{{\SetFigFont{7}{8.4}{\rmdefault}{\mddefault}{\updefault}{\color[rgb]{0,0,0}$F_0$}%
}}}}
\put(-674,-8836){\makebox(0,0)[b]{\smash{{\SetFigFont{7}{8.4}{\familydefault}{\mddefault}{\updefault}{\color[rgb]{0,0,0}$C_0$}%
}}}}
\end{picture}%

%% file: rectif-0-f.tex
\begin{picture}(0,0)%
\includegraphics{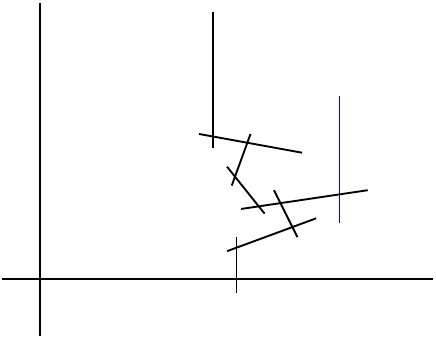}%
\end{picture}%
\setlength{\unitlength}{1973sp}%
\begingroup\makeatletter\ifx\SetFigFont\undefined%
\gdef\SetFigFont#1#2#3#4#5{%
  \reset@font\fontsize{#1}{#2pt}%
  \fontfamily{#3}\fontseries{#4}\fontshape{#5}%
  \selectfont}%
\fi\endgroup%
\begin{picture}(6966,5391)(-7232,-7969)
\put(-974,-5761){\makebox(0,0)[b]{\smash{{\SetFigFont{7}{8.4}{\rmdefault}{\mddefault}{\updefault}{\color[rgb]{0,0,0}$E_{n-1}$}%
}}}}
\put(-449,-7336){\makebox(0,0)[b]{\smash{{\SetFigFont{7}{8.4}{\familydefault}{\mddefault}{\updefault}{\color[rgb]{0,0,0}$C_0$}%
}}}}
\put(-3149,-6886){\makebox(0,0)[b]{\smash{{\SetFigFont{7}{8.4}{\rmdefault}{\mddefault}{\updefault}{\color[rgb]{0,0,0}$E_{-1}$}%
}}}}
\put(-6299,-2986){\makebox(0,0)[b]{\smash{{\SetFigFont{7}{8.4}{\familydefault}{\mddefault}{\updefault}{\color[rgb]{0,0,0}$F_\infty$}%
}}}}
\put(-4049,-2986){\makebox(0,0)[b]{\smash{{\SetFigFont{7}{8.4}{\rmdefault}{\mddefault}{\updefault}{\color[rgb]{0,0,0}$E_0$}%
}}}}
\put(-1499,-4711){\makebox(0,0)[b]{\smash{{\SetFigFont{7}{8.4}{\rmdefault}{\mddefault}{\updefault}{\color[rgb]{0,0,0}$A_0$}%
}}}}
\end{picture}%

%% file: rectif-a-f.tex
\begin{picture}(0,0)%
\includegraphics{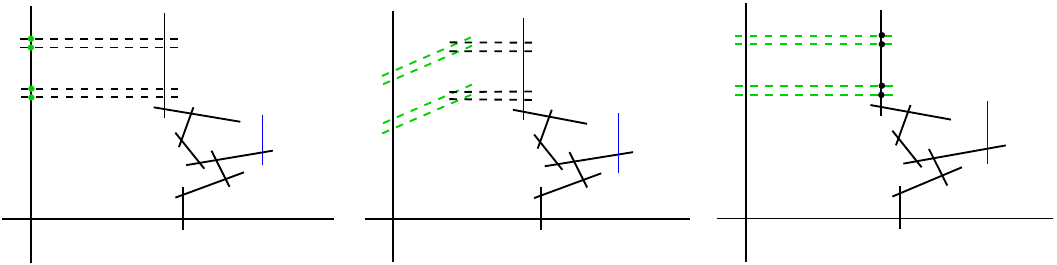}%
\end{picture}%
\setlength{\unitlength}{1973sp}%
\begingroup\makeatletter\ifx\SetFigFont\undefined%
\gdef\SetFigFont#1#2#3#4#5{%
  \reset@font\fontsize{#1}{#2pt}%
  \fontfamily{#3}\fontseries{#4}\fontshape{#5}%
  \selectfont}%
\fi\endgroup%
\begin{picture}(16886,4225)(-7265,-6352)
\put(1249,-2575){\makebox(0,0)[lb]{\smash{{\SetFigFont{7}{8.4}{\familydefault}{\mddefault}{\updefault}{\color[rgb]{0,0,0}$-2p$}%
}}}}
\put(-4537,-2506){\makebox(0,0)[lb]{\smash{{\SetFigFont{7}{8.4}{\familydefault}{\mddefault}{\updefault}{\color[rgb]{0,0,0}$-2p$}%
}}}}
\put(-6837,-2507){\makebox(0,0)[rb]{\smash{{\SetFigFont{6}{7.2}{\familydefault}{\mddefault}{\updefault}{\color[rgb]{0,0,0}$0$}%
}}}}
\put(-1037,-2568){\makebox(0,0)[rb]{\smash{{\SetFigFont{7}{8.4}{\familydefault}{\mddefault}{\updefault}{\color[rgb]{0,0,0}$-2p$}%
}}}}
\put(7493,-5496){\makebox(0,0)[b]{\smash{{\SetFigFont{7}{8.4}{\rmdefault}{\mddefault}{\updefault}{\color[rgb]{0,0,0}$E_{-1}$}%
}}}}
\put(-5730,-3315){\makebox(0,0)[b]{\smash{{\SetFigFont{7}{8.4}{\rmdefault}{\mddefault}{\updefault}{\color[rgb]{0,0,0}$\vdots$}%
}}}}
\put(-2842,-4182){\makebox(0,0)[b]{\smash{{\SetFigFont{7}{8.4}{\rmdefault}{\mddefault}{\updefault}{\color[rgb]{0,0,0}$A_0$}%
}}}}
\put(-2033,-5857){\makebox(0,0)[b]{\smash{{\SetFigFont{7}{8.4}{\familydefault}{\mddefault}{\updefault}{\color[rgb]{0,0,0}$C_0$}%
}}}}
\put(2863,-4213){\makebox(0,0)[b]{\smash{{\SetFigFont{7}{8.4}{\rmdefault}{\mddefault}{\updefault}{\color[rgb]{0,0,0}$A_0$}%
}}}}
\put(3654,-5852){\makebox(0,0)[b]{\smash{{\SetFigFont{7}{8.4}{\familydefault}{\mddefault}{\updefault}{\color[rgb]{0,0,0}$C_0$}%
}}}}
\put( 38,-3365){\makebox(0,0)[b]{\smash{{\SetFigFont{7}{8.4}{\rmdefault}{\mddefault}{\updefault}{\color[rgb]{0,0,0}$\vdots$}%
}}}}
\put(5723,-3270){\makebox(0,0)[b]{\smash{{\SetFigFont{7}{8.4}{\rmdefault}{\mddefault}{\updefault}{\color[rgb]{0,0,0}$\vdots$}%
}}}}
\put(8768,-3974){\makebox(0,0)[b]{\smash{{\SetFigFont{7}{8.4}{\rmdefault}{\mddefault}{\updefault}{\color[rgb]{0,0,0}$A_0$}%
}}}}
\put(9472,-5847){\makebox(0,0)[b]{\smash{{\SetFigFont{7}{8.4}{\familydefault}{\mddefault}{\updefault}{\color[rgb]{0,0,0}$C_0$}%
}}}}
\put(-6510,-2506){\makebox(0,0)[b]{\smash{{\SetFigFont{7}{8.4}{\familydefault}{\mddefault}{\updefault}{\color[rgb]{0,0,0}$F_\infty$}%
}}}}
\put(-4908,-2506){\makebox(0,0)[b]{\smash{{\SetFigFont{7}{8.4}{\rmdefault}{\mddefault}{\updefault}{\color[rgb]{0,0,0}$E_0$}%
}}}}
\put(-724,-2575){\makebox(0,0)[b]{\smash{{\SetFigFont{7}{8.4}{\familydefault}{\mddefault}{\updefault}{\color[rgb]{0,0,0}$F_\infty$}%
}}}}
\put(869,-2575){\makebox(0,0)[b]{\smash{{\SetFigFont{7}{8.4}{\rmdefault}{\mddefault}{\updefault}{\color[rgb]{0,0,0}$E_0$}%
}}}}
\put(4962,-2451){\makebox(0,0)[b]{\smash{{\SetFigFont{7}{8.4}{\familydefault}{\mddefault}{\updefault}{\color[rgb]{0,0,0}$F_\infty$}%
}}}}
\put(6588,-2451){\makebox(0,0)[b]{\smash{{\SetFigFont{7}{8.4}{\rmdefault}{\mddefault}{\updefault}{\color[rgb]{0,0,0}$E_0$}%
}}}}
\put(3318,-4649){\makebox(0,0)[b]{\smash{{\SetFigFont{7}{8.4}{\rmdefault}{\mddefault}{\updefault}{\color[rgb]{0,0,0}$E_{n-1}$}%
}}}}
\put(9281,-4559){\makebox(0,0)[b]{\smash{{\SetFigFont{7}{8.4}{\rmdefault}{\mddefault}{\updefault}{\color[rgb]{0,0,0}$E_{n-1}$}%
}}}}
\put(-2436,-4644){\makebox(0,0)[b]{\smash{{\SetFigFont{7}{8.4}{\rmdefault}{\mddefault}{\updefault}{\color[rgb]{0,0,0}$E_{n-1}$}%
}}}}
\put(-3982,-5511){\makebox(0,0)[b]{\smash{{\SetFigFont{7}{8.4}{\rmdefault}{\mddefault}{\updefault}{\color[rgb]{0,0,0}$E_{-1}$}%
}}}}
\put(1780,-5512){\makebox(0,0)[b]{\smash{{\SetFigFont{7}{8.4}{\rmdefault}{\mddefault}{\updefault}{\color[rgb]{0,0,0}$E_{-1}$}%
}}}}
\put(1345,-3693){\makebox(0,0)[lb]{\smash{{\SetFigFont{7}{8.4}{\familydefault}{\mddefault}{\updefault}{\color[rgb]{0,0,0}$\bigl\}l_{p,\mathbb{C}}(\mathbb{C})\bigr.$}%
}}}}
\put(7069,-2802){\makebox(0,0)[lb]{\smash{{\SetFigFont{7}{8.4}{\familydefault}{\mddefault}{\updefault}{\color[rgb]{0,0,0}$\bigl\}\tilde l_{1,\mathbb{C}}(\mathbb{C})\bigr.$}%
}}}}
\put(-4402,-2853){\makebox(0,0)[lb]{\smash{{\SetFigFont{7}{8.4}{\familydefault}{\mddefault}{\updefault}{\color[rgb]{0,0,0}$\bigl\}l_{1,\mathbb{C}}(\mathbb{C})\bigr.$}%
}}}}
\put(-4394,-3651){\makebox(0,0)[lb]{\smash{{\SetFigFont{7}{8.4}{\familydefault}{\mddefault}{\updefault}{\color[rgb]{0,0,0}$\bigl\}l_{p,\mathbb{C}}(\mathbb{C})\bigr.$}%
}}}}
\put(7076,-3611){\makebox(0,0)[lb]{\smash{{\SetFigFont{7}{8.4}{\familydefault}{\mddefault}{\updefault}{\color[rgb]{0,0,0}$\bigl\}\tilde l_{p,\mathbb{C}}(\mathbb{C})\bigr.$}%
}}}}
\put(1337,-2913){\makebox(0,0)[lb]{\smash{{\SetFigFont{7}{8.4}{\familydefault}{\mddefault}{\updefault}{\color[rgb]{0,0,0}$\bigl\}l_{1,\mathbb{C}}(\mathbb{C})\bigr.$}%
}}}}
\put(6954,-2452){\makebox(0,0)[lb]{\smash{{\SetFigFont{7}{8.4}{\familydefault}{\mddefault}{\updefault}{\color[rgb]{0,0,0}$0$}%
}}}}
\put(4573,-2451){\makebox(0,0)[rb]{\smash{{\SetFigFont{7}{8.4}{\familydefault}{\mddefault}{\updefault}{\color[rgb]{0,0,0}$-2p$}%
}}}}
\end{picture}%

%% file: rectif-b-f.tex
\begin{picture}(0,0)%
\includegraphics{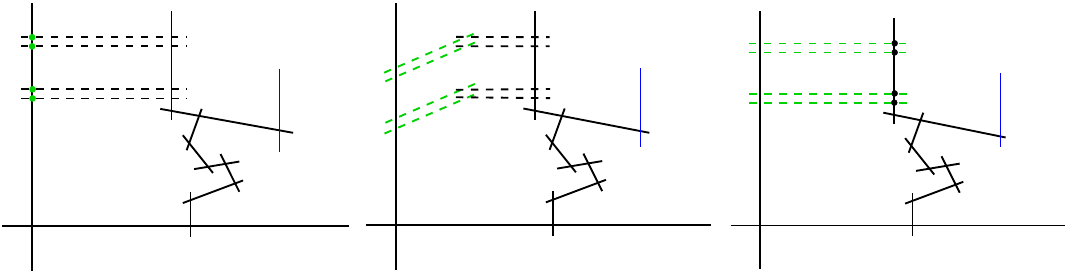}%
\end{picture}%
\setlength{\unitlength}{1973sp}%
\begingroup\makeatletter\ifx\SetFigFont\undefined%
\gdef\SetFigFont#1#2#3#4#5{%
  \reset@font\fontsize{#1}{#2pt}%
  \fontfamily{#3}\fontseries{#4}\fontshape{#5}%
  \selectfont}%
\fi\endgroup%
\begin{picture}(17075,4345)(-7239,-10777)
\put(-1049,-6736){\makebox(0,0)[rb]{\smash{{\SetFigFont{7}{8.4}{\familydefault}{\mddefault}{\updefault}{\color[rgb]{0,0,0}$-2p-2$}%
}}}}
\put(-6824,-6736){\makebox(0,0)[rb]{\smash{{\SetFigFont{7}{8.4}{\familydefault}{\mddefault}{\updefault}{\color[rgb]{0,0,0}$0$}%
}}}}
\put(-4349,-6736){\makebox(0,0)[lb]{\smash{{\SetFigFont{7}{8.4}{\familydefault}{\mddefault}{\updefault}{\color[rgb]{0,0,0}$-2p-1$}%
}}}}
\put(1501,-6736){\makebox(0,0)[lb]{\smash{{\SetFigFont{7}{8.4}{\familydefault}{\mddefault}{\updefault}{\color[rgb]{0,0,0}$-2p-1$}%
}}}}
\put(9151,-7786){\makebox(0,0)[b]{\smash{{\SetFigFont{7}{8.4}{\rmdefault}{\mddefault}{\updefault}{\color[rgb]{0,0,0}$A_0$}%
}}}}
\put(-5640,-7610){\makebox(0,0)[b]{\smash{{\SetFigFont{7}{8.4}{\rmdefault}{\mddefault}{\updefault}{\color[rgb]{0,0,0}$\vdots$}%
}}}}
\put(-1782,-10263){\makebox(0,0)[b]{\smash{{\SetFigFont{7}{8.4}{\familydefault}{\mddefault}{\updefault}{\color[rgb]{0,0,0}$C_0$}%
}}}}
\put(4018,-10249){\makebox(0,0)[b]{\smash{{\SetFigFont{7}{8.4}{\familydefault}{\mddefault}{\updefault}{\color[rgb]{0,0,0}$C_0$}%
}}}}
\put(9687,-10251){\makebox(0,0)[b]{\smash{{\SetFigFont{7}{8.4}{\familydefault}{\mddefault}{\updefault}{\color[rgb]{0,0,0}$C_0$}%
}}}}
\put(174,-7606){\makebox(0,0)[b]{\smash{{\SetFigFont{7}{8.4}{\rmdefault}{\mddefault}{\updefault}{\color[rgb]{0,0,0}$\vdots$}%
}}}}
\put(5963,-7690){\makebox(0,0)[b]{\smash{{\SetFigFont{7}{8.4}{\rmdefault}{\mddefault}{\updefault}{\color[rgb]{0,0,0}$\vdots$}%
}}}}
\put(-3777,-9901){\makebox(0,0)[b]{\smash{{\SetFigFont{7}{8.4}{\rmdefault}{\mddefault}{\updefault}{\color[rgb]{0,0,0}$E_{-1}$}%
}}}}
\put(2002,-9918){\makebox(0,0)[b]{\smash{{\SetFigFont{7}{8.4}{\rmdefault}{\mddefault}{\updefault}{\color[rgb]{0,0,0}$E_{-1}$}%
}}}}
\put(7781,-9902){\makebox(0,0)[b]{\smash{{\SetFigFont{7}{8.4}{\rmdefault}{\mddefault}{\updefault}{\color[rgb]{0,0,0}$E_{-1}$}%
}}}}
\put(6826,-6886){\makebox(0,0)[b]{\smash{{\SetFigFont{7}{8.4}{\rmdefault}{\mddefault}{\updefault}{\color[rgb]{0,0,0}$E_0$}%
}}}}
\put(5326,-6886){\makebox(0,0)[b]{\smash{{\SetFigFont{7}{8.4}{\familydefault}{\mddefault}{\updefault}{\color[rgb]{0,0,0}$F_\infty$}%
}}}}
\put(1051,-6736){\makebox(0,0)[b]{\smash{{\SetFigFont{7}{8.4}{\rmdefault}{\mddefault}{\updefault}{\color[rgb]{0,0,0}$E_0$}%
}}}}
\put(-599,-6736){\makebox(0,0)[b]{\smash{{\SetFigFont{7}{8.4}{\familydefault}{\mddefault}{\updefault}{\color[rgb]{0,0,0}$F_\infty$}%
}}}}
\put(-4724,-6736){\makebox(0,0)[b]{\smash{{\SetFigFont{7}{8.4}{\rmdefault}{\mddefault}{\updefault}{\color[rgb]{0,0,0}$E_0$}%
}}}}
\put(-6299,-6736){\makebox(0,0)[b]{\smash{{\SetFigFont{7}{8.4}{\familydefault}{\mddefault}{\updefault}{\color[rgb]{0,0,0}$F_\infty$}%
}}}}
\put(-2324,-7711){\makebox(0,0)[b]{\smash{{\SetFigFont{7}{8.4}{\rmdefault}{\mddefault}{\updefault}{\color[rgb]{0,0,0}$A_0$}%
}}}}
\put(-2099,-8761){\makebox(0,0)[b]{\smash{{\SetFigFont{7}{8.4}{\rmdefault}{\mddefault}{\updefault}{\color[rgb]{0,0,0}$E_{n-1}$}%
}}}}
\put(3676,-8761){\makebox(0,0)[b]{\smash{{\SetFigFont{7}{8.4}{\rmdefault}{\mddefault}{\updefault}{\color[rgb]{0,0,0}$E_{n-1}$}%
}}}}
\put(3376,-7711){\makebox(0,0)[b]{\smash{{\SetFigFont{7}{8.4}{\rmdefault}{\mddefault}{\updefault}{\color[rgb]{0,0,0}$A_0$}%
}}}}
\put(9526,-8761){\makebox(0,0)[b]{\smash{{\SetFigFont{7}{8.4}{\rmdefault}{\mddefault}{\updefault}{\color[rgb]{0,0,0}$E_{n-1}$}%
}}}}
\put(1563,-7956){\makebox(0,0)[lb]{\smash{{\SetFigFont{7}{8.4}{\familydefault}{\mddefault}{\updefault}{\color[rgb]{0,0,0}$\bigl\}l_{p+1,\mathbb{C}}(\mathbb{C})\bigr.$}%
}}}}
\put(1555,-7126){\makebox(0,0)[lb]{\smash{{\SetFigFont{7}{8.4}{\familydefault}{\mddefault}{\updefault}{\color[rgb]{0,0,0}$\bigl\}l_{1,\mathbb{C}}(\mathbb{C})\bigr.$}%
}}}}
\put(7309,-8029){\makebox(0,0)[lb]{\smash{{\SetFigFont{7}{8.4}{\familydefault}{\mddefault}{\updefault}{\color[rgb]{0,0,0}$\bigl\}\tilde l_{{p+1},\mathbb{C}}(\mathbb{C})\bigr.$}%
}}}}
\put(7301,-7225){\makebox(0,0)[lb]{\smash{{\SetFigFont{7}{8.4}{\familydefault}{\mddefault}{\updefault}{\color[rgb]{0,0,0}$\bigl\}\tilde l_{1,\mathbb{C}}(\mathbb{C})\bigr.$}%
}}}}
\put(-4253,-7127){\makebox(0,0)[lb]{\smash{{\SetFigFont{7}{8.4}{\familydefault}{\mddefault}{\updefault}{\color[rgb]{0,0,0}$\bigl\}l_{1,\mathbb{C}}(\mathbb{C})\bigr.$}%
}}}}
\put(-4246,-7960){\makebox(0,0)[lb]{\smash{{\SetFigFont{7}{8.4}{\familydefault}{\mddefault}{\updefault}{\color[rgb]{0,0,0}$\bigl\}l_{p+1,\mathbb{C}}(\mathbb{C})\bigr.$}%
}}}}
\put(4801,-6886){\makebox(0,0)[rb]{\smash{{\SetFigFont{7}{8.4}{\familydefault}{\mddefault}{\updefault}{\color[rgb]{0,0,0}$-2p-2$}%
}}}}
\put(7188,-6876){\makebox(0,0)[lb]{\smash{{\SetFigFont{7}{8.4}{\familydefault}{\mddefault}{\updefault}{\color[rgb]{0,0,0}$1$}%
}}}}
\end{picture}%

%% file: logkod0-Y333-f.tex
\begin{picture}(0,0)%
\includegraphics{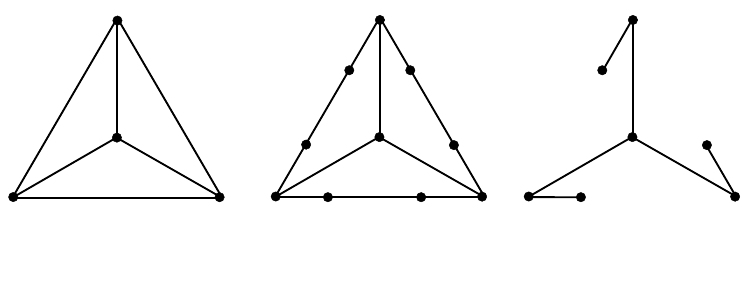}%
\end{picture}%
\setlength{\unitlength}{1973sp}%
\begingroup\makeatletter\ifx\SetFigFont\undefined%
\gdef\SetFigFont#1#2#3#4#5{%
  \reset@font\fontsize{#1}{#2pt}%
  \fontfamily{#3}\fontseries{#4}\fontshape{#5}%
  \selectfont}%
\fi\endgroup%
\begin{picture}(11843,4576)(-2384,-6682)
\put(7726,-6586){\makebox(0,0)[b]{\smash{{\SetFigFont{9}{10.8}{\familydefault}{\mddefault}{\updefault}{\color[rgb]{0,0,0}$\Gamma B$}%
}}}}
\put(-321,-2389){\makebox(0,0)[lb]{\smash{{\SetFigFont{9}{10.8}{\familydefault}{\mddefault}{\updefault}{\color[rgb]{0,0,0}$l_2$}%
}}}}
\put(-424,-4238){\makebox(0,0)[lb]{\smash{{\SetFigFont{9}{10.8}{\familydefault}{\mddefault}{\updefault}{\color[rgb]{0,0,0}$l_0$}%
}}}}
\put(973,-5706){\makebox(0,0)[lb]{\smash{{\SetFigFont{9}{10.8}{\familydefault}{\mddefault}{\updefault}{\color[rgb]{0,0,0}$l_3$}%
}}}}
\put(-2369,-5706){\makebox(0,0)[lb]{\smash{{\SetFigFont{9}{10.8}{\familydefault}{\mddefault}{\updefault}{\color[rgb]{0,0,0}$l_1$}%
}}}}
\put(3049,-3167){\makebox(0,0)[rb]{\smash{{\SetFigFont{9}{10.8}{\familydefault}{\mddefault}{\updefault}{\color[rgb]{0,0,0}$E_{12}$}%
}}}}
\put(5050,-4382){\makebox(0,0)[lb]{\smash{{\SetFigFont{9}{10.8}{\familydefault}{\mddefault}{\updefault}{\color[rgb]{0,0,0}$E_{23}$}%
}}}}
\put(2637,-5642){\makebox(0,0)[lb]{\smash{{\SetFigFont{9}{10.8}{\familydefault}{\mddefault}{\updefault}{\color[rgb]{0,0,0}$E_{13}$}%
}}}}
\put(2414,-4365){\makebox(0,0)[rb]{\smash{{\SetFigFont{9}{10.8}{\familydefault}{\mddefault}{\updefault}{\color[rgb]{0,0,0}$E_1$}%
}}}}
\put(4352,-3184){\makebox(0,0)[lb]{\smash{{\SetFigFont{9}{10.8}{\familydefault}{\mddefault}{\updefault}{\color[rgb]{0,0,0}$E_2$}%
}}}}
\put(4447,-5634){\makebox(0,0)[b]{\smash{{\SetFigFont{9}{10.8}{\familydefault}{\mddefault}{\updefault}{\color[rgb]{0,0,0}$E_3$}%
}}}}
\put(5173,-5654){\makebox(0,0)[lb]{\smash{{\SetFigFont{9}{10.8}{\familydefault}{\mddefault}{\updefault}{\color[rgb]{0,0,0}$l_3$}%
}}}}
\put(3879,-2337){\makebox(0,0)[lb]{\smash{{\SetFigFont{9}{10.8}{\familydefault}{\mddefault}{\updefault}{\color[rgb]{0,0,0}$l_2$}%
}}}}
\put(3776,-4186){\makebox(0,0)[lb]{\smash{{\SetFigFont{9}{10.8}{\familydefault}{\mddefault}{\updefault}{\color[rgb]{0,0,0}$l_0$}%
}}}}
\put(1831,-5654){\makebox(0,0)[lb]{\smash{{\SetFigFont{9}{10.8}{\familydefault}{\mddefault}{\updefault}{\color[rgb]{0,0,0}$l_1$}%
}}}}
\put(7585,-4234){\makebox(0,0)[rb]{\smash{{\SetFigFont{9}{10.8}{\familydefault}{\mddefault}{\updefault}{\color[rgb]{0,0,0}$+1$}%
}}}}
\put(9227,-4703){\makebox(0,0)[lb]{\smash{{\SetFigFont{9}{10.8}{\familydefault}{\mddefault}{\updefault}{\color[rgb]{0,0,0}$-2$}%
}}}}
\put(6901,-6061){\makebox(0,0)[b]{\smash{{\SetFigFont{9}{10.8}{\familydefault}{\mddefault}{\updefault}{\color[rgb]{0,0,0}$-2$}%
}}}}
\put(7849,-4234){\makebox(0,0)[lb]{\smash{{\SetFigFont{9}{10.8}{\familydefault}{\mddefault}{\updefault}{\color[rgb]{0,0,0}$l_0$}%
}}}}
\put(6901,-5686){\makebox(0,0)[b]{\smash{{\SetFigFont{9}{10.8}{\familydefault}{\mddefault}{\updefault}{\color[rgb]{0,0,0}$E_{13}$}%
}}}}
\put(9376,-5686){\makebox(0,0)[b]{\smash{{\SetFigFont{9}{10.8}{\familydefault}{\mddefault}{\updefault}{\color[rgb]{0,0,0}$l_3$}%
}}}}
\put(9376,-6061){\makebox(0,0)[b]{\smash{{\SetFigFont{9}{10.8}{\familydefault}{\mddefault}{\updefault}{\color[rgb]{0,0,0}$-2$}%
}}}}
\put(6601,-3511){\makebox(0,0)[b]{\smash{{\SetFigFont{9}{10.8}{\familydefault}{\mddefault}{\updefault}{\color[rgb]{0,0,0}$-2$}%
}}}}
\put(8101,-2761){\makebox(0,0)[b]{\smash{{\SetFigFont{9}{10.8}{\familydefault}{\mddefault}{\updefault}{\color[rgb]{0,0,0}$-2$}%
}}}}
\put(9226,-4336){\makebox(0,0)[lb]{\smash{{\SetFigFont{9}{10.8}{\familydefault}{\mddefault}{\updefault}{\color[rgb]{0,0,0}$E_{23}$}%
}}}}
\put(6076,-6061){\makebox(0,0)[b]{\smash{{\SetFigFont{9}{10.8}{\familydefault}{\mddefault}{\updefault}{\color[rgb]{0,0,0}$-2$}%
}}}}
\put(6076,-5686){\makebox(0,0)[b]{\smash{{\SetFigFont{9}{10.8}{\familydefault}{\mddefault}{\updefault}{\color[rgb]{0,0,0}$l_1$}%
}}}}
\put(6601,-3136){\makebox(0,0)[b]{\smash{{\SetFigFont{9}{10.8}{\familydefault}{\mddefault}{\updefault}{\color[rgb]{0,0,0}$E_{12}$}%
}}}}
\put(8101,-2386){\makebox(0,0)[b]{\smash{{\SetFigFont{9}{10.8}{\familydefault}{\mddefault}{\updefault}{\color[rgb]{0,0,0}$l_2$}%
}}}}
\put(-599,-6586){\makebox(0,0)[b]{\smash{{\SetFigFont{9}{10.8}{\familydefault}{\mddefault}{\updefault}{\color[rgb]{0,0,0}$\Gamma D$}%
}}}}
\put(3676,-6586){\makebox(0,0)[b]{\smash{{\SetFigFont{9}{10.8}{\familydefault}{\mddefault}{\updefault}{\color[rgb]{0,0,0}$\Gamma \tau^{-1}(D)$}%
}}}}
\end{picture}%